\documentclass[10pt]{article}

\usepackage{amsmath}
\usepackage{amssymb}
\usepackage{amsthm}

\newcommand{\ga}{\gamma}
\newcommand{\e}{\varepsilon}
\newcommand{\la}{\lambda}

\newtheorem{theorem}{Theorem}[section]
\newtheorem{lemma}[theorem]{Lemma}
\newtheorem{remark}[theorem]{Remark}
\newtheorem{proposition}[theorem]{Proposition}

\newtheorem{corollary}[theorem]{Corollary}
\newtheorem{conjecture}[theorem]{Conjecture}
\numberwithin{equation}{section}
\usepackage{color}
\begin{document}
\title{
Concentration profile, energy, and weak limits of radial solutions to semilinear elliptic equations with Trudinger-Moser critical nonlinearities
}
\author{Daisuke Naimen\footnote{naimen@mmm.muroran-it.ac.jp}}
\date{\small{\textit{Muroran Institute of Technology, 27-1, Mizumoto-cho, Muroran-shi, Hokkaido, 0508585, JAPAN}}}
\maketitle

\begin{abstract}
We investigate the next Trudinger-Moser critical equations, 
\[
\begin{cases}
-\Delta u=\lambda ue^{u^2+\alpha|u|^\beta}&\text{ in }B,\\
u=0&\text{ on }\partial B,
\end{cases}
\]
where $\alpha>0$, $(\la,\beta)\in(0,\infty)\times(0,2)$ and $B\subset \mathbb{R}^2$ is the unit ball centered at the origin. We classify the asymptotic behavior of energy bounded sequences of radial solutions. Via the blow--up analysis and a scaling technique, we deduce the limit profile, energy, and several asymptotic formulas of concentrating solutions together with precise information of the weak limit.  In particular, we obtain a new necessary condition on the amplitude of the weak limit at the concentration point. This gives a proof  of the conjecture by Grossi-Mancini-Naimen-Pistoia \cite{GMNP} in 2020 in the radial case. Moreover, in the case of $\beta\le1$, we show that any sequence carries at most one bubble.  This allows a new proof of the nonexistence of low energy  nodal radial solutions for $(\la,\beta)$ in a suitable range. Lastly, we discuss several counterparts of our classification result. Especially, we  prove the existence of a sequence of solutions which carries multiple bubbles and weakly converges to a sign-changing solution. 

\end{abstract}

\section{Introduction}\label{intro}
We study the following Trudinger-Moser critical equations,
\begin{equation}\label{p}
\begin{cases}
-\Delta u=\lambda ue^{u^2+\alpha|u|^\beta}&\text{ in }B,\\
u=0&\text{ on }\partial B,
\end{cases}
\end{equation}
where $B\subset \mathbb{R}^2$ is the unit ball centered at the origin, $\alpha>0$ and $(\lambda,\beta)\in(0,\infty)\times(0,2)$. Our aim is to investigate the asymptotic behavior of any energy bounded radial solutions. Especially, we are interested in the concentrating behavior of them. 

The study of \eqref{p} is motivated by the Trudinger-Moser inequality (\cite{P}, \cite{T} and \cite{M}). It gives the critical embedding of the Sobolev space in dimension two. Indeed, Moser \cite{M} proves that for any bounded domain $\Omega\subset \mathbb{R}^2$, it holds that
\begin{equation}\label{TM}
\sup_{u\in H^1_0(\Omega),\ \int_{\Omega}|\nabla u|^2dx\le1}\int_{\Omega} e^{a u^2}dx
\begin{cases}<\infty\text{ if }a\le 4\pi,\\
=\infty\text{ otherwise}.
\end{cases}
\end{equation}
A surprising fact is that the maximizer exists even in the critical case $\alpha=4\pi$. This is first proved by Carleson-Chang  when $\Omega=B$ in their celebrated paper \cite{CC}. The generalization of the domain  and  the dimension are given by \cite{F} and \cite{L} respectively. More recently, the sharp form of the inequality \eqref{TM} is discussed in \cite{Th}, \cite{IMNS}, and \cite{H}.

\eqref{TM} suggests that critical nonlinearities of semilinear elliptic problems in dimension two have the exponential growth. This  leads us to investigate the problem,
\begin{equation}\label{p0}
\begin{cases}
-\Delta u=\lambda h(u)e^{u^2}&\text{ in }\Omega,\\
u=0&\text{ on }\partial \Omega,
\end{cases}
\end{equation}
where $\la>0$, $\Omega\subset \mathbb{R}^2$ is a smooth bounded domain and $h:\mathbb{R}\to \mathbb{R}$ is a continuous function which has the subcritical growth at infinity,   $\lim_{|t|\to \infty}h(t)/e^{at^2}=0$ for any $a>0$, and satisfies $h(0)=0$ and some appropriate conditions.  Due to the lack of the compactness of the critical case of \eqref{TM}, \eqref{p0} holds the  non-compact phenomena. In this point of view, we can regard \eqref{p0} as the two dimensional counterpart of the Brezis-Nirenberg problem \cite{BN} in higher dimension. Due to the exponential nonlinearity, which is much stronger than the polynomial one in higher dimension, \eqref{p0} seems to contain new phenomena and difficulties. For example, as discussed in \cite{AP} and \cite{CT}, Palais-Smale sequences for the energy functional of \eqref{p0} may admit more complicated bahavior. 

Now, we start our discussion with the previous study of positive solutions of \eqref{p0}.  For simplicity, we only consider the typical case $h(t)=t$. Certain generalization is given in the following results. First, Adimurthi (\cite{A0}, \cite{A}) proves that \eqref{p} admits at least one positive solution for all $\la\in(0,\Lambda_1)$ where $\Lambda_1>0$ is the first eigenvalue of $-\Delta$ on $\Omega$ with the Dirichlet boundary condition. After that, in \cite{AS}, \cite{AD} and references therein, the authors investigate the asymptotic behavior of low energy positive solutions. By their results, we see that the least energy solution, obtained in \cite{A}, exhibits the single concentration and its full energy converges to $2\pi$ as $\la\to 0$. Moreover, we also observe in those results that the limit profile of the concentration is described by the Liouville equation via a suitable scaling. Furthermore, in \cite{D} and \cite{DT}, the classification of the asymptotic behavior of energy bounded sequences of positive solutions is accomplished. Especially, in \cite{D}, Druet proves that the limit full energy of any sequence is given by the sum of the energy of the weak limit and $2\pi N$ with a number $N\in \mathbb{N}\cup\{0\}$. Additionally, in \cite{DT}, Druet-Thizy show that the concentration occurs if and only if $\la\to0$. This implies that the weak limit of any concentrating solutions must be zero under their setting. Furthermore, they also give precise information of the location of concentration points and prove that they must be distinct. On the other hand, del Pino-Musson-Ruf \cite{DMR} obtain a sufficient condition for the existence of sequences of solutions carrying multiple bubbles. It gives a counterpart of the classification result by \cite{D} and \cite{DT}.  

More recently, the concentration behavior of sign-changing solutions is investigated. In this case, in view of the results in \cite{AY1} and \cite{AY2}, it is reasonable to consider a stronger perturbation $h(t)=te^{|t|^{\beta}}$ with $\beta\in(0,2)$. (We are still considering only the typical case for simplicity.) In fact, Adimrthi-Yadava \cite{AY2} show that \eqref{p0} admits at least one pair of sign-changing solutions for any $\la\in(0,\Lambda_1)$ and $\beta\in(1,2)$. They also prove that if $\Omega=B$, there exist infinitely many radial sign-changing solutions under the same assumption for $\la$ and  $\beta$. This stronger assumption is essential in the sense of the non-existence result by the same authors \cite{AY2}. In fact, they prove that if $\Omega=B$ and $\beta\le 1$, there exists a constant $\lambda_{\text{AY}}(\beta)>0$ such that \eqref{p0} permits no radial nodal solutions for all $\la\in(0,\lambda_{\text{AY}}(\beta))$. 

Motivated by this result, in \cite{GN}, the authors attack the blow--up analysis of  low energy nodal radial solutions. They investigate the behavior by fixing $\la\in(0,\lambda_{\text{AY}}(1))$ and taking the limit $\beta\downarrow 1$. Interestingly, the solutions sequence exhibits the multiple concentration at the origin and weakly converges to a nontrivial sign-definite solution of \eqref{p} with $\la\not=0$ and $\beta=1$. This behavior is very different from that in the positive case explained above. After that Grossi-Mancini-Naimen-Pistoia \cite{GMNP} construct a family of sign-changing solutions which concentrates at a $C^1$ stable critical point of a nontrivial residual mass. We remark that, in their construction, they choose the residual mass to be a sign-definite solution and assume that it is  nondegenerate and its amplitude is larger than $1/2$ at the concentration point. See (A1) and (A2) in their paper for more precise statements. Moreover, the authors conjecture that the largeness condition on the residual mass is essential for the existence of the concentrating solutions in the limit $\beta\downarrow 1$. See Remark 1.2 in their paper for the detail.

We lastly refer to some interesting results in the radial case. In \cite{MM1} and \cite{MM2}, the precise asymptotic expansion of concentrating positive solutions is obtained. It allows the proof of the multiplicity and nonexistence of critical points of the Trudinegr-Moser functional in the super critical case. Under another setting, Manicini-Thizy \cite{MT} construct concentrating radial solutions which weakly converge to a radial eigenfunction.  On the other hand, in two dimensional problem with the power type nonlinearity, Grossi-Grumiau-Pacella \cite{GGP} find that the singular limit profile appears in the asymptotic behavior of the sign-changing solutions. In higher dimensional problems with nearly Sobolev critical growth, Grossi-Salda\~{n}a-Hugo \cite{GSH} obtain sharp concentration estimates of nodal radial solutions for both of the Dirichlet and Neumann boundary value problems. 

In this paper, we proceed with our blow--up analysis of \eqref{p0} and answer some questions raised in the previous works. Moreover, we would like to find new concentration phenomena on \eqref{p0} in the sign-changing case. To this end, inspired by the previous works \cite{GN} and \cite{GMNP}, we focus on the strong perturbation problem \eqref{p}. Motivated also by \cite{MM1}, \cite{MM2} and \cite{GSH}, we focus on the radial case and establish the explicit and sharp estimates. 

More precisely, our first aim is to determine the limit profile and energy of every concentrating sequence. Especially, we would like to answer if the singular limit profile  (which was actually observed in the power type problem by \cite{GGP}) may appear or not in the sign changing case. We have already asked and negatively answered this question in the previous work \cite{GN}. But, in the crucial step to avoid the singular limit profile, we used the low energy characterization of the solutions. (See Case 2 in the proof of Lemma 4.3 in \cite{GN}.) Hence there remains a question if such characterization is essential for the conclusion. In the present paper, we will complete the answer for any energy bounded  sequence of radial solutions.  

Furthermore, our second goal is to deduce precise information of the weak limit of each concentrating sequence. The first question in this direction arose in the positive case by \cite{DT}. They asked if concentrating solutions could weakly converge to a nontrivial solution. Their answer was ``no" under their setting as noted above.  On the other hand, the result in \cite{GN} implies that  we may have a different answer in the sign-changing case. Moreover, a new question has arisen in Remark 2.1 of \cite{GMNP} about the relation between the concentration phenomenon and the amplitude of the weak limit at the concentration point. We will give an answer to this question in the radial case. 

As a consequence of our new calculation, we arrive at the next classification result with precise concentration estimates. It answers all the questions above. 
\subsection{Main theorems}\label{main}

Let us give our main results. Throughout this paper, we fix $\alpha>0$ and regard $(\la,\beta)$ as a parameter. Then, for any $(\la,\beta)\in(0,\infty)\times(0,2)$, we define the energy functional associated to \eqref{p},
\[
I_{\la,\beta}(u)=\frac12\int_B|\nabla u|^2dx-\la \int_BF_\beta(u)dx
\]
for all $u\in H_0^1(B)$ where $F_\beta(t)=\int_0^tte^{s^2+\alpha|s|^{\beta}}ds$. Then $I_{\la,\beta}$ is well-defined and $C^1$ functional on $H_0^1(B)$ by \eqref{TM}. Furthermore, the standard argument shows that every critical point of $I_{\la,\beta}$ corresponds to each solution of \eqref{p}. Moreover, let $\mathcal{N}_{\la,\beta}$ be the Nehari manifold defined by 
\[
\mathcal{N}_{\la,\beta}:=\left\{u\in H^1_0(B)\ | \text{ $u$ is radially symmetric and $\langle I_{\la,\beta}'(u),u\rangle=0$}\right\}.
\]
Using this, we put for any $k\in\{0\}\cup\mathbb{N}$,
\[
\begin{split}
c_{k,\la,\beta}:=\inf\Big\{&I_{\la,\beta}(u)\ |\ \text{$u\in\mathcal{N}_{\la,\beta}$, }\exists r_i\in(0,1) \text{ such that}\ \\
&0=r_0<r_1<\cdots<r_{k+1}=1, \ (0=r_0<r_1=1\text{ if }k=0,)\\
&u(x)=0\text{ if }|x|=r_i,\ u_i:=u|_{\{r_{i-1}<|x|<r_{i}\}},\ (-1)^{i-1}u_i>0,\text{ and } \\& u_i\in \mathcal{N}_{\la,\beta}\text{ (by zero extension), for any }\ 1\le i\le k+1\Big\}.
\end{split}
\]
In addition, we define the set of radial solutions as follows.
\[
\begin{split}
S_{k,\la,\beta}:=\Big\{&\text{$u\in C^2(B)\cap C^0(\overline{B})$: a radial solution of \eqref{p} } |\ \exists r_i\in(0,1)\ \text{ such}\\
&\text{that }0=r_0<r_1<\cdots<r_{k+1}=1,\ (0=r_0<r_1=1\text{ if $k=0$},)\\
&u(x)=0\text{ if }|x|=r_i,\ u_i:=u|_{\{r_{i-1}<|x|<r_{i}\}},\ \text{sgn}{(u(0))}(-1)^{i-1}u_i>0,\\
&\text{for all }1\le i\le k+1\Big\}.
\end{split}
\]
We remark that for any $(k,\la,\beta)\in\{0\}\times(0,\Lambda_1)\times (0,2)\cup\mathbb{N}\times(0,\Lambda_1)\times (1,2)$, there exists an element $u\in S_{k,\la,\beta}$ such that $I_{\la,\beta}(u)=c_{k,\la,\beta}$ by \cite{A} and \cite{AY1}. We call any element $u\in S_{k,\la,\beta}$ a nodal radial solution if $k\not=0$ and a positive (negative) solution  or  a sign-definite solution if  $k=0$ and $u(0)>0$ ($<0$ respectively). Note that the result in \cite{GNN} shows  that any element $u \in S_{0,\la,\beta}$ satisfies $\max_{x\in \overline{B}}|u(x)|=|u(0)|$. Lastly, for any $k\in \mathbb{N}$, let $\Lambda_k$ be the eigenvalue of $-\Delta$ on $B$ with the Dirichlet boundary condition which corresponds to the radial eigenfunction $\varphi_k$ which has  numbers $0=\tau_0<\tau_1<\cdots<\tau_k=1$ such that $\varphi_k(x)=0$ if $|x|=\tau_i$ and $(-1)^{i-1}\varphi_k>0$ on $(\tau_{i-1},\tau_i)$ for all $i=1,\cdots,k$ and is normalized by $\varphi_k(0)=1$. Our main result is the following.
\begin{theorem}[Classification result]\label{a1}
Assume $(k,\la_*,\beta_*)\in \{0\}\times[0,\infty)\times(0,2)\cup \mathbb{N}\times[0,\infty)\times(0,3/2)$ and let $\{(\la_n,\beta_n)\}\subset(0,\infty)\times(0,2)$ be a sequence of values such that $(\la_n,\beta_n)\to(\la_*,\beta_*)$ as $n\to \infty$. Furthermore, we suppose $(u_n)$ is a sequence of solutions of \eqref{p} such that $u_n\in S_{k,\la_n,\beta_n}$ and  $u_n(0)>0$ for all $n\in \mathbb{N}$. In addition, if $k\not=0$, we assume,
\[
\int_B|\nabla u_n|^2dx\text{ is uniformly bounded  for all $n\in \mathbb{N}$}.
\] 
Then, after subtracting a suitable subsequence if necessary, we have a function $u_0$ and  a number $N\in\{0,1,\cdots,k+1\}$ such that $u_n \to u_0$ in $C^{2}_{\text{loc}}(\overline{B}\setminus\{0\})$,
\[
I_{\la_n,\beta_n}(u_n)\to 2\pi N+I_{\la_*,\beta_*}(u_0),
\]
and
\[
\int_B|\nabla u_n|^2dx\to 4\pi N+\int_B|\nabla u_0|^2dx,
\] 
as $n\to \infty$. Moreover, if $\max_{x\in\overline{B}}|u_n(x)|\to \infty$ as $n\to \infty$, then, up to a subsequence, either one of the next assertions holds,
\begin{enumerate}
\item[(i)] $\la_*=0$, $\beta_n>1$ for all $n\in \mathbb{N}$,  $N=k+1$, and $u_0=0$, or
\item[(ii)] $\la_*=0$, $\beta_n\le1$ for all $n\in \mathbb{N}$,  $k=0$, $N=1$, and $u_0=0$, or
\item[(iii)] $\la_*\not=0$, $\beta_n\downarrow1$, $k\not=0$, $0<N<k+1$, and $u_0\in S_{k-N,\la_*,1}$ with $(-1)^Nu_0(0)\ge\alpha/2$, or 
\item[(iv)] $\la_*\not=0$, $\beta_n=1$ for all $n\in \mathbb{N}$, $k\not=0$, $N=1$,  and $u_0\in S_{k-1,\la_*,1}$ with $-u_0(0)=\alpha/2$ or
\item[(v)] $\la_*\not=0$, $\beta_n\uparrow 1$ for all $n\in \mathbb{N}$, $k\not=0$, $N=1$,  and $u_0\in S_{k-1,\la_*,1}$ with $0< -u_0(0)\le\alpha/2$, or otherwise
\item[(vi)] $\la_*=\Lambda_k$, $\beta_n<1$ for all  $n\in \mathbb{N}$, $k\not=0$, $N=1$,  and $u_0=0$.
\end{enumerate}
On the other hand, if $\max_{x\in\overline{B}}|u_n(x)|$ is uniformly  bounded for all $n\in \mathbb{N}$, then
\begin{enumerate}
\item[(vii)] $\la_*\not=0$, $N=0$, $u_n \to u_0$ in  $C^{2}(\overline{B})$,  $u_0\in S_{k,\la_*,\beta_*}\cup\{0\}$,  and further, we get $u_0=0$ only if $\la_*= \Lambda_{k+1}$. 
\end{enumerate}
In particular, (i) ((ii)) happens if and only if $\la_*=0$ ($k=0$ and $\la_*=0$ respectively) and (vii) occurs if $k=0$ and $\la_*\not=0$, or $k\ge1$, $\la_*\not=0$, and $\beta_*>1$, or $k\ge1$, $0<\la_*\not=\Lambda_k$, and $\beta_*<1$. 
\end{theorem}

Theorem \ref{a1} implies that the full (Dirichlet) energy of the sequence is decomposed by $2\pi N$ ($4\pi N$ respectively) for a number $N\in \{0,1,\cdots,k+1\}$ and the energy of the weak limit $u_0$. This is the typical energy quantization phenomenon observed also in the previous works. In our theorem, (i)-(vi) describe the non-compact behavior and (vii) corresponds to the compact one. In the former case, there are three situations. The first one is found in (i) and (ii) which means that the $(k+1)$-concentration occurs with the zero weak limit. This phenomenon happens  if and only if $\la_*=0$. The second one is observed in (iii), (iv), (v). It shows the $N(<k+1)$-concentration happens with the nontrivial weak limit. (Notice that the weak limit is possibly sign-changing.) This behavior yields $\la_*\not=0$ and $\beta_*=1$. Moreover, the sum of the number $N$ of bubbles, and the number $k-N+1$ of nodal domains of the weak limit is always given by $N+(k-N+1)=k+1$. It comes from the fact that, in this case, if we focus on the behavior of the solution on each nodal domain, it weakly converges to zero if and only if it blows up if and only if it exhibits the single concentration. See the next theorem for the detail.  The third one is observed in (vi). This means that the $N(=1<k+1)$-concentration occurs with the zero weak limit. This behavior requires $\la_*=\Lambda_k$ and $\beta_n<1$ for all $n\in \mathbb{N}$. As we will see in the next theorems, in this case, only the local maximum value $u_n(0)$ at the origin diverges to  infinity and the other ones converge to zero. 

Moreover, a remarkable result is found in the final assertions in (iii)-(vi). It gives the necessary condition on  the amplitude $|u_0(0)|$ of the weak limit  at the origin. Especially, (iii) gives a proof of the conjecture in Remark 1.2 of \cite{GMNP} in the radial case. It ensures that if the concentration occurs in the limit $\la_n\to \la_*\not=0$ and $\beta_n\downarrow1$, then $|u_0(0)|$ needs to be  greater than or equal to $\alpha/2$. Notice that our necessary condition is valid in any case of $\la_*>0$ as far as the concentration occurs as $\beta_n\downarrow 1$ while the previous results in \cite{GN} and \cite{GMNP} focus on the case of small $\la_*>0$. Moreover, in the cases $\beta_n=1$ for all $n\in \mathbb{N}$, $\beta_n\uparrow1$, and $\beta_*<1$, we deduce new necessary conditions, $|u_0(0)|=\alpha/2$, $0\le|u_0(0)|\le\alpha/2$, and $u_0(0)=0$ respectively. (A related result is observed for a radial positive sequence in another setting by Theorem 0.3 in \cite{MT}.) These conditions will be useful to detect new concentrating sequences of solutions.  See Section \ref{sec:cor} for more discussion. We will discuss several counterparts of assertions in the previous theorem later.

In addition, we remark that there is  a striking difference between the cases $(\beta_n)\subset(1,2)$ ((i) and (iii)) and $(\beta_n)\subset(0,1]$ ((ii), (iv), (v), and (vi)). In the former case, $N$ can be greater than $1$ while in the latter case it has to be equal to one. This is a consequence of our blow-up analysis in Sections \ref{le} and \ref{energy}. See Remark \ref{rmk:e1} for more explanation. Then, we notice that Theorem \ref{a1} contains the following nonexistence result. 
\begin{corollary}\label{cor:ne}
For any number $k\in \mathbb{N}$ and sequence $\{(\la_n,\beta_n)\}\subset(0,\infty)\times(0,1]$ such that $(\la_n,\beta_n)\to (\la_*,\beta_*)\in\{0\}\times(0,1]$, there exists no sequence of solutions $(u_n)$ such that $u_n\in S_{k,\la_n,\beta_n}$ and $\int_B|\nabla u_n|^2dx$ is uniformly bounded for all $n\in \mathbb{N}$.
\end{corollary}

This conclusion allows a partial proof of the nonexistence result by \cite{AY2} via a different approach. 
\begin{corollary}\label{cor:final} Choose any $\beta\in(0,1]$. Then for each value $E>0$, there exists a constant $\hat{\lambda}=\hat{\lambda}(\beta,E)>0$ such that for all $\la\in(0,\hat{\la})$, \eqref{p} admits no radial nodal solution $u$ with $\int_B|\nabla u|^2dx\le E$. In particular, for each number $k\in \mathbb{N}$, there is a value $\hat{\la}_0=\hat{\la}_0(\beta,k)>0$ such that for all  $\la\in(0,\hat{\la}_0)$, there exists no solution $u\in S_{k,\la,\beta}$ which attains $c_{k,\la,\beta}$.
\end{corollary}

We finally remark on the additional condition on $\beta_*$ in the previous theorem.
\begin{remark}\label{rmk:a1}
In the case of $k\ge1$, we additionally assumed $\beta_*<3/2$. This condition will first appear in \eqref{e2} of Lemma \ref{lem:e} below. This does not seem simply a technical assumption. As discussed in Remark \ref{rmk:e2}, in the case $k\ge1$ and $\beta_*\ge3/2$, we would have different formulas of the asymptotic energy expansion in Theorem \ref{a10} and also the concentration estimates in Theorem \ref{a31} below. Since the proof for the former case has already used many pages, we leave the latter case for our next works.  
We here restrict ourselves to only conjecture that, in the sign-changing case, the stronger perturbation ($\beta\ge3/2$ in \eqref{p}) would delicately affects the precise asymptotic formulas of concentrating solutions.  
\end{remark}

Next, we shall check the detail of the behavior stated above. The previous theorem is a direct consequence of the next two theorems.
\begin{theorem}[Limit profile and energy]\label{a10} Let values $k,\la_*,\beta_*$ and sequences $\{(\la_n,\beta_n)\}$, $(u_n)$ be chosen as in the assumption of Theorem \ref{a1} and assume $u_0\in H^1_0(B)$ is the weak limit of $(u_n)$ by subtracting a subsequence if necessary. Furthermore, we write $u_n=u_n(|x|)$ and $u_0=u_0(|x|)$ for $x\in \overline{B}$ and then define values  $0=r_{0,n}<r_{1,n}<\cdots<r_{k+1,n}=1$ ($0=r_{0,n}<r_{1,n}=1$ if $k=0$) so that $u_n(r_{i,n})=0$ and  $(-1)^{i-1}u_n(r)>0$ if $r_{i-1,n}<r<r_{i,n}$ for all $1\le i\le k+1$.  In addition, we set $u_{i,n}=u_n|_{[r_{i-1,n},r_{i,n}]}$ and define numbers $\rho_{i,n}\in[r_{i-1,n},r_{i,n})$ and $\mu_{i,n}>0$ by $\mu_{i,n}=|u_{i,n}(\rho_{i,n})|=\max_{r\in [r_{i-1,n},r_{i,n}]}|u_{i,n}(r)|$ for all $1\le i\le k+1$.  Finally for each number $i=1,\cdots,k+1$ such that $\mu_{i,n}\to \infty$, we set $\ga_{i,n}>0$ so that 
\[
1=2\lambda_n  \mu_{i,n}f_n(\mu_{i,n})\ga_{i,n}^2
\]
where $f_n(t)=te^{t^2+|t|^{\beta_n}}$, and define
\[
z_{i,n}(r):=2\mu_{i,n}(|u_{i,n}(\gamma_{i,n} r+\rho_{i,n})|-\mu_{i,n})
\text{ for all }
r\in\left[\frac{r_{i-1,n}-\rho_{i,n}}{\ga_{i,n}},\frac{r_{i,n}-\rho_{i,n}}{\ga_{i,n}}\right],
\] 
and $z(r):=\log{(64/(8+r^2)^2)}$ which  is a solution of the Liouville equation
\[
-z''-\frac{1}{r}z'=e^z\text{ in }(0,\infty),\ z(0)=0=z'(0),\  \int_0^\infty e^zrdr<\infty.
\]
Then if $\max_{r\in[0,1]}|u_n(r)|\to \infty$, either one of the next  assertions (i) and (ii) holds up to a subsequence.
\begin{enumerate}
\item[(i)] For all $i=1,\cdots,k+1$, we have $\mu_{i,n}\to \infty$, $\rho_{k+1,n}\to 0$, $\ga_{i,n}\to0$, $z_{i,n}\to z\text{ in }C^2_{\text{loc}}((0,\infty))\cap C^1_{\text{loc}}([0,\infty))$,
\[
\int_{r_{i-1,n}}^{r_{i,n}}u_{i,n}'(r)^2rdr= 2-\frac{\alpha \beta_*}{\mu_{i,n}^{2-\beta_n}}+o\left(\frac{1}{\mu_{i,n}^{2-\beta_n}}\right),
\]
\[
\mu_{i,n}\int_{r_{i-1,n}}^{r_{i,n}}\la_n f_n(u_{i,n})rdr= 2-\frac{\alpha \beta_*}{\mu_{i,n}^{2-\beta_n}}+o\left(\frac{1}{\mu_{i,n}^{2-\beta_n}}\right),
\]
and further, $\la_*=0$ and  $u_n\to u_0=0$ in $C^2_{\text{loc}}((0,1])$.
\item[(ii)] There exists a number $N\in\{1,\cdots,k\}$ such that for all $i=1,\cdots,N$, we have $\mu_{i,n}\to \infty$, $r_{i,n}\to 0$, $\ga_{i,n}\to0$, $z_{i,n}\to z\text{ in }C^2_{\text{loc}}((0,\infty))\cap C^1_{\text{loc}}([0,\infty))$, and 
\[
\int_{r_{i-1,n}}^{r_{i,n}}u_{i,n}'(r)^2rdr= 2-\frac{\alpha \beta_*}{\mu_{i,n}^{2-\beta_n}}+o\left(\frac{1}{\mu_{i,n}^{2-\beta_n}}\right),
\] 
\[
\mu_{i,n}\int_{r_{i-1,n}}^{r_{i,n}}\la_n f_n(u_{i,n})rdr= 2-\frac{\alpha \beta_*}{\mu_{i,n}^{2-\beta_n}}+o\left(\frac{1}{\mu_{i,n}^{2-\beta_n}}\right),
\]
while for all $i=N+1,\cdots,k+1$, there exist values $\mu_{i}\ge0$, $r_i\in(0,1]$ and $\rho_i\in[0,1)$ such that $\mu_{i,n}\to \mu_{i}$, $r_{i,n}\to r_i$, $\rho_{i,n}\to \rho_i$, and $0=\rho_{N+1}<r_{N+1}<\cdots<\rho_{k+1}<r_{k+1}=1$ if $N<k$ and $0=\rho_{k+1}<r_{k+1}=1$ if $N=k$, and further, it holds that $\la_*\not=0$, $u_n|_{[r_{N,n},1]}\to u_0$ in $C^2_{\text{loc}}((0,1])$, $\lim_{r\to 0+0}(-1)^Nu_0(r)=\mu_{N+1}$, and 
\[
\int_{r_{N,n}}^{1}u_{n}'(r)^2rdr\to \int_{0}^{1}u_0'(r)^2rdr.
\]
Moreover, either one of the next assertions holds,
\begin{enumerate}
\item[(a)] $\mu_{N+1}>0$, $u_0(r_i)=0$, and $(-1)^{i-1}u_0>0$ on $(r_{i-1},r_i)$, for all $i=N+1,\cdots,k+1$, or
\item[(b)] $\mu_{N+1}=0$, $u_0=0$, and further, $u_n|_{[r_{N,n},1]}/\mu_{N+1}\to(-1)^N\varphi_{k-N+1}$ in $C^2_{\text{loc}}((0,1])$ and $\la_*=\Lambda_{k-N+1}$.
\end{enumerate}
\end{enumerate}
On the other hand, if $\max_{r\in[0,1]}|u_n(r)|$ is uniformly bounded, choosing a subsequence if necessary, we get that,
\begin{enumerate}
\item[(iii)] for any $i=1,\cdots,k+1$, there exist values $\mu_i\ge0$, $r_i\in(0,1]$ and $\rho_i\in[0,1)$ such that $\mu_{i,n}\to \mu_i$, $r_{i,n}\to r_{i}$, $\rho_{i,n}\to \rho_i$, and $0= \rho_1<r_1<\cdots <\rho_{k+1}<r_{k+1}=1$ if $k\ge1$,  and further, it holds that $\la_*\not=0$, $u_n\to u_0$ in $C^2([0,1])$, and  
\[
\int_{0}^{1}u_{n}'(r)^2rdr\to \int_0^{1}u_0'(r)^2rdr.
\]
In addition, either one of the next assertions is true,
\begin{enumerate}
\item[(a)] $u_0(0)>0$, $u_0(r_i)=0$, and $(-1)^{i-1}u_0>0$ on $(r_{i-1},r_i)$, for all $i=1,\cdots,k+1$, or
\item[(b)] $u_0=0$, $u_n/u_{n}(0)\to \varphi_{k+1}$ in $C^2([0,1])$, and $\la_*=\Lambda_{k+1
}$.
\end{enumerate}
\end{enumerate}
\end{theorem}

This theorem shows the behavior on every part $u_{i,n}$ between neighboring two zero points $r_{i-1,n}<r_{i,n}$. From the behavior in (i) and (ii), we see that if $u_{i,n}$ blows up, it concentrates at the origin. Especially, the local maximum point $\rho_{i,n}$ converges to the origin. Actually, this is the reasonable and the only way for any blowing up solution to ensure the uniform boundedness of the energy. Furthermore, the limit profile is determined uniquely by the classical solution $z$ of the Liouville equation and the limit energy  is  just equal to $2$. This implies  that  neither the singular limit profile, observed in the power type problem in \cite{GGP}, nor the multiple concentration, occurs on any $u_{i,n}$.

Notice also that due to our strong perturbation, the second term of the right hand side of the energy expansion in (i) and (ii) is very different from that in the case of $\alpha=0$  in view of its sign and the exponent on $\mu_{i,n}$. (See Theorem 1 in \cite{MM2}.) 

Finally, we obtain precise concentration estimates  in terms of $(\la_n,\beta_n)$.
\begin{theorem}[Concentration estimates]\label{a31}
We assume as in Theorem \ref{a10} and either one of (i) and (ii) in the theorem occurs. First suppose (i) happens. Then we have that $(\beta_n)\subset(0,1]$ yields $k=0$. Moreover, if $k\in \{0\}\cup \mathbb{N}$, we get that
\begin{align}
&\lim_{n\to \infty}\frac{\log{\frac1{\la_n}}}{\mu_{k+1,n}^{\beta_n}}=\alpha \left(1-\frac{\beta_*}{2}\right),\label{aaa1}
\end{align}
and if $k\ge1$, we obtain for all $1\le i\le k$ that
\begin{align}
&\lim_{n\to \infty}\frac{\log{\frac1{\la_n}}}{\mu_{i,n}^{\beta_n(\beta_n-1)^{k-i+1}}}=\left\{\alpha \left(1-\frac{\beta_*}{2}\right)\right\}^{\frac{2-\beta_*(\beta_*-1)^{k-i+1}}{2-\beta_*}},\label{aa1}
\end{align}
Furthermore, if $k\in \{0\}\cup \mathbb{N}$, we have that
\begin{align}
&\lim_{n\to \infty}\left(\log{\frac1{\la_n}}\right)^{\frac1{\beta_n}}|u_{k+1,n}'(1)|=2\left\{\alpha\left(1-\frac{\beta_*}{2}\right)\right\}^{\frac1{\beta_*}},
\label{aa44}
\end{align}
and if $k\ge1$, we get for all $1\le i\le k$ that
\begin{equation}
\begin{split}
\lim_{n\to \infty}&\frac{\log{\frac1{\la_n}}}{\left(\log{\frac{1}{r_{i,n}}}\right)^{(\beta_n-1)^{k-i+1}}}\\
&\ \ \ \ \ \ \ \ \ \ \ \ \ \ \ \ =2^{(\beta_*-1)^{k-i+1}}\left\{\alpha \left(1-\frac{\beta_*}{2}\right)\right\}^{\frac{2-2(\beta_*-1)^{k-i+1}}{2-\beta_*}},
\end{split}\label{aa2}
\end{equation}
and
\begin{equation}
\begin{split}
\lim_{n\to \infty}&\frac{\log{\frac1{\la_n}}}{(\log{|u_{i,n}'(r_{i,n})|})^{(\beta_n-1)^{k-i+1}}}\\
&\ \ \ \ \ \ \ \ \ \ \ \ \ \ \ \ =2^{(\beta_*-1)^{k-i+1}}\left\{\alpha \left(1-\frac{\beta_*}{2}\right)\right\}^{\frac{2-2(\beta_*-1)^{k-i+1}}{2-\beta_*}}.
\end{split}\label{aa4}
\end{equation}
In addition, if $k\ge2$, we have for all $i=2,\cdots,k$, that 
\begin{equation}
\begin{split}
\lim_{n\to \infty}&\frac{\log{\frac{1}{\la_n}}}{\left(\log{\frac{1}{\rho_{i,n}}}\right)^{\frac{\beta_n(\beta_n-1)^{k-i+1}}{2}}}\\
&\ \ \ \ \ \ \ \ \ \ \ \ \ \ \ \ =2^{\frac{\beta_*(\beta_*-1)^{k-i+1}}{2}}\left\{\alpha\left(1-\frac{\beta_*}{2}\right)\right\}^{\frac{2-\beta_*(\beta_*-1)^{k-i+1}}{2-\beta_*}}.
\end{split}\label{aa3}
\end{equation}
Lastly, suppose $k\ge1$. Then if there exists a number $L\ge0$ such that
\begin{equation}\label{ab11}
\lim_{n\to \infty}\frac{\log{\log{\frac1{\la_n}}}}{(\beta_n-1)\left(\log{\frac1{\la_n}}\right)^{\frac{2}{\beta_n}}}= L,
\end{equation}
we get
\begin{equation}\label{ab1}
\begin{split}
\lim_{n\to \infty}\frac{\log{\frac1{\la_n}}}{\left(\log{\frac{1}{\rho_{k+1,n}}}\right)^{\frac{\beta_n}{2}}}&=2^{\frac{\beta_*}{2}}\alpha \left(1-\frac{\beta_*}{2}\right)\\
&\ \ \ \times\left[1+L\left\{\alpha \left(1-\frac{\beta_*}{2}\right)\right\}^{\frac{2}{\beta_*}}\right]^{-\frac{\beta_*}{2}},
\end{split}
\end{equation}
and, on the other hand, if
\[
\lim_{n\to \infty}\frac{\log{\log{\frac1{\la_n}}}}{(\beta_n-1)\left(\log{\frac1{\la_n}}\right)^{\frac{2}{\beta_n}}}=\infty,
\]
we necessarily have $\beta_*=1$ and obtain
\begin{equation}\label{ab2}
\lim_{n\to \infty}\frac{\log{\log{\frac1{\la_n}}}}{(\beta_n-1)\log{\frac{1}{\rho_{k+1,n}}}}=2.
\end{equation}

Next, we suppose (ii) occurs. Then we get $\beta_*\le1$ and that $(\beta_n)\subset (0,1]$ implies $N=1$. Moreover, if $1\le N\le k$, either one of the next assertions (a) and $(b)$ is true.
\begin{enumerate}
\item[(a)] $\mu_{N+1}>0$, $\beta_*=1$, and 
for all $1\le i\le N$, it holds that
\begin{align}
&\lim_{n\to \infty}\mu_{i,n}^{(\beta_n-1)^{N-i+1}}=\frac{2\mu_{N+1}}{\alpha},\label{aa5}\\
&\lim_{n\to \infty}\left(\log{\frac1{r_{i,n}}}\right)^{(\beta_n-1)^{N-i+1}}=\frac{2\mu_{N+1}}{\alpha},\label{aa6}\\
&\lim_{n\to \infty}\left(\log{|u_{i,n}'(r_{i,n})|}\right)^{(\beta_n-1)^{N-i+1}}=\frac{2\mu_{N+1}}{\alpha},
\label{aa7}
\end{align}
and
\begin{align}
&\lim_{n\to \infty}u_{N+1,n}'(r_{N+1,n})=-\frac{1}{r_{N+1}}\la_* \int_0^{r_{N+1}}f_*(u_0)rdr\label{aa10},\\
&\lim_{n\to \infty}\rho_{N+1,n}^{\beta_n-1}= \sqrt{\frac{\alpha}{2\mu_{N+1}}},\label{aa9}
\end{align}
where $f_*(t)=te^{t^2+\alpha|t|}$. Especially, $2\mu_{N+1}/\alpha>1$ ($\in (0,1)$) implies $\beta_n>1$ ($<1$ respectively) for all $n\in \mathbb{N}$. On the other hand, $\beta_n>1$ ($=1$, $<1$) for all $n\in \mathbb{N}$ requires $2\mu_{N+1}/\alpha\ge 1$ ($=1$, $\le1$ respectively). Finally, if $1<N\le k$, (which yields $k\ge2$, $\beta_n>1$ for all $n\in \mathbb{N}$, and $2\mu_{N+1}/\alpha\ge1$), assuming $2\mu_{N+1}/\alpha>1$, we get for all $2\le i\le N$,
\begin{align}
&\lim_{n\to \infty}\left(\log{\frac1{\rho_{i,n}}}\right)^{(\beta_n-1)^{N-i+1}}=\left(\frac{2\mu_{N+1}}{\alpha}\right)^{2}.
\label{aa8}
\end{align}
\item[(b)] $\mu_i=0$ for all $i=N+1,\cdots,k+1$ and $\beta_n<1$ for all $n\in \mathbb{N}$. 
\end{enumerate}
\end{theorem}
\begin{remark} We assumed $2\mu_{N+1}/\alpha>1$ for \eqref{aa8}. This can be verified, for example, when  $N=k$ and $\la_*\in(0,\Lambda_1)$ is small enough  by Lemma \ref{posi} below. 
\end{remark}

This theorem describes the speed of the divergence or convergence of $\mu_{i,n}$, $r_{i,n}$, $\rho_{i,n}$ and $u_{i,n}'(r_{i,n})$ in terms of the parameter $(\la_n,\beta_n)$. Especially, in the case of (i), thanks to \eqref{aaa1} and \eqref{aa1}, we get that $\mu_{k+1,n}=(\alpha(1-\beta_*/2)+o(1))^{-1/\beta_*} (\log{(1/\la_n)})^{1/\beta_n}$ and, if $k\ge1$, that $\mu_{k+1,n}/\mu_{i,n}\to0$ as $n\to \infty$ for all $i=1,\cdots,k$. Then, combining this together with the asymptotic energy formula in Theorem \ref{a10}, we can also write the energy expansion in terms of $(\la_n,\beta_n)$ as follows.
\begin{corollary}\label{expansion} Assume as in Theorem \ref{a1} and suppose (i) or (ii) of the theorem occurs. Then we get 
\[
\int_B|\nabla u_n|^2dx=4\pi(k+1)-\frac{2\pi\alpha^{\frac{2}{\beta_*}}\beta_*\left(1-\frac{\beta_*}{2}\right)^{\frac{2-\beta_*}{\beta_*}}}{\left(\log\frac1{\la_n}\right)^{\frac{2-\beta_n}{\beta_n}}}+o\left(\frac1{{\left(\log\frac1{\la_n}\right)^{\frac{2-\beta_n}{\beta_n}}}}\right)
\] 
as $n\to \infty$.
\end{corollary}

Moreover, in the case of (i), we observe a delicate behavior when $(\la_*,\beta_*)=(0,1)$ by the formulas \eqref{ab1} and \eqref{ab2}. They show that the asymptotic behavior $\rho_{k+1,n}\to0$ of the local maximum point is described by either one of two different formulas \eqref{ab1} and \eqref{ab2}. The choice is determined by  the balance of the speed of two limits $\la_n\to 0$ and $\beta_n\to1$. If former one is quicker than the latter one in the sense of \eqref{ab11}, we have \eqref{ab1} and otherwise we get \eqref{ab2}. Actually, in the latter formula, the effect of the limit $\beta_n\to1$ appears more clearly. This phenomenon comes from the combined effect of the two different behavior, the $(k+1)$-concentration with the zero weak limit in the case $\la_n\to0$, and the $k$-concentration with the nontrivial weak limit in the case $0<\la_*\ll 1$ and $\beta_*=1$ (which is observed in \cite{GN} and (b) of Proposition \ref{cr:1} below). 

Finally, in the case of (ii), one of the most  important results is \eqref{aa5}. This proves the necessary condition, explained before, in the final assertions in (iii)-(v) of Theorem \ref{a1}. 

In the following sections, we give the proof of our main theorems.
\subsection{Outline of the proof}
We carry out the blow--up analysis based on a scaling technique. We begin with studying the limit profile of the concentration part as in \cite{GN}. The first difficulty arises here since we do not have the low energy characterization of solutions (Lemma 2.1 in \cite{GN}). In order to admit our wider setting, we change the proof and argue  by induction. Our idea is to use useful assumptions (\eqref{a4} and \eqref{a7}) which are ensured by the previous step of the induction argument (Proposition \ref{f0}). Utilizing this idea, we succeed in avoiding the case of the singular limit profile. See the proof of Proposition \ref{dd}. 

Next, we will determine the energy of each concentrating part as in Theorem \ref{a10}. An important step is to ensure that only the single concentration occurs in each nodal domain. The point wise estimate in Lemma \ref{le70} will work for it. This is an extension of Lemma 5 in \cite{MM1} (see also Lemma 13 in  \cite{MM2}) to our setting which allows the strong perturbation and also the sign-changing case. Using this and arguing as in Proof of Theorem 1 in \cite{MM2}, we obtain the energy expansion in Proposition \ref{e00}. To accomplish these proofs, some careful remarks are needed. In particular, the estimate \eqref{e2} for the error term is a key for the proof in the sign-changing case. 

Finally, we will obtain the precise concentration estimate in Theorem \ref{a31}.  The proof is inspired by the argument in \cite{GSH} for the power type problem in higher dimension. A new difficulty comes from the fact that, of course, the form of the nonlinearity is very different. In particular, the Pohozaev identity does not seem work well for our aim.  In our proof, utilizing the useful identity in Lemma \ref{lem:id}, instead of the Pohozaev identity, with the energy expansion in Proposition \ref{e00}, we get the key assertions in Proposition \ref{f0}. This is also crucial to proceed with our argument by induction.

We lastly remark that our approach mentioned above allows the proof without quoting the uniqueness of solutions which has not been completed for \eqref{p} yet except for large positive solutions (\cite{AKG}).
\subsection{Organization, definitions and notations}\label{org}
This paper consists of 6 sections. We begin with two sections, Sections  \ref{le} and \ref{energy}, which are mainly devoted to obtain the limit equation and the limit energy of  concentrating solutions respectively. Next in Section \ref{more}, we analyze the behavior of non-concentrating parts of solutions. This is important to deduce the precise information  of the weak limit of solutions. Next, in Sections \ref{proof}, we complete the proof of main theorems. Finally, in Section \ref{sec:cor}, we discuss some counterparts of our classification result. Additionally, the proof of Lemma \ref{ap1} is given in Appendix \ref{ap} for the readers' convenience. 

Throughout these  sections, we assume $\{(\la_n,\beta_n)\}\subset (0,\infty)\times (0,2)$, $(\la_*,\beta_*)\in [0,\infty)\times (0,2)$ and $(\la_n,\beta_n)\to (\la_*,\beta_*)$ as $n\to \infty$.  We will impose more conditions on $\la_n,\beta_n,\la_*,\beta_*$ when  needed. Moreover, we choose any $k\in\{0\}\cup \mathbb{N}$ and consider a sequence of solutions $(u_n)$ such that $u_n\in S_{k,\la_n,\beta_n}$ for all $n\in \mathbb{N}$. We set $f_n(t)=te^{t^2+\alpha |t|^{\beta_n}}$ and $f_*(t)=te^{t^2+\alpha |t|^{\beta_*}}$. 
 
Furthermore, we define the norm in $H^1_0(B)$ by $\|\cdot\|_{H^1_0(B)}:=\left(\int_B|\nabla \cdot|^2dx\right)^{1/2}$. Moreover, 
 we denote the first kind Bessel function of order zero by $J_0$ which is defined by 
\[
J_0(r):=\sum_{j=0}^{\infty}\frac{(-1)^j}{(j!)^2}\left(\frac r2\right)^{2j} \ (r\in \mathbb{R}).
\]
For any $k\in \mathbb{N}$, let $\Lambda_k$ and $\varphi_k$ be the eigenvalue and radial eigenfunction of $-\Delta$ on $B$ with the Dirichlet boundary condition defined above. Then letting $0<t_1<t_2<\cdots$ be all the zeros of $J_0$ on $(0,\infty)$,  we have that $\Lambda_k=t_k^2$ and $\varphi_k(x)=J_0(t_k |x|)$.

Finally, in the proofs, we often use the same character $C$ to denote several constants when the explicit value is not very important.
\section{Limit profile} \label{le}
Let us start the proof of main theorems. In the following, we refer to Radial Lemma in \cite{S}. In our two dimensional setting, it is reduced to the following.
\begin{lemma}[\cite{S}]\label{RL}
There exists a constant $c>0$ such that  every radial function $u\in H_0^1(B)$ is almost everywhere equal to a function $\tilde{u}(x)$,  continuous for $x\not=0$, such that 
\[
|\tilde{u}(x)|\le c|x|^{-\frac{1}{2}}\|u\|_{H^1_0(B)}\ (x\in B\setminus\{0\}).
\]
\end{lemma} 

Now, assume $k\in \{0\}\cup\mathbb{N}$. For any $1\le i\le k+1$, we define  $u_{i,n}$, $r_{i-1,n}$, $r_{i,n}$, $\rho_{i,n}$, and  $\mu_{i,n}$ as in Theorem \ref{a10}.  Then we have 
\begin{equation}\label{rad}
\begin{cases}
-u_{i,n}''-\frac{1}{r}u_{i,n}'=\la_n f_n(u_{i,n}),\ (-1)^{i-1}u_{i,n}>0\text{ in }(r_{i-1,n},r_{i,n}),\\
u_{i,n}(r_{i,n})=0=u_{i,n}'(\rho_{i,n}),\\
u_{i,n}(r_{i-1,n})=0\text{ if }i\ge 2,
\end{cases}
\end{equation}
where $f_n(t):=te^{t^2+\alpha|t|^{\beta_n}}$.  We often use the next identity.
\begin{lemma}\label{lem:id} For any $i=1,\cdots,k+1$, we have 
\[
u_{i,n}(\rho_{i,n})=\int_{\rho_{i,n}}^{r_{i,n}}\la_nf_n(u_{i,n})r\log{\frac{r_{i,n}}{r}}dr.
\]
Moreover, if $i\not=1$, we get
\[
u_{i,n}(\rho_{i,n})=\int_{r_{i-1,n}}^{\rho_{i,n}}\la_nf_n(u_{i,n})r\log{\frac{r}{r_{i-1,n}}}dr
\]
\end{lemma}
\begin{proof} Let us show the first formula. Fix any $i=1,\cdots,k+1$. Multiplying the equation in \eqref{rad} by $r\log{r}$, and integrating by parts from  $\rho_{i,n}$ to $r_{i,n}$, we get
\[
\begin{split}
\int_{\rho_{i,n}}^{r_{i,n}}\la_n f_n(u_{i,n})r\log{r}dr&=\int_{\rho_{i,n}}^{r_{i,n}}(-ru_{i,n}'(r))'\log{r}dr\\
&=\int_{\rho_{i,n}}^{r_{i,n}}\la_n f_n(u_{i,n})rdr\log{r_{i,n}}-u_{i,n}(\rho_{i,n}),
\end{split}
\]
where we used $r_{i,n}u_{i,n}'(r_{i,n})=-\int_{\rho_{i,n}}^{r_{i,n}}\la_n f_n(u_{i,n})rdr$ for the last equality. This shows the first formula. Assuming $i\not=1$, the second assertion is similarly obtained by integrating by parts from $r_{i-1,n}$ to $\rho_{i,n}$. 
\end{proof}

We also use the next assertion.
\begin{lemma}\label{newl1}
If $\int_{r_{i-1,n}}^{r_{i,n}}u_{i,n}'(r)^2rdr\to 0$, then $u_{i,n}\to 0$ in $C([0,1])$. In particular, if $\limsup_{n\to \infty} \mu_{i,n}>0$, then by taking a subsequence if necessary, we get a constant $K_0>0$ such that $\int_{r_{i-1,n}}^{r_{i,n}}u_{i,n}'(r)^2rdr\ge K_0$ for all $n\in\mathbb{N}$.
\end{lemma}
\begin{proof} We put $A_n^2:=\int_{r_{i-1,n}}^{r_{i,n}}u_{i,n}'(r)^2rdr$. Then, from Lemma \ref{lem:id} and the H\"{o}lder inequality, we get a constant $C>0$ such that
\[
\begin{split}
\mu_{i,n}&\le \left|\int_{\rho_{i,n}}^{r_{i,n}}\la_nf(u_n)r\log{\frac{r_{i,n}}{r}}dr\right|\\
&\le C\left(\int_{\rho_{i,n}}^{r_{i,n}} u_{i,n}^4rdr\right)^{\frac14}\left(\int_{\rho_{i,n}}^{r_{i,n}} e^{4(1+\alpha)2\pi A_n^2\left(\frac{u_{i,n}}{\sqrt{2\pi}A_n}\right)^2} rdr\right)^{\frac14} \\
&\ \ \ \ \times\left(\int_{\rho_{i,n}}^{r_{i,n}}r\log^2{\frac{1}{r}}dr\right)^\frac12
\end{split}
\]
Then noting $A_n\to0$, we use the  Trudinger-Moser \eqref{TM} and Sobolev inequalities to obtain that the right hand side converges to zero. This finishes the proof.
\end{proof}

Let us begin our main discussion of this section. We study the limit equation of concentrating solutions. To this end, we fix a number $1\le i\le k+1$ and suppose
\begin{equation}\label{infty}
\mu_{i,n}\to \infty\text{ as }n\to \infty.
\end{equation}
Moreover, if $i\not=1$, we also assume that
\begin{align}
&\sup_{n\in \mathbb{N}}\int_{r_{i-1,n}}^{r_{i,n}}u_{i,n}'(r)^2rdr<\infty,\label{bdd}\\
&\lim_{n\to \infty}\frac{\log{\frac{1}{\la_nr_{i-1,n}^2}}}{\mu_{i-1,n}^{\beta_n}}=\alpha\left(1-\frac{\beta_*}2\right),\label{a4}
\end{align}
and
\begin{align}
&\lim_{n\to \infty}\mu_{i-1,n}r_{i-1,n} |u_{i-1,n}'(r_{i-1,n})|=2.\label{a7}
\end{align}
Our goal  is to prove the following.
\begin{proposition}\label{dd}
Assume \eqref{infty}-\eqref{a7}, put $\ga_{i,n}>0$ so that 
\[
1=2\la_n \mu_{i,n} f_n(\mu_{i,n})\ga_{i,n}^2,
\] 
and define 
\[
z_{i,n}(r):=2\mu_{i,n}(|u_{i,n}(\ga_{i,n} r+\rho_{i,n})|-\mu_{i,n})\ \ \left(r\in\left[\frac{r_{i-1,n}-\rho_{i,n}}{\ga_{i,n}},\frac{r_{i,n}-\rho_{i,n}}{\ga_{i,n}}\right]\right),
\]
for all $n\in \mathbb{N}$. Then we have $\ga_{i,n}\to0$, $(r_{i-1,n}-\rho_{i,n})/\ga_{i,n}\to0$, $(r_{i,n}-\rho_{i,n})/\ga_{i,n}\to\infty$, and further, $z_{i,n}\to z$ in $C^2_{\text{loc}}((0,\infty))\cap C^1_{\text{loc}}([0,\infty))$ where 
\begin{equation}\label{d00}
z(r)=\log{\frac{64}{(8+r^2)^2}}
\end{equation}
 which satisfies 
\begin{equation}\label{d01}
-z''-\frac{1}{r}z'=e^z\text{ in }(0,\infty),\ z(0)=0=z'(0) \text{ and }\int_0^\infty e^zrdr=4.
\end{equation}
\end{proposition}

Before staring the proof, note that $z_{i,n}$ satisfies
\begin{equation}\label{d1}
\begin{cases}
-z_{i,n}''-\frac1{r+\frac{\rho_{i,n}}{\ga_{i,n}}}z_{i,n}'=\left(\frac{z_{i,n}}{2\mu_{i,n}^2}+1\right)\\
\ \ \ \ \ \ \ \ \ \ \ \ \ \ \ \ \ \ \ \ \ \ \ \ \ \ \ \ \ \times  e^{z_{i,n}+\frac{z_{i,n}^2}{4\mu_{i,n}^2}+\alpha \mu_{i,n}^{\beta_{n}}\left\{\left(\frac{z_{i,n}}{2\mu_{i,n}^2}+1\right)^{\beta_n}-1\right\}},\\
z_{i,n}\le0,\ \ \ \ \ \ \ \ \ \ \ \ \ \ \ \ \ \ \ \ \ \ \ \ \ \ \ \ \ \ \ \ \ \  \ \ \ \ \ \hbox{ in }\left(\frac{r_{i-1,n}-\rho_{i,n}}{\ga_{i,n}},\frac{r_{i,n}-\rho_{i,n}}{\ga_{i,n}}\right),\\
z_{i,n}(0)=0=z_{i,n}'(0),\ z_{i,n}\left(\frac{r_{i,n}-\rho_{i,n}}{\ga_{i,n}}\right)=-2\mu_{i,n}^2,\\
z_{i,n}\left(\frac{r_{i-1,n}-\rho_{i,n}}{\ga_{i,n}}\right)=-2\mu_{i,n}^2\ (i\not=1).
\end{cases}
\end{equation}
Put $l:=\lim_{n\to \infty}(\rho_{i,n}-r_{i-1,n})/\ga_{i,n}$. Then as in the proof of Lemma 4.3 in \cite{GN}, the crucial step is to deduce $l=0$. Hence the case $i=1$ is easier. In the case $i>1$, we have to exclude the cases $l=\infty$ and $l\in(0,\infty)$. As a first step, we can prove by \eqref{bdd} that the case $l=\infty$ does not occur. 
\begin{lemma}\label{ap1} Assume \eqref{infty} and \eqref{bdd} and define $\ga_{i,n}$ and $z_{i,n}$ as in the previous proposition. Moreover, put $l,m\in[0,\infty]$ so that 
\[
l=\lim_{n\to \infty}\frac{\rho_{i,n}-r_{i-1,n}}{\ga_{i,n}}\text{ and }m=\lim_{n\to \infty}\frac{\rho_{i,n}}{\ga_{i,n}},
\] 
by subtracting a subsequence if necessary. Then we get that $\lim_{n\to \infty}\rho_{i,n}/r_{i,n}=0$, $\lim_{n\to \infty}(r_{i,n}-\rho_{i,n})/\ga_{i,n}=\infty$, $0\le l=m<\infty$ and further, 
 there exists a function $z$ such that $z_{i,n} \to z$ in $C^2_{\text{loc}}((-l,\infty))$ (in $C^2_{\text{loc}}((0,\infty))\cap C^1_{\text{loc}}([0,\infty))$ if $l=0$). 
\end{lemma}
\begin{proof} 
Using Lemmas \ref{RL}, \ref{newl1} and our assumptions \eqref{infty} and \eqref{bdd}, the proof is similar to that of  Lemma 4.3 in \cite{GN}. (Especially, see the argument in ``Case 1" there.) For the readers' convenience we show the proof in Appendix \ref{ap}.
\end{proof}

Now our final aim becomes to prove $l=0$. In order to prove this, the variational characterization of solutions by \cite{AY1} was useful in the previous work \cite{GN}. This allowed us to get the energy estimate in Lemma 2.1 (and also Lemma 2.5) in \cite{GN}. Using this, we could prove that the case $l\in(0,\infty)$ does not happen. (See the argument for ``Case 2" in the proof of Lemma 4.3 in \cite{GN}.) Since we only assume  the boundedness of the energy in this paper, we need a new argument. We accomplish the proof  with the aid of our new assumptions \eqref{a4} and \eqref{a7} as follows.
\begin{proof}[Proof of Proposition \ref{dd}] Without losing the generality, we may suppose $u_{i.n}\ge0$. Let $l,m$ and $z$ as in Lemma \ref{ap1}.  Then we get $l=m<\infty$ by the lemma. Let us prove $l=0$. If $i=1$, this is trivial. Hence, we suppose $i\ge2$ and $l>0$ on the contrary. Then, by \eqref{d1} and Lemma \ref{ap1}, the limit function $z$ satisfies 
\begin{equation*}
\begin{cases}
-z''-\frac1{l+r}z'=
e^{z},\ z\le0\ \ \hbox{ in }\left(-l,+\infty\right)\\
z(0)=0=z'(0).
\end{cases}
\end{equation*}
It follows  that
\begin{equation*}
z(r)=\log\frac{4A^2l^{A+2}(r+l)^{A-2}}
{\left((A+2)l^A+(A-2)(r+l)^A\right)^2},
\end{equation*}
where $A=\sqrt{2l^2+4}$. (See Proof of Proposition 3.1 in \cite{GGP} or the proof of Lemma 4.3 in \cite{GN}). Then, we use Lemma \ref{lem:id} to get 
\begin{equation}
\begin{split} 
&2\mu_{i,n}^2\\
&\ =2\mu_{i,n}\int_{r_{i-1,n}}^{\rho_{i,n}}\lambda_n f(u_{i,n})r\log{\frac{r}{r_{i-1,n}}}dr\\
&\ =\log{\frac{\ga_{i,n}}{r_{i-1,n}}}\int_{\frac{r_{i-1,n}-\rho_{i,n}}{\ga_{i,n}}}^{0}\left(\frac{z_{i,n}(r)}{2\mu_{i,n}^2}+1\right) 
\\&\ \ \ \ \ \ \ \ \ \ \ \ \  \ \ \ \ \times e^{z_{i,n}(r)+\frac{z_{i,n}^2(r)}{4\mu_{i,n}^2}+\alpha\mu_{i,n}^{\beta_n}\left\{\left(\frac{z_{i,n}(r)}{2\mu_{i,n}^2}+1\right)^{\beta_n}-1\right\}}\left(r+\frac{\rho_{i,n}}{\ga_{i,n}}\right)dr\\
&\ \ +\int_{\frac{r_{i-1,n}-\rho_{i,n}}{\ga_{i,n}}}^{0}\left(\frac{z_{i,n}(r)}{2\mu_{i,n}^2}+1\right) e^{z_{i,n}(r)+\frac{z_{i,n}^2(r)}{4\mu_{i,n}^2}+\alpha\mu_{i,n}^{\beta_n}\left\{\left(\frac{z_{i,n}(r)}{2\mu_{i,n}^2}+1\right)^{\beta_n}-1\right\}}
\\&\ \ \ \ \ \ \ \ \ \ \ \ \ \ \ \ \ \ \ \ \ \times \left(r+\frac{\rho_{i,n}}{\ga_{i,n}}\right)\log{\left(r+\frac{\rho_{i,n}}{\ga_{i,n}}\right)}dr.
\end{split}\label{d4}
\end{equation}
Noting that $l=m$ implies $r_{i-1,n}/\ga_{i,n}\to0$,  we apply the Lebesgue convergence theorem to obtain
\begin{equation}
\lim_{n\to \infty}\frac{2\mu_{i,n}^2 }{\log{\frac{\ga_{i,n}}{r_{i-1,n}}}}=\int_{-l}^{0}e^z(r+l)dr
=\sqrt{2l^2+4}-2.\label{d5}
\end{equation}
On the other hand, we have by the definition of $\ga_{i,n}$,
\begin{equation}
\begin{split}
\frac{2\mu_{i,n}^2}{\log{\frac{\ga_{i,n}}{r_{i-1,n}}}}&=\frac{4\mu_{i,n}^2}{2\log{\frac{1}{r_{i-1,n}}}-\log{2\la_n}-2\log{\mu_{i,n}}-\mu_{i,n}^2-\alpha \mu_{i,n}^{\beta_n}} \\
&=\frac{4}{\frac{\log{\frac{1}{\la_n r_{i-1,n}^2}}}{\mu_{i,n}^2}-1+o(1)}.
\end{split}\label{d6}
\end{equation}
Combining \eqref{d5} with \eqref{d6}, we get
\[
\lim_{n\to \infty} \frac{\log{\frac{1}{\la_n r_{i-1,n}^2}}}{\mu_{i,n}^2}=\frac{2\sqrt{2l^2+4}+4+l^2}{l^2}\in(0,\infty).
\]
Then since 
\[
\frac{\log{\frac{1}{\la r_{i-1,n}^2}}}{\mu_{i,n}^2}=\frac{\log{\frac{1}{\la r_{i-1,n}^2}}}{\mu_{i-1,n}^{\beta_n}}\left(\frac{\mu_{i-1,n}}{\mu_{i,n}}\right)^{\beta_n}\frac{1}{\mu_{i,n}^{2-\beta_n}},
\]
using our assumptions \eqref{infty} and \eqref{a4}, we deduce that 
\begin{equation}
\lim_{n\to \infty}\frac{\mu_{i,n}}{\mu_{i-1,n}}=0.\label{d7}
\end{equation}
On the other hand, since $u_n=u_n(r)$ satisfies $-u_n''-u_n'/r=\la_n f_n(u_n)$ on $(\rho_{i-1,n},\rho_{i,n})$, multiplying this equation by $r$ and integrating over $(\rho_{i-1,n},\rho_{i,n})$, we get
\[
\int_{\rho_{i-1,n}}^{r_{i-1,n}}\la f(u_{i-1,n})rdr=-\int_{r_{i-1,n}}^{\rho_{i,n}}\la f(u_{i,n})rdr.
\]
Then, it follows from \eqref{a7} and the similar scaling argument as in \eqref{d4} that 
\begin{equation}\label{d777}
\frac{\mu_{i,n}}{\mu_{i-1,n}}=-\frac{\mu_{i,n}\int_{r_{i-1,n}}^{\rho_{i,n}}\la f(u_{i,n})rdr}{\mu_{i-1,n}\int_{\rho_{i-1,n}}^{r_{i-1,n}}\la f(u_{i-1,n})rdr}\to \frac{\sqrt{2l^2+4}-2}{4}>0
\end{equation}
as $n\to \infty$. This contradicts \eqref{d7}. Hence we get $l=0$. Then \eqref{d1} and Lemma \ref{ap1} prove that $z$ satisfies  
\begin{equation*}
\begin{cases}
-z''-\frac1{r}z'=
e^{z},\ z\le0\ \ \hbox{ in }\left(0,+\infty\right)\\
z(0)=0=z'(0).
\end{cases}
\end{equation*}
After integration (see Proof of Proposition 3.1 in \cite{GGP}  or the proof of Lemma 4.3 in \cite{GN}), we conclude that $z$ satisfies \eqref{d00} and \eqref{d01}. We complete the proof.  
\end{proof}

In the proof above, we get the following.
\begin{lemma}\label{lem:d8} Assume \eqref{infty}-\eqref{a7}. Then we have
\begin{equation}
\lim_{n\to\infty}\frac{\rho_{i,n}}{\ga_{i,n}}=0,\label{d8}
\end{equation}
and
\begin{equation}
\lim_{n\to \infty}\frac{\mu_{i,n}}{\mu_{i-1,n}}=0.\label{d77}
\end{equation}
\end{lemma}
\begin{proof}
In the previous proof, we get $l=m=0$. This proves \eqref{d8}. Using $l=0$ in \eqref{d777}, we obtain \eqref{d77}. This completes the proof.
\end{proof}
\section{Limit energy}\label{energy}
In this section, we study the limit energy of concentrating solutions. As in the previous section, we fix $i=1,\cdots,k+1$ and suppose \eqref{infty}--\eqref{a7}. Without loss of the generality we assume $u_{i,n}\ge0$. Moreover, we define $\ga_{i,n}$ and $z_{i,n}$ as in the previous section. Our main goal is to prove the next asymptotic energy expansion.
\begin{proposition}\label{e00}
Assume \eqref{infty}-\eqref{a7}. Moreover, if $i\not=1$, we suppose $\beta_*<3/2$. Then we have 
\[
\int_{r_{i-1,n}}^{r_{i,n}} u_{i,n}'(r)^2rdr=2-\frac{\alpha \beta_*}{\mu_{i,n}^{2-\beta_n}}+o\left(\frac{1}{\mu_{i,n}^{2-\beta_n}}\right),
\]
and 
\[
\mu_{i,n}\int_{r_{i-1,n}}^{r_{i,n}} \la_n f_n(u_{i,n})rdr=2-\frac{\alpha \beta_*}{\mu_{i,n}^{2-\beta_n}}+o\left(\frac{1}{\mu_{i,n}^{2-\beta_n}}\right),
\]
as $n\to \infty$.
\end{proposition}

For the proof, we begin with the next lemma. 
\begin{lemma}\label{lem:e0} Let $i\ge2$ and suppose \eqref{infty}-\eqref{a7}. Then we get $r_{i-1,n}/\rho_{i,n}\to0$ and
\begin{equation}
\lim_{n\to\infty}\frac{\la_n \rho_{i,n}^2 f_n(\mu_{i,n})\log{\frac{\rho_{i,n}}{r_{i-1,n}}}}{\mu_{i,n}}=2=\lim_{n\to\infty}\frac{\la_n \rho_{i,n}^2 f_n(\mu_{i,n})}{r_{i-1,n} u_{i,n}'(r_{i-1,n})}.\label{e0}
\end{equation}
\end{lemma}
\begin{proof} We put $\tilde{r}_n:=r_{i-1,n}/\rho_{i,n}$. Moreover, we write  $\mu_n=\mu_{i,n}$ and $\rho_n=\rho_{i,n}$ for simplicity. 
We first claim that $\lim_{n\to \infty}\tilde{r}_n=0$. Actually, we get 
\[
\mu_n=\int_{r_{i-1,n}}^{\rho_n} u_n'(r)dr\le \left(\int_{r_{i-1,n}}^{\rho_n}|u_n'(r)|^2rdr\right)^{\frac12}\left(\log{\frac{1}{\tilde{r}_n}}\right)^{\frac12}.
\]
Then the claim follows by \eqref{infty} and \eqref{bdd}. Next we define a scaled function $\tilde{z}_n(r):=2\mu_n(u_n(\rho_nr)-\mu_n)$ for $r\in(\tilde{r}_n,1)$. Then it satisfies
\[
\begin{cases}
-\tilde{z}_n''-\frac{1}{r}\tilde{z}_n'=2\la_n \rho_n^2\mu_nf_n(\mu_n)\frac{f_n\left(\frac{\tilde{z}_n}{2\mu_n}+\mu_n\right)}{f_n(\mu_n)} 
\text{ in }(\tilde{r}_n,1),\\
\tilde{z}_n(1)=0=\tilde{z}_n'(1).
\end{cases}
\]
Thanks to \eqref{d8}, we get that 
\[
-\tilde{z}_n''-\frac{1}{r}\tilde{z}_n'=o(1)
\]
where $o(1)\to0$ uniformly in $(\tilde{r}_{n},1)$. Integrating this formula and using $\tilde{z}_n'(1)=0=\tilde{z}_n(1)$, we see $\tilde{z}_n\to0$ in $C^1_{\text{loc}}((0,1])$. Then similarly to \eqref{d4}, we use Lemma \ref{lem:id} to derive
\begin{equation*}
\begin{split}
\mu_n&=
\la_n \rho_n^2f_n(\mu_n)\int_{\tilde{r}_n}^{1}\frac{f_n\left(\frac{\tilde{z}_n}{2\mu_n}+\mu_n\right)}{f_n(\mu_n)}r\log{\frac{r}{\tilde{r}_n}}dr\\
&=\la_n \rho_n^2f_n(\mu_n)\log{\frac{1}{\tilde{r}_n}}\int_{\tilde{r}_n}^{1}\left(\frac{\tilde{z}_n}{2\mu_n^2}+1\right)e^{\tilde{z}_n+\frac{\tilde{z}_n^2}{4 \mu_n^2}+\alpha \mu_n^{\beta_n}\left\{\left(\frac{\tilde{z}_n}{2\mu_n^2}+1\right)^{\beta_n}-1\right\}}
rdr\\
&\ \ \ \ +\la_n \rho_n^2f_n(\mu_n)\int_{\tilde{r}_n}^{1}\left(\frac{\tilde{z}_n}{2\mu_n^2}+1\right)e^{\tilde{z}_n+\frac{\tilde{z}_n^2}{4 \mu_n^2}+\alpha \mu_n^\beta\left\{\left(\frac{\tilde{z}_n}{2\mu_n^2}+1\right)^{\beta_n}-1\right\}}
r\log{r}dr.
\end{split}
\end{equation*}
Therefore, it follows from the Lebesgue convergence theorem and our first claim  that
\[
\lim_{n\to \infty}\frac{\mu_n}{\la_n \rho_n^2f_n(\mu_n)\log{\frac{1}{\tilde{r}_n}}}=\frac12.
\]
This shows  the first equality of \eqref{e0}. Finally, since 
\[
\begin{split}
 r_{i-1,n} &u_n'(r_{i-1,n})\\
&=\int_{r_{i-1,n}}^{\rho_n} \la_n f_n(u_n)rdr\\
&=\la_n \rho_n^2f_n(\mu_n)\int_{\tilde{r}_n}^{1}\left(\frac{\tilde{z}_n}{2\mu_n^2}+1\right)e^{\tilde{z}_n+\frac{\tilde{z}_n^2}{4 \mu_n^2}+\alpha \mu_n^{\beta_n}\left\{\left(\frac{\tilde{z}_n}{2\mu_n^2}+1\right)^{\beta_n}-1\right\}}rdr,
\end{split}
\] 
we similarly get the second one. This completes the proof.
\end{proof}

By the previous lemma, \eqref{a4} and \eqref{a7}, we get the following.
\begin{lemma}\label{lem:e}Let $i\ge2$ and assume \eqref{infty}-\eqref{a7}. Then we have  that
\begin{equation}
\limsup_{n\to \infty}\mu_{i,n}^{2-\beta_n}\left(\frac{\rho_{i,n}}{\ga_{i,n}}\right)^{2(\beta_n-1)}= 8^{\beta_*-1}\alpha\left(1-\frac{\beta_*}{2}\right),\label{e1}
\end{equation}
and 
\begin{equation}\label{f11}
\lim_{n\to \infty}\frac{\mu_{i,n}}{\mu_{i-1,n}^{\beta_n-1}}=\alpha\left(1-\frac{\beta_*}{2}\right).
\end{equation}
In particular,  if $\beta_*<3/2$, we get
\begin{equation}
\lim_{n\to \infty}\mu_{i,n}^{2-\beta_n}\frac{\rho_{i,n}}{\ga_{i,n}}=0.\label{e2}
\end{equation}
\end{lemma}
\begin{proof}
From the first equality in \eqref{e0}, we get
\begin{equation}
\frac{\la_n \rho_{i,n}^2f_n(\mu_{i,n})}{\mu_{i,n}}\left(\log{(\la_n\rho_{i,n}^2)}+\log{\frac1{\la_n r_{i-1,n}^2}}\right)=4+o(1).\label{e3}
\end{equation}
On the other hand, using \eqref{a7} and \eqref{a4} for the second equality in \eqref{e0}  implies
\[
\begin{split}
2+o(1)&=\frac{\la_n \rho_{i,n}^2f_n(\mu_{i,n})\mu_{i-1,n}}{2+o(1)}\\
&=\frac{\la_n \rho_{i,n}^2f_n(\mu_{i,n})}{2+o(1)}\left\{\frac{\log{\frac{1}{\la_n r_{i-1,n}^2}}}{\alpha\left(1-\frac{\beta_*}{2}\right)+o(1)}\right\}^{\frac1{\beta_n}}.
\end{split}
\]
It follows from the first equality of this formula that
\begin{equation}\label{e3-1}
\la_n \rho_{i,n}^2=\frac{4+o(1)}{f_n(\mu_{i,n})\mu_{i-1,n}},
\end{equation}
and from the second one that
\begin{equation}\label{e3-2}
\log{\frac1{\la_n r_{i-1,n}^2}}=\frac{4^{\beta_*}\alpha\left(1-\frac{\beta_*}{2}\right)+o(1)}{\left\{\la_n \rho_{i,n}^2 f_n(\mu_{i,n})\right\}^{\beta_n}}.
\end{equation}
We substitute \eqref{e3-2} into \eqref{e3} and get
\[
\begin{split}
4+o(1)&=\frac{\la_n \rho_{i,n}^2f(\mu_{i,n})}{\mu_{i,n}}\left[\log{(\la_n\rho_{i,n}^2)}+\frac{4^{\beta_*}\alpha\left(1-\frac{\beta_n}{2}\right)+o(1)}{\left\{\la_n \rho_{i,n}^2 f_n(\mu_{i,n})\right\}^{\beta_n}}\right]\\
&=\frac{4+o(1)}{\mu_{i,n}\mu_{i-1,n}}\log{\frac{4+o(1)}{f_n(\mu_{i,n})\mu_{i-1,n}}}+\frac{4^{\beta_*}\alpha\left(1-\frac{\beta_*}{2}\right)+o(1)}{\mu_{i,n}\left\{\la_n \rho_{i,n}^2 f_n(\mu_{i,n})\right\}^{\beta_n-1}}
\end{split}
\]
where we used \eqref{e3-1} for the second equality. Notice that \eqref{d77} implies that the first term on the right hand side converges to zero as $n\to \infty$.  Consequently, by the definition of $\ga_{i,n}$, we obtain
\[
\mu_{i,n}^{2-\beta_n}\left(\frac{\rho_{i,n}}{\ga_{i,n}}\right)^{2(\beta_n-1)}= 8^{\beta_*-1}\alpha\left(1-\frac{\beta_*}{2}\right)+o(1).\]
This proves \eqref{e1}. Furthermore, substituting the definition of $\ga_{i,n}$ and \eqref{e3-1} into this formula, we see
\[
\mu_{i,n}^{2-\beta_n}\left(\frac{(8+o(1))\mu_{i,n}}{\mu_{i-1,n}}\right)^{\beta_n-1}= 8^{\beta_*-1}\alpha\left(1-\frac{\beta_*}{2}\right)+o(1).
\]
This ensures \eqref{f11}. Finally, if $\beta_*\in(0,3/2)$, \eqref{e1} and \eqref{d8} show that 
\[
\mu_{i,n}^{2-\beta_n}\frac{\rho_{i,n}}{\ga_{i,n}}= \left\{8^{\beta_*-1}\alpha\left(1-\frac{\beta_*}{2}\right) +o(1)\right\}\left(\frac{\rho_{i,n}}{\ga_{i,n}}\right)^{3-2\beta_n}\to0
\]
as $n\to \infty$. We finish the proof.
\end{proof}
\begin{remark}\label{rmk:e1} By \eqref{e1} and \eqref{d8}, we see that if $(\beta_n)\subset(0,1]$, then $\mu_{i,n}$ is uniformly bounded for all $n\in \mathbb{N}$. This contradicts our basic assumption \eqref{infty}. This will prove that, interestingly, if $\int_0^1u_n'(r)^2rdr$ is uniformly bounded and $\beta_n\le1$ for all $n\in \mathbb{N}$, then $u_{i,n}$ does never blow up for any $i\ge2$. For the detail, see Proof of Theorem \ref{a31} in Section \ref{proof}. This remark suggests that, in the rest of the argument in this section, we may restrict our attention only on the case $i=1$ if $(\beta_n)\subset(0,1]$.
\end{remark}
\begin{remark}\label{rmk:e2}
As in the statement, the assumption $\beta_*<3/2$ ensures \eqref{e2}. We will see that the next lemmas strongly depend on this fact. For example, it allows the assertion $\e_n=o(\mu_{i,n}^{2-\beta_n})$ in Lemma \ref{le70} below. On the other hand, if $i\not=1$ and $(\beta_n)\subset [3/2,2)$, this assertion fails by \eqref{e1} above. This implies that the effect of the error term $\rho_{i,n}/\ga_{i,n}$ would appear in the strong point wise estimate like \eqref{e70}. More precisely, the term $-\alpha \beta_*/(2\mu_{i,n}^{2-\beta_*})$ in \eqref{e70} would be modified to the one with $\mu_{i,n}^{-(2-\beta_n)/(2(\beta_n-1))}$ in view of  \eqref{e1}. This change would affect all the results, for example, the energy expansion in Theorem \ref{a10} and the asymptotic formulas in Theorem \ref{a31}, based on \eqref{e70}. 
\end{remark}

Notice that in the following lemmas, we additionally assume $\beta_*<3/2$ if $i\not=1$. We next prove the following.
\begin{lemma}\label{e333}
Assume $i\ge1$ and \eqref{infty}-\eqref{a7}. Moreover, we suppose $\beta_*<3/2$ if $i\not=1$. Let $z_{i,n},z$ be functions defined in Proposition \ref{dd} and put $\phi_n:=\mu_{i,n}^{2-\beta_n}(z_{i,n}-z)$. Then we get $\phi_n\to \phi$ in $C^2_{\text{loc}}(0,\infty)\cap C^0_{\text{loc}}([0,\infty))$ where $\phi$ satisfies 
\begin{equation}
-\phi''-\frac{1}{r}\phi'=e^{z}\left( \phi+\frac{\alpha \beta_*}{2}z\right)\text{ in }(0,\infty),\ \phi(0)=0.\label{e30}
\end{equation}
In particular, we obtain
\begin{equation}
\begin{split}
\phi(r)
=\alpha\beta_*\left(\log{(8+r^2)}+\frac{8}{8+r^2}-1-\log{8}\right)\ (r\in[0,\infty)).
\end{split}\label{e31}
\end{equation}
\end{lemma}
\begin{proof} We write  $\rho_n=\rho_{i,n}$, $\mu_n=\mu_{i,n}$, $\ga_n=\ga_{i,n}$ and $z_n=z_{i,n}$ for simplicity. Then using the definition of $\phi_n$ and the equations in \eqref{d1} and \eqref{d01}, we get 
\begin{equation}\label{eee}
\begin{split}
&\displaystyle -\phi_n''-\frac{1}{r+\frac{\rho_{n}}{\ga_{n}}}\phi_n'\\
&=\mu_{n}^{2-\beta_n}\Bigg\{\left(\frac{z_n}{2\mu_n^2}+1\right)e^{z_n+\frac{z_n^2}{4\mu_n^2}+\alpha \mu_n^{\beta_n}\left(\left(\frac{z_n}{2\mu_n^2}+1\right)^{\beta_n}-1\right)}\\
&\ \ \ \ \ \ \ \ \ \ \ \ \ \ \ \ \ \ \ \ \ \ \ \ \ \ \ \ \ \ \ \ \ \ \ \ \ \ \ \ \ \ \ \ \ \ \ \ -e^z+\left(\frac{1}{r+\frac{\rho_{n}}{\ga_{n}}}-\frac{1}{r}\right)z'\Bigg\},
\end{split}
\end{equation}
for all $r\in\left(0,\frac{r_{i,n}-\rho_{n}}{\ga_{n}}\right)$. Here recalling that $z_n$ is  locally uniformly bounded in $[0,\infty)$, we use the Taylor theorem to see 
\[
\left(\frac{z_n}{2\mu_n^2}+1\right)e^{z_n+\frac{z_n^2}{4\mu_n^2}+\alpha \mu_n^{\beta_n}\left\{\left(\frac{z_n}{2\mu_n^2}+1\right)^{\beta_n}-1\right\}}=e^{z_n}+\frac{\alpha \beta_n }{2\mu_n^{2-\beta_n}}z_ne^{z_n}+o\left(\frac{1}{\mu_n^{2-\beta_n}}\right)
\]
where $\mu_n^{2-\beta_n}\cdot o(1/\mu_n^{2-\beta_n})\to0$ locally uniformly in $[0,\infty)$. Then after substituting this into \eqref{eee}, for all $r\in\left(0,\frac{r_{i,n}-\rho_{n}}{\ga_{n}}\right)$ and $n\in \mathbb{N}$, we apply the mean value theorem to obtain a constant $\theta=\theta (n,r)\in(0,1)$ such that 
\begin{equation}
\begin{cases}\displaystyle -\phi_n''(r)-\frac{1}{r+\frac{\rho_{n}}{\ga_{n}}}\phi_n'(r)=e^{z(r)+\theta(z_n(r)-z(r))}\phi_n(r)\\
\ \ \ \ \ \ \ \ \ \ \ \ \ \ \ \ \ \ \ \ \ \ \ \ \ \ \ \ \ \ \ \ \  +\frac{\alpha \beta_n}{2}z_n(r) e^{z_n(r)}\\
\ \ \ \ \ \ \ \ \ \ \ \ \ \ \ \ \ \ \ \ \ \ \ \ \ \ \ \ \ \ \ \ \ +\mu_{n}^{2-\beta_n}\left(\frac{1}{r+\frac{\rho_{n}}{\ga_{n}}}-\frac{1}{r}\right)z'(r)+o(1),\\
\phi_n(0)=0=\phi_n'(0),
\end{cases}\label{e4}
\end{equation}
where $o(1)\to 0$ uniformly locally in $[0,\infty)$. Here, notice that third term of the right hand side of the equation is nontrivial if $i\ge2$. But thanks to  \eqref{e2}, we have 
\begin{equation}\label{e5}
\mu_{n}^{2-\beta_n}\left(\frac{1}{r+\frac{\rho_{n}}{\ga_{n}}}-\frac{1}{r}\right)z'(r)=\mu_{n}^{2-\beta_n}\frac{\rho_{n}}{\ga_{n}}\frac{4}{(8+r^2)\left(r+\frac{\rho_{n}}{\ga_{n}}\right)}\to0
\end{equation}
uniformly locally in $(0,\infty)$ as $n\to \infty$. Now we claim that $\phi_n$ is locally uniformly bounded in $[0,\infty).$ If not, there exist a constant $R>0$ and a sequence $(\xi_n)\subset [0,R]$ such that $\phi_n(\xi_n)=\max_{r\in[0,R]}|\phi_n(r)|$ and $\lim_{n\to \infty}|\phi_n(\xi_n)|=\infty$. Then putting $\tilde{\phi}_n:=\phi_n/\phi_n(\xi_n)$ and multiplying the equation in \eqref{e4} by $\phi(\xi_n)^{-1}$, we obtain that for all $r\in(0,R]$ 
\begin{equation}\label{e44}
\begin{cases}\displaystyle-\tilde{\phi}_n''(r)-\frac{1}{r+\frac{\rho_{n}}{\ga_{n}}}\tilde{\phi}_n'(r)=e^{z(r)+\theta(z_n(r)-z(r))}\tilde{\phi}_n(r)\\
\ \ \ \  \ \ \ \ \  \ \ \ \ \ \ \ \ \ \ \ \ \ \ \ \ \ \ \ \ \ \ \ \displaystyle +\frac{\mu_{n}^{2-\beta_n}}{\phi(\xi_n)}\frac{\rho_{n}}{\ga_{n}}\frac{4}{(8+r^2)\left(r+\frac{\rho_{n}}{\ga_{n}}\right)}
+o(1),\\
\tilde{\phi}_n(0)=0=\tilde{\phi}_n'(0),
\end{cases}
\end{equation}
where $o(1)\to0$ uniformly in $[0,R]$. It follows that
\[
-\tilde{\phi}_n''(r)-\frac{1}{r+\frac{\rho_{n}}{\ga_{n}}}\tilde{\phi}_n'(r)=O(1) +\frac{\mu_{n}^{2-\beta_n}}{\phi(\xi_n)}\frac{\rho_{n}}{\ga_{n}}\frac{4}{(8+r^2)\left(r+\frac{\rho_{n}}{\ga_{n}}\right)}\ (r\in(0,R]),
\]
where $O(1)$ is uniformly bounded in $[0,R]$. Then, for any $r\in(0,R]$, multiplying this formula by $(r+\rho_n/\ga_n)$ and integrating over  $(0,r)$ give
\begin{equation}\label{c4}
-\tilde{\phi}_n'(r)=O\left(\frac{\frac12r^2+\frac{\rho_{n}}{\ga_{n}}r}{r+\frac{\rho_{n}}{\ga_{n}}}\right)+\frac{\mu_{n}^{2-\beta_n}}{\phi_n(\xi_n)}\frac{\rho_{n}}{\ga_{n}}\left(r+\frac{\rho_n}{\ga_n}\right)^{-1}\int_0^r\frac{4}{(8+r^2)}dr.
\end{equation}
Consequently, with the aid of \eqref{e2}, we have that $\tilde{\phi}_n$ is uniformly bounded in $C^1([0,R])$. Therefore, it follows from the Ascoli-Arzel\`{a} theorem, the equation in \eqref{e44} and \eqref{e5} that there exists a function $\tilde{\phi}$ such that $\tilde{\phi}_n\to \tilde{\phi}$ in $C^2_{\text{loc}}((0,R])\cap C([0,R])$. Then $\tilde{\phi}$ satisfies 
\[
-\tilde{\phi}''-\frac{1}{r}\tilde{\phi}'=e^z\tilde{\phi} \text{ in }(0,R],\ \tilde{\phi}(0)=0.
\]
This implies
\[
\tilde{\phi}(r)=\tilde{c}_1\frac{8-r^2}{8+r^2}+\tilde{c}_2\frac{(8-r^2)\log{r}+16}{8+r^2}\ (r\in(0,R])
\]
for some constants $\tilde{c}_1,\tilde{c}_2\in \mathbb{R}$. Since $\lim_{r\to0+0}\tilde{\phi}(r)=0$, we get $\tilde{c}_1=\tilde{c}_2=0$ and thus, obtain $\tilde{\phi}=0$ in $[0,R]$. This contradicts the fact that $\tilde{\phi}_n(\xi_n)=1$ for all $n\in \mathbb{N}$. 
This proves the claim. Then, for any $r\in(0,\infty)$, multiplying the equation in \eqref{e4} by $(r+\rho_n/\ga_n)$ and integrating over $(0,r)$, a similar calculation as above gives that 
$\phi_n$ is uniformly bounded in $C_{\text{loc}}^1([0,\infty))$ thanks to \eqref{e2}. Then again by the Ascoli--Arzel\`{a} theorem, \eqref{e4} and \eqref{e5}, we obtain a function $\phi$ such that $\phi_n\to \phi$ in $C^2_{\text{loc}}((0,\infty))\cap C^0_{\text{loc}}([0,\infty))$. It follows from \eqref{e4} that
\begin{equation}\label{e6}
-\phi''-\frac{1}{r}\phi'=e^{z} \phi+\frac{\alpha \beta_*}{2}ze^{z}\text{ in }(0,\infty),\ \phi(0)=0.
\end{equation}
Then we compute that for all $r>0$,
\begin{equation}
\begin{split}
\frac{1}{\alpha \beta_*}\phi(r)&=c_1\frac{8-r^2}{8+r^2}+c_2\frac{(8-r^2)\log{r}+16}{8+r^2}\\&
\ \ +\log{(8+r^2)}+\frac{2(8-r^2)\log{r}+8(3-2\log{8})}{8+r^2},\\
\end{split}\label{e7}
\end{equation}
where $c_1,c_2\in \mathbb{R}$ are some constants. By $\lim_{r\to 0+0}\phi(r)=0$, we get $c_2=-2$ and then conclude $c_1=1+\log{8}$. This completes the proof.
\end{proof}

The next estimate is important. 
\begin{lemma}\label{le70} Let $i\ge1$ and assume \eqref{infty}-\eqref{a7}. In addition, if $i\not=1$, we suppose $\beta_*<3/2$. Then for any $R>0$, there exists a constant $C>0$ such that 
\begin{equation}\label{e70}
z_{i,n}(r)\le \left(1-\frac{\alpha\beta_*}{2\mu_{i,n}^{2-\beta_n}}\right)z(r)+C \e_n\log{r}
\end{equation}
for all $r\in[R,(r_{i,n}-\rho_{i,n})/\ga_{i,n}]$ and all large $n\in \mathbb{N}$ where we put
\[
\e_n:=\max\left\{\frac{1}{\mu_{i,n}^{4-2\beta_n}},\ \frac{1}{\mu_n^2},\ \frac{|\beta_n-\beta_*|}{\mu_n^{2-\beta_n}},\ \frac{\rho_{i,n}}{\ga_{i,n}}\right\}=o\left(\frac{1}{\mu_{i,n}^{2-\beta_n}}\right).\]
\end{lemma}
\begin{proof} We put $\rho_n=\rho_{i,n}$, $\mu_n=\mu_{i,n}$, $\ga_{n}=\ga_{i,n}$, and $z_n=z_{i,n}$ for simplicity. We apply the contraction mapping argument  in the proof of Lemma 5 in  \cite{MM1} (see also \cite{MM2}). We define a function $\psi_n$ on $[0,(r_{i,n}-\rho_n)/\ga_n]$ by
\[
\psi_n:=z_n-z-\frac{\phi}{\mu_n^{2-\beta_n}},
\]
where $\phi$ is taken from the previous lemma. Then from \eqref{d1}, \eqref{d01}, and \eqref{e30}, we get
\begin{equation}\label{eq:psi}
\begin{split}
&-\psi_n''-\frac{1}{r+\frac{\rho_n}{\ga_n}}\psi_n'\\
&=\left(\frac{z_n}{2\mu_n^2}+1\right)e^{z_n+\frac{z_n^2}{4\mu_n^2}+\mu_n^{\beta_n}\left\{\left(\frac{z_n}{2\mu_n^2}+1\right)^{\beta_n}-1\right\}}\\
&\ \ \ \ -e^z-\frac{1}{\mu_n^{2-\beta_n}}e^z\left(\phi+\frac{\alpha\beta_*}{2}z\right)+\left(\frac{1}{r+\frac{\rho_n}{\ga_n}}-\frac1r\right)\left(z'+\frac{\phi'}{\mu_n^{2-\beta_n}}\right)\\
&=\Phi_n(\psi_n)
\end{split}
\end{equation}
where we defined
\begin{equation}
\begin{split}
\Phi_n(\psi):=&e^z\Bigg[\left\{1+\frac{1}{2\mu_n^2}\left(z+\frac{\phi
}{\mu_n^{2-\beta_n}}+\psi\right)\right\}e^{h_n(\psi)}
\\
&\ \ \ \ \ \ \ \ \ \ \ \ \ \ \ \ \ \ \ \ \ \ \ \ \ \ \ \ \ \ -1-\frac{1}{\mu_n^{2-\beta_n}}\left(\phi+\frac{\alpha\beta_*}{2}z\right)\Bigg]\\
&+\frac{\rho_n}{\ga_n}\frac{4}{(r+\frac{\rho_n}{\ga_n})(r^2+8)}\left(1-\frac{\alpha\beta_*}{2\mu_n^{2-\beta_n}}\frac{r^2}{r^2+8}\right)
\end{split}\label{phi}
\end{equation} 
with
\begin{equation}\label{hn}
\begin{split}
h_n(\psi)&:=\psi+\frac{\phi}{\mu_n^{2-\beta_n}}+\frac{1}{4\mu_n^{2}}\left(z+\frac{\phi}{\mu_n^{2-\beta_n}}+\psi\right)^2\\&+\alpha\mu_n^{\beta_n}\left[\left\{\frac{1}{2\mu_n^{2}}\left(z+\frac{\phi}{\mu_n^{2-\beta_n}}+\psi\right)+1\right\}^{\beta_n}-1\right].
\end{split}
\end{equation}
 We first claim that for any $T>0$, there exists a constant $C(T)>0$ such that 
\begin{equation}\label{e80}
|\psi_n(r)|\le C(T)\e_n \text{ and }|\psi_n'(r)|\le C(T)\e_n\ (r\in[0,T]),
\end{equation}
for all large $n\in \mathbb{N}$ where $\e_n$ is defined as in the statement of this lemma and $\e_n=o\left(\mu_n^{-(2-\beta_n)}\right)$ by \eqref{e2}. To show the claim, fix any $T>0$. Then since $\psi_n\to0$ uniformly in $[0,T]$ by Lemma \ref{e333}, we have $|h_n(\psi_n)|\le1$ for large $n\in \mathbb{N}$. It follows that $e^{h_n(\psi_n)}=1+h_n(\psi_n)+O(h_n(\psi_n)^2)$ on $[0,T]$ by the Taylor theorem. Using this for \eqref{phi} with $\psi=\psi_n$, we compute that
\begin{equation}\label{ef}
\begin{split}
\Phi_n(\psi_n)
&=e^{z}\left[\left(1+O\left(\frac{1}{\mu_n^{2-\beta_n}}\right)\right)\psi_n+O\left(\psi_n^2\right)+O(\e_n)\right]\\
&+O\left(\frac{\rho_n}{\ga_n}\frac{1}{\left(r+\frac{\rho_n}{\ga_n}\right)\left(8+r^2\right)}\right)\text{ on } [0,T].
\end{split}
\end{equation}
Then, putting $\psi_n=\e_n\bar{\psi}_n$ in \eqref{eq:psi}, we get
\[
\begin{split}
-\bar{\psi}_n''-\frac{1}{r+\frac{\rho_n}{\ga_n}}\bar{\psi}_n'&=
e^z\left(1+O\left(\frac{1}{\mu_n^{2-\beta_n}}\right)\right)\bar{\psi}_n+O\left(\e_n\bar{\psi}_n^2\right)+O\left(1\right)\\
&\ \ \ \ +O\left(\frac{1}{\e_n}\frac{\rho_n}{\ga_n}\frac{1}{\left(r+\frac{\rho_n}{\ga_n}\right)\left(8+r^2\right)}\right) \text{ on } [0,T],
\end{split}
\]
and $\bar{\psi}_n(0)=0=\bar{\psi}'(0)$. Using this equation and noting $\e_n^{-1}\rho_n/\ga_n\le1$, we get that $\bar{\psi}_n$ is uniformly bounded in $C^1([0,T])$. (The detail of the proof is similar to the argument in confirming the locally uniformly boundedness of $\phi_n$ and $\phi_n'$ in the proof of Lemma \ref{e333}.) This ensures the claim. 
 
Next let us extend the estimate \eqref{e80} to a suitable expanding interval. To this end, we choose a sufficiently large number $T>0$ and a sequence $(s_n)\subset (T,(r_{i,n}-\rho_n)/\ga_n)$ so that $s_n\to \infty$. (More precise choice of $T$ and $(s_n)$ is given  later.) Then we consider an initial value problem
\[
-\psi''-\frac{1}{r+\frac{\rho_n}{\ga_n}}\psi'=\Phi_n(\psi)\text{ on }(T,s_n],\ \psi(T)=\psi_n(T),\ \psi'(T)=\psi_n'(T).
\]
Putting $\omega=\left(r+\frac{\rho_n}{\ga_n}\right)\psi'$, we get the 
 equivalent system,
\begin{equation}
\begin{cases}
\psi'=\frac{\omega}{\left(r+\frac{\rho_n}{\ga_n}\right)}\text{ on }(T,s_n],\\
\omega'=-\left(r+\frac{\rho_n}{\ga_n}\right)\Phi_n(\psi)\text{ on }(T,s_n],\\ \psi(T)=\psi_n(T),\ \omega(T)=\left(T+\frac{\rho_n}{\ga_n}\right)\psi_n'(T).
\end{cases}\label{e800}
\end{equation}
Notice that, by the uniqueness, the solution $(\psi,\omega)$ satisfies $(\psi,w)=(\psi_n,(r+\rho_n/\ga_n)\psi_n')$ on $[T,s_n]$. In order to construct the solution with an appropriate estimate, we reduce \eqref{e800} into an integral equation on a suitable function space. To do this, fix a constant $\tilde{C}>0$ such that
\begin{equation}\label{tildeC}
\tilde{C}\ge 4(C(T)(T+2)+1),
\end{equation}
where $C(T)>0$ is taken from \eqref{e80}. Then we define a complete metric space 
\[
\begin{split}
\mathcal{B}_{\tilde{C}}=&\Big\{(\psi,\omega)\in C^0([T,s_n])\times C^0([T,s_n])\ \big|\ 
\|\psi-\psi_n(T)\|_1\le \tilde{C} \e_n,\\&\ \ \ \ \|\omega\|_2\le \tilde{C}\e_n,\ 
                \psi(T)=\psi_n(T),\ \omega(T)=\left(T+\frac{\rho_n}{\ga_n}\right)\psi_n'(T)\Big\},
\end{split}
\]
with the norms $\|\cdot\|_1$ and $\|\cdot\|_2$ defined by 
\[
\|f\|_1=\sup_{r\in(T,s_n]}\left|\frac{f(r)}{\log{\left(r+\frac{\rho_n}{\ga_n}\right)}-\log{\left(T+\frac{\rho_n}{\ga_n}\right)}}\right|,\ \ \|f\|_2=2\sup_{r\in[T,s_n]}|f(r)|.
\]
Moreover, we set a map $\mathcal{F}:\mathcal{B}_{\tilde{C}}\to C^0([T,s_n])\times C^0([T,s_n])$ so that
$\mathcal{F}(\psi,\omega)=(F_{1}(\psi,\omega),F_{2}(\psi,\omega))$ and
\[
\begin{split}
&F_{1}(\psi,\omega)(r):=\psi(T)+\int_T^{r}\frac{\omega}{\left(s+\frac{\rho_n}{\ga_n}\right)}ds,\\
&F_{2}(\psi,\omega)(r):=\omega(T)-\int_T^{r}\left(s+\frac{\rho_n}{\ga_n}\right)\Phi_n(\psi)ds,
\end{split}
\]
for $r\in[T,s_n]$. We shall find a fixed point of $\mathcal{F}$ in $\mathcal{B}_{\tilde{C}}$. 

To this end, we fix a small number $0<\delta\ll1$ (independently of $T$) and choose the sequence $(s_n)$ so that $s_n\le \sqrt{e^{\delta \mu_n^{\min\{1,2-\beta_n\}}}-8}$  for all $n\in \mathbb{N}$ and $\liminf_{n\to \infty}\e_n^{1/2}s_n>0$. It follows that 
\begin{equation}\label{e81}
\frac{1}{\mu_n^{\min\{1,2-\beta_n\}}}\log{(8+r^2)}\le \delta
\end{equation}
for all $r\in [T,s_n]$ and $n\in \mathbb{N}$. Moreover, by \eqref{e80} and \eqref{e81}, there exists a number $n_0=n_0(T,\delta)$ such that if $n\ge n_0$, it holds that
\begin{equation}\label{e810}
\sup_{s\in[T,s_n]}|\psi(s)|\le \delta\ 
\end{equation}
for any $\psi \in C^0([T,S])$ with $\|\psi-\psi_n(T)\|_1\le \tilde{C}\e_n$. Then similarly to \eqref{ef}, we calculate by \eqref{d00}, \eqref{e31}, \eqref{e81} and \eqref{e810} that 
\begin{equation}
\begin{split}
\Phi_n(\psi)
&=e^z\left[\left(1+O\left(\delta
\right)\right)\psi+O\left(\e_n\log^2{(8+r^2)}\right)
\right]\\
&+O\left(\frac{\rho_n}{\ga_n}\frac{1}{(r+\frac{\rho_n}{\ga_n})(8+r^2)}\right)\text{  on }[T,s_n]
\end{split}\label{e811}
\end{equation}
for any $\psi \in C^0([T,S])$ such that $\|\psi-\psi_n(T)\|_1\le \tilde{C}\e_n$ and all $n\ge n_0$. Analogously, we compute that 
\begin{equation}\label{e812}
|\Phi_n(\psi)-\Phi_n(\bar{\psi})|\le \left(1+O\left(\delta\right)\right)e^z|\psi-\bar{\psi}|\text{ on }[T,s_n],
\end{equation}
for all $\psi, \bar{\psi} \in C^0([T,S])$ verifying $\|\psi-\psi_n(T)\|_1\le \tilde{C}\e_n$ and $\|\bar{\psi}-\psi_n(T)\|_1\le \tilde{C}\e_n$ and all $n\ge n_0$. After this we always assume $n\ge n_0$. 

Now, we first claim that $\mathcal{F}: \mathcal{B}_{\tilde{C}}\to \mathcal{B}_{\tilde{C}}$. In fact, for any $(\psi,\omega)\in \mathcal{B}_{\tilde{C}}$, we get
\[
\begin{split}
|F_{1}(\psi,\omega)(r)-\psi_n(T)|&\le\frac12\|\omega\|_2\left|\int_T^r\frac{1}{s+\frac{\rho_n}{\ga_n}}ds\right|.
\end{split}
\]
It follows that
\[
\|F_{1}(\psi,\omega)-\psi_n(T)\|_1\le \frac12\tilde{C}\e_n.
\] 
On the other hand, from our choice of $(\psi,\omega)$, \eqref{e80} and \eqref{e811}, we have for any $r\in[T,s_n]$ that
\[
\begin{split}
&|F_{2}(\psi,\omega)(r)|\\
&\le \e_n\Bigg[C(T)\left(T+\frac{\rho_n}{\ga_n}\right)+65\int_T^r\frac{\left(s+\frac{\rho_n}{\ga_n}\right)\left(C(T)+\tilde{C}\log{\frac{r+\frac{\rho_n}{\ga_n}}{T+\frac{\rho_n}{\ga_n}}}\right)}{(8+s^2)^2}ds\\
&\ \ \ \ \ \ \ \ \ +O\left(\int_T^r\frac{\left(s+\frac{\rho_n}{\ga_n}\right)\log^2(8+s^2)}{(8+s^2)^2}ds\right)
+O\left(\e_n^{-1}\frac{\rho_n}{\ga_n}\int_T^r\frac{1}{8+s^2}ds\right)\Bigg]\\
&\le \e_n\Bigg[C(T)\left(T+1+65\int_T^r\frac{\left(s+\frac{\rho_n}{\ga_n}\right)}{(8+s^2)^2}ds\right)+65\tilde{C}\int_T^r\frac{\left(s+\frac{\rho_n}{\ga_n}\right)\log{\frac{r+\frac{\rho_n}{\ga_n}}{T+\frac{\rho_n}{\ga_n}}}}{(8+s^2)^2}ds\\
&\ \ \ \ \ \ \ \ \ +O\left(\int_T^r\frac{\left(s+\frac{\rho_n}{\ga_n}\right)\log^2(8+s^2)}{(8+s^2)^2}ds\right)
+O\left(\e_n^{-1}\frac{\rho_n}{\ga_n}\int_T^r\frac{1}{8+s^2}ds\right)\Bigg].
\end{split}
\]
Since $\e_n^{-1}\rho_n/\ga_n\le1$, taking $T>0$ large enough, we get
\[
|F_{2}(\psi,\omega)(r)|\le \e_n\left[C(T)(T+2)+\frac{\tilde{C}}{4}+1\right].
\]
We fix this $T$. Then it follows from \eqref{tildeC} that
\[
\|F_{2}(\psi,\omega)\|_2\le \tilde{C}\e_n.
\]
This proves the claim. Next we shall show that $\mathcal{F}$ is a contraction mapping. Indeed, for any $(\psi,\omega),(\bar{\psi},\bar{\omega})\in \mathcal{B}_{\tilde{C}}$, we obtain 
\[
\begin{split}
|F_{1}(\psi,\omega)(r)-F_{1}(\bar{\psi},\bar{\omega})(r)|
&\le \frac12 \|\omega-\bar{\omega}\|_2\log{\frac{r+\frac{\rho_n}{\ga_n}}{T+\frac{\rho_n}{\ga_n}}}.
\end{split}
\]
This implies that 
\[
\|F_{1}(\psi,\omega)-F_{1}(\bar{\psi},\bar{\omega})\|_1 \le \frac12\|\omega-\bar{\omega}\|_2.
\]
Moreover, we get by \eqref{e812} that 
\[
\begin{split}
|F_{2}(\psi,\omega)(r)-&F_{2}(\bar{\psi},\bar{\omega})(r)|\\
&\le 65\int_T^r\left(s+\frac{\rho_n}{\ga_n}\right)\frac{|\psi-\bar{\psi}|}{(8+s^2)^2}ds\\
&\le 65 \sup_{s\in(T,s_n]}\left|\frac{\psi-\bar{\psi}}{\log{\frac{s+\frac{\rho_n}{\ga_n}}{T+\frac{\rho_n}{\ga_n}}}}\right|\int_T^r\left(s+\frac{\rho_n}{\ga_n}\right)\frac{\log{\frac{s+\frac{\rho_n}{\ga_n}}{T+\frac{\rho_n}{\ga_n}}}}{(8+s^2)^2}ds.
\end{split}
\]
Choosing $T>0$ larger if necessary, we see
\[
\|F_{1}(\psi,\omega)-F_1(\bar{\psi},\bar{\omega})\|_2 \le \frac12\|\psi-\bar{\psi}\|_1.
\]
Consequently, $\mathcal{F}$ is a contraction mapping from $\mathcal{B}_{\tilde{C}}$ to itself. This suggests that there exists a fixed point $(\psi,\omega)\in \mathcal{B}_{\tilde{C}}$ of $\mathcal{F}$. Then, as noted above, we get $(\psi,\omega)=(\psi_n, (r+\rho_n/\ga_n)\psi_n')$ on $[T,s_n]$. Since $(\psi,\omega)\in \mathcal{B}_{\tilde{C}}$, we have by \eqref{e80},
\begin{equation}
|\psi_n(r)|\le \e_n\left(C(T)+\tilde{C} \log{\frac{r+\frac{\rho_n}{\ga_n}}{T+\frac{\rho_n}{\ga_n}}}\right),\label{e820}
\end{equation}
and 
\begin{equation}
\left(r+\frac{\rho_n}{\ga_n}\right)|\psi_n'(r)|\le \frac{\tilde{C}}{2}\e_n\label{e82}
\end{equation}
for all $r\in[T,s_n]$. 

Let us finish the proof. Fix $T>0$ as above and choose any $R>0$. If $R< T$, we get by the definition of $\psi_n$, \eqref{d00}, \eqref{e31} and \eqref{e80} that 
\[
z_n(r)\le \left(1-\frac{\alpha \beta_*}{2\mu_n^{2-\beta_n}}\right)z(r)-\frac{\alpha\beta_n}{\mu_n^{2-\beta_n}}\frac{r^2}{r^2+8}+O(\e_n)<\left(1-\frac{\alpha \beta_*}{2\mu_n^{2-\beta_n}}\right)z(r)
\]
for all $r\in[R,T]$ if $n$ is large enough. Hence we may assume $T\le R$. Then, similarly, it follows from \eqref{e820} that
\begin{equation}\label{zn}
\begin{split}
&z_n(r)\\
&\le\left(1-\frac{\alpha \beta_*}{2\mu_n^{2-\beta_n}}\right)z(r)-\frac{\alpha\beta_n}{\mu_n^{2-\beta_n}}\frac{r^2}{r^2+8}+\left(C(T)+\tilde{C} \log{\frac{r+\frac{\rho_n}{\ga_n}}{T+\frac{\rho_n}{\ga_n}}}\right)\e_n\\
&\le\left(1-\frac{\alpha \beta_*}{2\mu_n^{2-\beta_n}}\right)z(r)+O(\e_n\log{r})
\end{split}
\end{equation}
for all $r\in[R,s_n]$. Moreover, for any $r\in[s_n,(r_{i,n}-\rho_n)/\ga_n]$, we have
\begin{equation}\label{znn}
\begin{split}
\left(r+\frac{\rho_n}{\ga_n}\right)z_n'(r)&\le\int_0^{s_n}\left\{\left(r+\frac{\rho_n}{\ga_n}\right)z_n'(r)\right\}'dr\\
&=\left(s_n+\frac{\rho_n}{\ga_n}\right)\left(z'(s_n)+\frac{\phi'(s_n)}{\mu_n^{2-\beta_n}}+\psi_n'(s_n)\right).
\end{split}
\end{equation}
Here we use \eqref{d00}, \eqref{e31} and \eqref{e82} to see
\[
\begin{split}
&\left(s_n+\frac{\rho_n}{\ga_n}\right)z'(s_n)=-4+O\left(\frac{1}{s_n^2}\right)+o\left(\frac{\rho_n}{\ga_n}\right),\\
&\left(s_n+\frac{\rho_n}{\ga_n}\right)\frac{\phi'(s_n)}{\mu_n^{2-\beta_n}}=\frac{2\alpha \beta_*}{\mu_n^{2-\beta_n}}+o\left(\frac{1}{s_n^2}\right)+o\left(\frac{\rho_n}{\ga_n}\right),
\end{split}
\]
and 
\[
\left|\left(s_n+\frac{\rho_n}{\ga_n}\right)\psi_n'(s_n)\right|\le \frac{\tilde{C}}{2}\e_n.
\] 
Substituting these formulas into \eqref{znn} and using  $\liminf_{n\to \infty}s_n\e_n^{1/2}>0$, we get
\[
\begin{split}
\left(r+\frac{\rho_n}{\ga_n}\right)z_n'(r)
&\le -4+\frac{2\alpha \beta_*}{\mu_n^{2-\beta_n}}+O(\e_n)
\end{split}
\]
for all $r\in[s_n,(r_{i,n}-\rho_n)/\ga_n]$.  Lastly, for any $r\in[s_n,(r_{i,n}-\rho_n)/\ga_n]$, dividing this inequality by $(r+\rho_n/\ga_n)$ and integrating over $[s_n,r]$, we compute that
\[
\begin{split}
z_n(r)&\le -\left(4-\frac{2\alpha \beta_*}{\mu_n^{2-\beta_n}}\right)\log{\left(r+\frac{\rho_n}{\ga_n}\right)}+\left(4-\frac{2\alpha \beta_*}{\mu_n^{2-\beta_n}}\right)\log{\left(s_n+\frac{\rho_n}{\ga_n}\right)}\\&\ \ \ +O(\e_n\log{r})+O(\e_n\log{s_n})+z_n(s_n)\\
&\le \left(1-\frac{\alpha \beta_*}{2\mu_n^{2-\beta_n}}\right)z(r)+O(\e_n\log{r})
\end{split}
\]
if $n\in \mathbb{N}$ is large enough. This completes the proof.
\end{proof}

Using the previous lemma, we deduce the following asymptotic expansion of the energy. 
\begin{proposition}\label{f5} Assume $i\ge1$ and \eqref{infty}-\eqref{a7}. Moreover, suppose $\beta_*<3/2$ if $i\not=1$. Then we get
\begin{equation}\label{f6}
\mu_{i,n}\int_{\rho_{i,n}}^{r_{i,n}}\la_nf(u_{i,n})rdr=2-\frac{\alpha \beta_*}{\mu_{i,n}^{2-\beta_n}}+o\left(\frac{1}{\mu_{i,n}^{2-\beta_n}}\right),
\end{equation}
and 
\begin{equation}\label{f7}
\int_{\rho_{i,n}}^{r_{i,n}}\la_nu_{i,n}f(u_{i,n})rdr=2-\frac{\alpha \beta_*}{\mu_{i,n}^{2-\beta_n}}+o\left(\frac{1}{\mu_{i,n}^{2-\beta_n}}\right).
\end{equation}
\end{proposition}
\begin{proof} We write $\mu_n=\mu_{i,n}$, $\rho_n=\rho_{i,n}$,  $r_n=r_{i,n}$,  $\ga_n=\ga_{i,n}$, and $z_n=z_{i,n}$. We refer to the proof of Theorem 1 in \cite{MM2}. We first note that
\[
\begin{split}
&\mu_n\int_{\rho_{n}}^{r_{n}}\la_nf(u_{n})rdr
\\&=\frac12\int_{0}^{\frac{r_n-\rho_{n}}{\ga_{n}}}\left(\frac{z_n}{2\mu_n^2}+1\right) e^{z_n+\frac{z_n^2}{4\mu_n^2}+\alpha\mu_n^{\beta_n}\left\{\left(\frac{z_n}{2\mu_n^2}+1\right)^{\beta_n}-1\right\}}\left(r+\frac{\rho_{n}}{\ga_{n}}\right)dr
\end{split}
\]
and
\[
\begin{split}
&\int_{\rho_{n}}^{r_{n}}\la_nu_{n}f(u_{n})rdr
\\&=\frac12\int_{0}^{\frac{r_n-\rho_{n}}{\ga_{n}}}\left(\frac{z_n}{2\mu_n^2}+1\right)^2 e^{z_n+\frac{z_n^2}{4\mu_n^2}+\alpha\mu_n^{\beta_n}\left\{\left(\frac{z_n}{2\mu_n^2}+1\right)^{\beta_n}-1\right\}}\left(r+\frac{\rho_{n}}{\ga_{n}}\right)dr.
\end{split}
\]
So it suffices to show that
\begin{equation}
\begin{split}
&\frac12\int_{0}^{\frac{r_{n}-\rho_{n}}{\ga_{n}}}\left(\frac{z_n}{2\mu_n^2}+1\right)^m e^{z_n+\frac{z_n^2}{4\mu_n^2}+\alpha\mu_n^{\beta_n}\left\{\left(\frac{z_n}{2\mu_n^2}+1\right)^{\beta_n}-1\right\}}\left(r+\frac{\rho_{n}}{\ga_{n}}\right)dr\\
&=2-\frac{\alpha \beta_*}{\mu_n^{2-\beta_n}}+o\left(\frac{1}{\mu_n^{2-\beta_n}}\right)
\end{split}\label{ee}
\end{equation}
for $m=1,2$. To prove \eqref{ee}, we first claim
\begin{align}
I_n&:=\frac12\int_{0}^{s_n}\left(\frac{z_n}{2\mu_n^2}+1\right)^m e^{z_n+\frac{z_n^2}{4\mu_n^2}+\alpha\mu_n^{\beta_n}\left\{\left(\frac{z_n}{2\mu_n^2}+1\right)^{\beta_n}-1\right\}}\left(r+\frac{\rho_{n}}{\ga_{n}}\right)dr\notag\\
&=2-\frac{\alpha \beta_*}{\mu_n^{2-\beta_n}}+o\left(\frac{1}{\mu_n^{2-\beta_n}}\right)\label{e11}
\end{align}
for $m=1,2$ where $s_n>0$ is chosen as in the proof of Lemma \ref{le70}. 
 In fact, using the equation in \eqref{d1} and noting $z_n=O(\log{(8+r^2)})$ on $[0,s_n]$ by \eqref{d00}, \eqref{e31}, \eqref{e80} and \eqref{e820}, we get
\[
\begin{split}
I_n&=-\frac12\int_0^{s_n}\left\{1+O\left(\frac{\log{(8+r^2)}}{\mu_n^{2}}\right)\right\}^{m-1}\\
&\ \ \ \ \ \ \ \ \ \ \ \ \ \ \ \ \ \ \ \ \ \ \ \ \ \ \ \ \ \ \ \ \times \left\{\left(r+\frac{\rho_{n}}{\ga_{n}}\right)\left(z'+\frac{\phi'}{\mu_n^{2-\beta_n}}+\psi_n'\right)\right\}'dr,
\end{split}
\]
for $m=1,2$. Here, we estimate by \eqref{d00}, our choice of $(s_n)$ and \eqref{e2} that
\[
\begin{split}
I_{1,n}:=\int_0^{s_n}\left\{\left(r+\frac{\rho_{n}}{\ga_{n}}\right)z'\right\}'dr=-4+o\left(\frac{1}{\mu_n^{2-\beta_n}}\right),
\end{split}
\]
and by \eqref{e31} that
\[
\begin{split}
I_{2,n}:=\frac{1}{\mu_n^{2-\beta_n}}\int_0^{s_n}\left\{\left(r+\frac{\rho_{n}}{\ga_{n}}\right)\phi'\right\}'dr
=\frac{2\alpha \beta_*}{\mu_n^{2-\beta_n}}+o\left(\frac{1}{\mu_n^{2-\beta_n}}\right),
\end{split}
\]
and further, from \eqref{e82}, that
\[
\begin{split}
I_{3,n}:=\int_0^{s_n}\left\{\left(\frac{\rho_{n}}{\ga_{n}}+r\right)\psi_n'\right\}'dr=o\left(\frac{1}{\mu_n^{2-\beta_n}}\right).
\end{split}
\]
Moreover we assert
\[
\begin{split}
\int_0^{s_n}&\frac{\log{(8+r^2)}}{\mu_n^2} \left\{\left(r+\frac{\rho_{n}}{\ga_{n}}\right)\left(z'+\frac{\phi'}{\mu_n^{2-\beta_n}}+\psi_n'\right)\right\}'dr=o\left(\frac{1}{\mu_n^{2-\beta_n}}\right).
\end{split}
\]
Indeed, if $\beta_*>1$, we have $\mu_n^{2-\beta_n}\cdot\mu_n^{-2}\log{(8+r^2)}\to0$ uniformly on $[0,s_n]$ by \eqref{e81} and then, using the previous three formulas for $I_{1,n},I_{2,n}$ and $I_{3,n}$ above we easily get the assertion. On the other hand, if $\beta_*\le1$, we get by integrating by parts and \eqref{d00}, \eqref{e31}, \eqref{e80} and  \eqref{e82} that
\[
\begin{split}
\int_0^{s_n}&\frac{\log{(8+r^2)}}{\mu_n^2} \left\{\left(r+\frac{\rho_{n}}{\ga_{n}}\right)\left(z'+\frac{\phi'}{\mu_n^{2-\beta_n}}+\psi_n'\right)\right\}'dr\\
&=\frac{1}{\mu_n^2}\bigg\{\log{(8+s_n^2)}\left(I_{1,n}+I_{2,n}+I_{3,n}\right)\\
&\ \ \ \ \ \ \ \ -\int_0^{s_n}\frac{2r}{8+r^2}\left(r+\frac{\rho_n}{\ga_n}\right)\left(\frac{-4r}{(8+r^2)^2}+O\left(\frac{1}{\mu_n^{2-\beta_n}(1+r)}\right)\right)dr\bigg\}\\
&=o\left(\frac{1}{\mu_n^{2-\beta_n}}\right)
\end{split}
\]
by the three formulas for $I_{1,n},I_{2,n}$ and $I_{3,n}$ above and a direct calculation. This  proves the assertion. As a consequence, we get 
\[
I_n=-\frac12(I_{1,n}+I_{2,n}+I_{3,n})+o\left(\frac{1}{\mu_n^{2-\beta_n}}\right).
\] 
This shows the claim.  Next we claim that
\begin{equation}
\begin{split}
J_n^m&:=\int_{s_n}^{\frac{r_n-\rho_{n}}{\ga_{n}}}\left(\frac{z_n(r)}{2\mu_n^2}+1\right)^m e^{z_n+\frac{z_n^2}{4\mu_n^2}+\alpha\mu_n^{\beta_n}\left\{\left(\frac{z_n}{2\mu_n^2}+1\right)^{\beta_n}-1\right\}}rdr\\
&=o\left(\frac{1}{\mu_n^{2-\beta_n}}\right)
\end{split}\label{e12}
\end{equation} 
for $m=1,2$. For the proof, it suffices to consider the case $m=1$ since $J_n^2\le J_n^1$. Hence let us show \eqref{e12} for $m=1$. To this end, we first set a sequence $(c_n)\subset (0,1]$ so that $c_n:=1$ if $n\in \mathbb{N}$ satisfies  $\beta_n\ge1$ and $c_n:=\beta_n$ otherwise. Then, we define a value $c_*>0$ by $\lim_{n\to \infty}c_n=c_*$. Furthermore, we fix a small constant  $\eta>0$ so that $c_*\alpha-(\alpha \beta_*+\eta)/2>0$. Noting \eqref{e70} we may assume 
\begin{equation}\label{e770}
z_n(r)\le\left(1-\frac{\alpha \beta_*+\eta}{2\mu_n^{2-\beta_n}}\right)z(r)
\end{equation}
for all $r\in[s_n,(r_n-\rho_n)/\ga_n]$. In addition, we have that 
\[
\left(\frac{z_n(r)}{2\mu_n^2}+1\right)^{\beta_n}-1\le \frac{c_n z_n(r)}{2\mu_n^2}
\]
for any $r\in[s_n,(r_n-\rho_n)/\ga_n]$. This is clearly obtained by noting  $z_n/(2\mu_n^2)+1\in[0,1]$ in the case $\beta_n\ge1$ and by using the mean value theorem if $\beta_n\in(0,1)$. It follows that
\begin{equation}
\begin{split}
J_n^1&=\int_{s_n}^{\frac{r_n-\rho_{n}}{\ga_{n}}}\left(\frac{z_n}{2\mu_n^2}+1\right) e^{z_n+\frac{z_n^2}{4\mu_n^2}+\alpha\mu_n^{\beta_n}\left\{\left(\frac{z_n}{2\mu_n^2}+1\right)^{\beta_n}-1\right\}}rdr\\
&\le \int_{s_n}^{\frac{r_n-\rho_{n}}{\ga_{n}}}\left(\frac{z_n}{2\mu_n^2}+1\right) e^{\left(1+\frac{\alpha c_n}{2\mu_n^{2-\beta_n}}\right)z_n+\frac{z_n^2}{4\mu_n^2}}rdr\\
&\le \int_{s_n}^{\frac{r_n-\rho_{n}}{\ga_{n}}}\left\{\frac{\left(1-\frac{\alpha \beta_*+\eta}{2\mu_n^{2-\beta_n}}\right)z}{2\mu_n^2}+1\right\}\\
& \ \ \ \  \ \ \ \ \ \times e^{\left(1+\frac{\alpha c_n}{2\mu_n^{2-\beta_n}}\right)\left(1-\frac{\alpha \beta_*+\eta}{2\mu_n^{2-\beta_n}}\right)z+\frac{\left(1-\frac{\alpha \beta_*+\eta}{2\mu_n^{2-\beta_n}}\right)^2z^2}{4\mu_n^2}}rdr
\end{split}\label{e10}
\end{equation}
by \eqref{e770}. Here we note that $\left(1-\frac{\alpha \beta_*+\eta}{2\mu_n^{2-\beta_n}}\right)z(r)/(2\mu_n^2)+1\ge 0$ if and only if 
\[
r\le\sqrt{8\left(e^{\frac{\mu_n^2}{1-\frac{\alpha \beta_*+\eta}{2\mu_n^{2-\beta_n}}}}-1\right)}=:R_n.
\]
Then it is clear by \eqref{e770} that $R_n\ge (r_n-\rho_n)/\ga_n$. On the other hand since $\liminf_{n\to \infty}s_n\e_n^{\frac12}>0$,  it follows from \eqref{e2} that there exists a sequence $(M_n)\subset (0,\infty)$  such that $M_n\to \infty$ and $s_n\ge \sqrt{ M_n\mu_n^{2-\beta_n}}$ for all $n\in \mathbb{N}$. Then from \eqref{e10} we compute with changing the variable by $\tau=-\left(1-\frac{\alpha \beta_*+\eta}{2\mu_n^{2-\beta_n}}\right)z(r)$ and putting $\tilde{s}_n:=-\left(1-\frac{\alpha \beta_*+\eta}{2\mu_n^{2-\beta_n}}\right)z(s_n)$ and $\tilde{R}_n:=-\left(1-\frac{\alpha \beta_*+\eta}{2\mu_n^{2-\beta_n}}\right)z(R_n)$ that
\[
\begin{split}
J_n^1&\le\int_{s_n}^{R_n}\left\{\frac{\left(1-\frac{\alpha \beta_*+\eta}{2\mu_n^{2-\beta_n}}\right)z}{2\mu_n^2}+1\right\}\\
& \ \ \ \  \ \ \ \ \ \times e^{\left(1+\frac{\alpha c_n}{2\mu_n^{2-\beta_n}}\right)\left(1-\frac{\alpha \beta_n+\eta}{2\mu_n^{2-\beta_n}}\right)z+\frac{\left(1-\frac{\alpha \beta_*+\eta}{2\mu_n^{2-\beta_n}}\right)^2z^2}{4\mu_n^2}}rdr\\
&= \frac{2}{\left(1-\frac{\alpha \beta_*+\eta}{2\mu_n^{2-\beta_n}}\right)}\int_{\tilde{s}_n}^{\tilde{R}_n}\left(1-\frac{\tau}{2\mu_n^2}\right)e^{\frac{\tau^2}{4\mu_n^2}+\left\{\frac{1}{2\left(1-\frac{\alpha \beta_*+\eta}{2\mu_n^{2-\beta_n}}\right)}-1-\frac{\alpha c_n}{2\mu_n^{2-\beta_n}}\right\}\tau}d\tau\\
\end{split}
\]
Here notice that $\tilde{s}_n=4\log{s_n}+O(1)$ by \eqref{e81} and $\tilde{R}_n=2\mu_n^2$ from the definition.  Again changing the variable by
\[
t=\frac{\tau}{2\mu_n}+\mu_n\left\{\frac{1}{2\left(1-\frac{\alpha \beta_*+\eta}{2\mu_n^{2-\beta_n}}\right)}-1-\frac{\alpha c_n}{2\mu_n^{2-\beta_n}}\right\}
\]
and setting
\[
\begin{split}
\bar{R}_n&=\frac{\tilde{R}_n}{2\mu_n}+\mu_n\left\{\frac{1}{2\left(1-\frac{\alpha \beta_*+\eta}{2\mu_n^{2-\beta_n}}\right)}-1-\frac{\alpha c_n}{2\mu_n^{2-\beta_n}}\right\}
\end{split}
\] 
and 
\[
\bar{s}_n=\frac{\tilde{s}_n}{2\mu_n}+\mu_n\left\{\frac{1}{2\left(1-\frac{\alpha \beta_*+\eta}{2\mu_n^{2-\beta_n}}\right)}-1-\frac{\alpha c_n}{2\mu_n^{2-\beta_n}}\right\},
\]
we get
\begin{equation}\label{p1}
\begin{split}
J_n^1&\le \frac{4\mu_n}{\left(1-\frac{\alpha \beta_*+\eta}{2\mu_n^{2-\beta_n}}\right)\exp{\left[\mu_n^2\left\{\frac{1}{2\left(1-\frac{\alpha \beta_*+\eta}{2\mu_n^{2-\beta_n}}\right)}-1-\frac{\alpha c_n}{2\mu_n^{2-\beta_n}}\right\}^2\right]}}\\
&\ \ \ \ \ \ \ \times\int_{\bar{s}_n}^{\bar{R}_n}\left\{-\frac{t}{\mu_n}+\frac{1}{2\left(1-\frac{\alpha \beta_*+\eta}{2\mu_n^{2-\beta_n}}\right)}-\frac{\alpha c_n}{2\mu_n^{2-\beta_n}}\right\}e^{t^2}dt.
\end{split}
\end{equation}
Now we calculate by \eqref{e81} and our choices of $(M_n)$ and $\eta$ that   
\begin{equation}\label{p2}
\begin{split}
&\mu_n^{2-\beta_n}\frac{4\mu_n}{\left(1-\frac{\alpha \beta_*+\eta}{2\mu_n^{2-\beta_n}}\right)\exp{\left[\mu_n^2\left\{\frac{1}{2\left(1-\frac{\alpha \beta_*+\eta}{2\mu_n^{2-\beta_n}}\right)}-1-\frac{\alpha c_n}{2\mu_n^{2-\beta_n}}\right\}^2\right]}}\\
&\ \ \ \times\int_{\bar{s}_n}^{\bar{R}_n}\frac{t}{\mu_n}e^{t^2}dt\\
&=O\left(\mu_n^{2-\beta_n}\exp{\left\{-\left(c_*\alpha-\frac{\alpha\beta_*+\eta}{2}+o(1)\right)\mu_n^{\beta_n}\right\}}\right)+O\left(\frac{1}{M_n}\right)\\
&\to0,
\end{split}
\end{equation}
as $n\to \infty$. Similarly, we get 
\begin{equation}\label{p3}
\begin{split}
A_n:&=\mu_n^{2-\beta_n}\frac{4\mu_n}{\left(1-\frac{\alpha \beta_*+\eta}{2\mu_n^{2-\beta_n}}\right)\exp{\left[\mu_n^2\left\{\frac{1}{2\left(1-\frac{\alpha \beta_*+\eta}{2\mu_n^{2-\beta_n}}\right)}-1-\frac{\alpha c_n}{2\mu_n^{2-\beta_n}}\right\}^2\right]}}\\
&\ \ \ \times\int_{\bar{s}_n}^{\bar{R}_n}e^{t^2}dt\\
&\to0
\end{split}
\end{equation}
as $n\to \infty$. In fact, noting $\bar{s}_n=-(1+o(1))\mu_n/2+O(1)$ by \eqref{e81} and $\bar{R}_n=(1+o(1))\mu_n/2$ we decompose 
\[
A_n=A_{1,n}+A_{2,n}
\]
where
\[
\begin{split}
A_{1,n}&:=\mu_n^{2-\beta_n}\frac{4\mu_n}{\left(1-\frac{\alpha \beta_*+\eta}{2\mu_n^{2-\beta_n}}\right)\exp{\left[\mu_n^2\left\{\frac{1}{2\left(1-\frac{\alpha \beta_*+\eta}{2\mu_n^{2-\beta_n}}\right)}-1-\frac{\alpha c_n}{2\mu_n^{2-\beta_n}}\right\}^2\right]}}\\
&\ \ \ \ \times\int_{\{|t|\le \mu_n/4\}}e^{t^2}dt.
\end{split}
\]
and
\[
\begin{split}
A_{2,n}&:=\mu_n^{2-\beta_n}\frac{4\mu_n}{\left(1-\frac{\alpha \beta_*+\eta}{2\mu_n^{2-\beta_n}}\right)\exp{\left[\mu_n^2\left\{\frac{1}{2\left(1-\frac{\alpha \beta_*+\eta}{2\mu_n^{2-\beta_n}}\right)}-1-\frac{\alpha c_n}{2\mu_n^{2-\beta_n}}\right\}^2\right]}}\\
&\ \ \ \ \times\int_{[\bar{s}_n,\bar{R}_n]\cap\{|t|\ge\mu_n/4\}}e^{t^2}dt.
\end{split}
\]
Since 
\[
A_{1,n}=O\left(\frac{\mu_n^{3-\beta_n}}{e^{(1+o(1))\mu_n^2/4}}\cdot \frac{\mu_n}{2} e^{\mu_n^2/16}\right),
\]
we easily get $A_{1,n}\to0$. On the other hand, since  
\[
\begin{split}
A_{2,n}&\le\mu_n^{2-\beta_n}\frac{4\mu_n}{\left(1-\frac{\alpha \beta_*+\eta}{2\mu_n^{2-\beta_n}}\right)\exp{\left[\mu_n^2\left\{\frac{1}{2\left(1-\frac{\alpha \beta_*+\eta}{2\mu_n^{2-\beta_n}}\right)}-1-\frac{\alpha c_n}{2\mu_n^{2-\beta_n}}\right\}^2\right]}}\\
&\ \ \ \ \times 4\int_{[\bar{s}_n,\bar{R}_n]\cap\{|t|\ge\mu_n/4\}}\frac{|t|}{\mu_n}e^{t^2}dt,
\end{split}
\]
we get $A_{2,n}\to0$ similarly to the calculation above. Consequently, using \eqref{p2} and \eqref{p3} for \eqref{p1}, we get $J_n^1=o(\mu_n^{-(2-\beta_n)})$. This proves  our claim \eqref{e12}. Then by \eqref{e11} and \eqref{e12}, we readily obtain \eqref{ee}.  This finishes the proof.
\end{proof}

We also get the following.
\begin{lemma}
Suppose $i\ge1$ and \eqref{infty}-\eqref{a7} and further, if $i\not=1$, let $\beta_*<3/2$. Then we have 
\begin{equation}\label{f66}
\mu_{i,n}\int^{\rho_{i,n}}_{r_{i-1,n}}\la_nf(u_{i,n})rdr=o\left(\frac{1}{\mu_{i,n}^{2-\beta_n}}\right),
\end{equation}
and 
\begin{equation}\label{f77}
\int^{\rho_{i,n}}_{r_{i-1,n}}\la_nu_{i,n}f(u_{i,n})rdr=o\left(\frac{1}{\mu_{i,n}^{2-\beta_n}}\right).
\end{equation}
\end{lemma}
\begin{proof}
By the first assertion in Lemma \ref{lem:e0} and \eqref{e2}, we get $(\rho_{i,n}-r_{i-1,n})/\ga_{i,n}=o(\mu_{i,n}^{-(2-\beta_n)})$. It folows that 
\[
\begin{split}
&\frac12\int_{\frac{r_{i-1,n}-\rho_{i.n}}{\ga_{i,n}}}^0\left(\frac{z_{i,n}}{2\mu_n^2}+1\right)^m e^{z_{i,n}+\frac{z_{i,n}^2}{4\mu_n^2}+\alpha\mu_n^{\beta_n}\left\{\left(\frac{z_{i,n}}{2\mu_n^2}+1\right)^{\beta_n}-1\right\}}\left(r+\frac{\rho_{n}}{\ga_{n}}\right)dr\\
&=o\left(\frac1{\mu_{i,n}^{2-\beta_n}}\right),
\end{split}
\]
for $m=1,2$. This proves \eqref{f66} and \eqref{f77}. We finish the proof.
\end{proof}

We get the proof of Proposition \ref{e00}.
\begin{proof}[Proof of Proposition \ref{e00}] The first formula follows from \eqref{rad}, \eqref{f7} and \eqref{f77}. The second one is proved by \eqref{f6} and \eqref{f66}. This finishes the proof. 

\end{proof}

We end this section by proving the next key lemma. 
\begin{proposition}\label{f0}
We assume $i\ge1$ and \eqref{infty}-\eqref{a7}. Moreover, suppose $\beta_*<3/2$ if $i>1$. Then we get 
\begin{align}
\lim_{n\to \infty}\frac{\log{\frac{1}{\la_nr_{i,n}^2}}}{\mu_{i,n}^{\beta_n}}=\alpha\left(1-\frac{\beta_*}2\right),
\label{f1}
\end{align}
and
\begin{align}
\lim_{n\to \infty}\mu_{i,n}r_{i,n} |u_{i,n}'(r_{i,n})|=2.
\label{f4}
\end{align}
\end{proposition}
\begin{proof}
We denote $\mu_n=\mu_{i,n}$, $\rho_n=\rho_{i,n}$, $r_n=r_{i,n}$, $\ga_n=\ga_{i,n}$,   and $z_n=z_{i,n}$ as usual. To deduce \eqref{f1}, we first claim that there exists a constant $C>0$ such that 
\[
\la_n r_n^2 \mu_n^2e^{\left\{\alpha\left(1-\frac{\beta_*}2\right)+o(1)\right\}\mu_n^{\beta_n}}\le C
\]
for all $n\in \mathbb{N}$. In fact, \eqref{e70} implies that for any $r\in [1,(r_{n}-\rho_n)/\ga_n]$,
\[
z_n(r)\le \left(1-\frac{\alpha \beta_*+o(1)}{2\mu_n^{2-\beta_n}}\right)z(r),
\]
if $n\in \mathbb{N}$ is large enough. Choosing $r=(r_n-\rho_n)/\ga_n$, we get by the first and second  assertions in Lemma \ref{ap1} that
\[
\begin{split}
0& \le \mu_n^2-\left(1-\frac{\alpha \beta_*+o(1)}{2\mu_n^{2-\beta_n}}\right)\log{2\la_nr_{n}^2\mu_n^2e^{\mu_n^2+\alpha\mu_n^{\beta_n}}}+O(1).\end{split}
\]
This implies that there exists a constant $\bar{C}>0$ such that
\[
\log{\la_nr_{n}^2\mu_n^2e^{\left\{\alpha\left(1-\frac{\beta_*}{2}\right)+o(1)\right\}\mu_n^{\beta_n}}}\le \bar{C}
\]
for all large $n\in \mathbb{N}$. This proves the claim. It follows that
\begin{equation}
\liminf_{n\to \infty}\frac{\log{\frac{1}{\la_nr_{n}^2}}}{\mu_n^{\beta_n}}\ge \alpha\left(1-\frac{\beta_*}{2}\right).\label{f9}
\end{equation}
Next, we shall show
\begin{equation}
\limsup_{n\to \infty}\frac{\log{\frac{1}{\la_nr_{n}^2}}}{\mu_n^{\beta_n}}\le \alpha\left(1-\frac{\beta_*}{2}\right).\label{f10}
\end{equation}
To do this, we use the first formula in Lemma \ref{lem:id} to obtain 
\begin{equation}
\begin{split}
&2\mu_n^2\\
&=\log{\frac{r_{n}}{\ga_n}}\\
&\times\int_0^{\frac{r_{n}-\rho_n}{\ga_n}}\left(\frac{z_n}{2\mu_n^2}+1\right) e^{z_n+\frac{z_n^2}{4\mu_n^2}+\alpha\mu_n^{\beta_n}\left\{\left(\frac{z_n}{2\mu_n^2}+1\right)^{\beta_n}-1\right\}}\left(r+\frac{\rho_{n}}{\ga_{n}}\right)dr\\
&+\int_0^{\frac{r_{n}-\rho_n}{\ga_n}}\left(\frac{z_n}{2\mu_n^2}+1\right) e^{z_n+\frac{z_n^2}{4\mu_n^2}+\alpha\mu_n^{\beta_n}\left\{\left(\frac{z_n}{2\mu_n^2}+1\right)^{\beta_n}-1\right\}}\\
&\ \ \ \ \ \ \ \ \ \ \ \ \ \ \ \ \ \ \ \ \ \ \ \ \ \ \ \ \ \ \ \ \ \ \ \ \ \ \ \ \ \ \ \times\left(r+\frac{\rho_{n}}{\ga_{n}}\right)\log{\frac{1}{\left(r+\frac{\rho_{n}}{\ga_{n}}\right)}}dr.
\end{split}\label{f8}
\end{equation}
Here, we observe that for any value $R_0>1$,
\[
\begin{split}
&\int_0^{\frac{r_{n}-\rho_n}{\ga_n}}\left(\frac{z_n}{2\mu_n^2}+1\right) e^{z_n+\frac{z_n^2}{4\mu_n^2}+\alpha\mu_n^{\beta_n}\left\{\left(\frac{z_n(r)}{2\mu_n^2}+1\right)^{\beta_n}-1\right\}}\\
&\ \ \ \ \ \ \ \ \ \ \ \ \ \ \ \ \ \ \ \ \ \ \ \ \ \ \ \ \ \ \ \ \ \ \ \ \ \ \ \ \ \ \ \times\left(r+\frac{\rho_{n}}{\ga_{n}}\right)\log{\frac1{\left(r+\frac{\rho_{n}}{\ga_{n}}\right)}}dr\\
&\le \int_0^{R_0}e^zr\log{\frac1r}dr+o(1)
\end{split}
\]
where $o(1)\to0$ as $n\to \infty$. Then, since we can choose $R_0>1$ so large that $\int_0^{R_0}e^zr\log{\frac1r}dr<0$, the second term of the right hand side of \eqref{f8} is negative value for all large $n\in \mathbb{N}$. Hence using this and \eqref{f6} for \eqref{f8}, we get 
\[
2\mu_n^2\le \left(2-\frac{\alpha \beta_*}{\mu_n^{2-\beta_n}}+o\left(\frac{1}{\mu_n^{2-\beta_n}}\right)\right)\log{2\la_n r_{n}^2\mu_nf_n(\mu_n)}.
\]
It follows that
\[
\frac{\log{\frac{1}{\la_nr_{n}^2}}}{\mu_n^{\beta_n}}\le \alpha\left(1-\frac{\beta_*}{2}\right)+o(1).
\]
This proves \eqref{f10}. Then \eqref{f9} and \eqref{f10} show \eqref{f1}. Finally we shall ensure \eqref{f4}. For any $r\in(\rho_n,r_n)$, multiplying the equation in \eqref{rad} by $r$ and integrating by parts over $(\rho_n,r_n)$, we get
\[
-r_{n}u_n'(r_{n})=\la_n\int_{\rho_{n}}^{r_{n}}f(u_n)rdr.
\]
Hence we obtain from \eqref{f6} that
\[
\lim_{n\to \infty}\mu_nr_{n}|u_n'(r_{n})|=2.
\]
This gives \eqref{f4}. We complete the proof. 
\end{proof}

\section{Behavior of non-concentrating parts}\label{more}
In this section we mainly discuss the behavior of $(u_{i,n})$ which does not blow up. This is useful to deduce the precise information of the weak limit. Especially, Lemmas \ref{b21} and \ref{lem:g6} will be important for the proof in the case of (ii) of Theorems \ref{a10} and \ref{a31}. We begin with the next basic lemma. Let $\mathcal{N}_{\la,\beta}$ be the Nehari manifold defined in Section \ref{main}. 
\begin{lemma}\label{b1} For any $\Lambda_0\in(0,\Lambda_1)$ and $\beta_0\in(0,2)$, we have a constant $K>0$ such that
\[
\int_B|\nabla u|^2dx\ge K
\]
for all $u\in \mathcal{N}_{\la,\beta}$ and all $(\la,\beta)\in(0,\Lambda_0]\times[\beta_0,2]$.
\end{lemma} 
\begin{proof} The proof is standard. For the readers' convenience we show the proof. First fix $\Lambda_0\in(0,\Lambda_1)$ and $\beta_0\in(0,2)$ and assume $\la\in(0,\Lambda_0]$ and $\beta\in[\beta_0,2]$. Next  choose $\e>0$ so that $(1+\e)\Lambda_0<\Lambda_1$. Then for any $p>2$, we can find a constant $C_1>0$ independently of $\beta\in[\beta_0,2]$ such that $|t^2e^{t^2+\alpha|t|^\beta}|\le (1+\e)t^2+C_1t^pe^{(1+\alpha)t^2}$ for all $t\in \mathbb{R}$. Then the H\"{o}lder,   Poincare, Sobolev, and Trudinger-Moser \eqref{TM} inequalities suggest that there exists a constant $C_2>0$, which is independent of $\beta\in[\beta_0,2]$, such that 
\[
\begin{split}
\int_Bu^2e^{u^2+\alpha|u|^\beta}dx&\le (1+\e)\int_Bu^2dx+C_1\left(\int_B|u|^{2p}dx\right)^{\frac{1}{2}}\left(\int_Be^{2(1+\alpha)u^2}dx\right)^{\frac12}\\
&\le \frac{1+\e}{\Lambda_1}\|u\|^2_{H^1_0(B)}+C_2\|u\|_{H^1_0(B)}^p,
\end{split}
\]
for all $u\in H^1_0(B)$ with $\|u\|_{H^1_0(B)}^2\le 4\pi/(2(1+\alpha))$. Hence, it follows  that
\[
\langle I_{\la,\beta}'(u),u\rangle\ge \left(1-\frac{(1+\e)\Lambda_0}{\Lambda_1}\right)\|u\|_{H^1_0(B)}^2-\Lambda_0C_2\|u\|_{H^1_0(B)}^p
\]
for any $u\in H^1_0(B)$ satisfying $\|u\|_{H^1_0(B)}^2\le 4\pi/(2(1+\alpha))$. Since $2<p$,  we get a constant $C_3>0$, which is independent of $\la\in(0,\Lambda_0]$ and $\beta\in[\beta_0,2]$, such that $\langle I_{\la,\beta}(u),u\rangle>0$ for all $u\in H^1_0(B)$ with $\|u\|_{H^1_0(B)}\le C_3$. Therefore $u\in \mathcal{N}_{\la,\beta}$ implies $\|u\|_{H^1_0(B)}\ge C_3$. This finishes the proof. 
\end{proof}

From now on, as usual, let $k\in\{0\}\cup\mathbb{N}$ and $\{(\la_n,\beta_n)\}\subset (0,\infty)\times (0,2)$ be a sequence such that $(\la_n,\beta_n)\to (\la_*,\beta_*)$ as $n\to \infty$ for some value $(\la_*,\beta_*)\in [0,\infty)\times (0,2)$. Moreover, assume that $(u_n)$ is a sequence of solutions satisfying $u_n\in S_{k,\la_n,\beta_n}$ for all $n\in \mathbb{N}$.  In the following lemmas, we always suppose 
\begin{equation}\label{bdd4}
\int_{B}|\nabla u_n|^2dx\text{ uniformly bounded for all $n\in \mathbb{N}$},
\end{equation}
if $k\not=0$. All the other notations below are defined as in the main theorems. We get the following. 
\begin{lemma}\label{b20} Assume \eqref{bdd4}. If $\mu_{i,n}\to \infty$ as $n\to \infty$ for some $i\in\{1,\cdots,k+1\}$, then we have $\lim_{n\to \infty}\rho_{i,n}=0.$ On the other hand, if $\lim_{n\to \infty}r_{i,n}=0$ for some $i\in\{1,\cdots,k\}$, then we get $\mu_{j,n}\to \infty$ for all $j=1,\cdots,i$. Finally, if $\la_*=0$, then we obtain $\mu_{k+1,n}\to \infty$.
\end{lemma}
\begin{proof} First assume $\mu_{i,n}\to \infty$ as $n\to \infty$ for some $i\in\{1,\cdots,k+1\}$. Then Lemma \ref{RL} implies that there exists a constant $c>0$ such that
\[
\rho_{i,n}\mu_{i,n}^2\le c^2 \int_B|\nabla u_n|^2dx,
\]
for all $n\in \mathbb{N}$. Hence, by our assumptions, we get $\rho_{i,n}\to0$ as $n\to \infty$. This shows the first assertion. Next, we suppose $r_{i,n}\to0$ for some $i\in\{1,\cdots,k\}$. Then, assume there exists a number $j\in\{1,\cdots,i\}$ such that $\mu_{j,n}$ is uniformly bounded up to a subsequence on the contrary. Then for any $x\in B$, we put  $\tilde{u}_{j,n}(x):=u_n(r_{j,n}x)$ if $r_{j-1,n}/r_{j,n}<|x|<1$ and $\tilde{u}_{j,n}(x):=0$ otherwise. It follows that $\tilde{u}_{j,n}\in \mathcal{N}_{\la_n r_{j,n}^2,\beta_n}$. Then since $\la_n r_{i,n}^2\to0$, we get by Lemma \ref{b1} that there exists a constant $K>0$ such that
\[
K\le \int_B\la_nr_{i,n}^2 \tilde{u}_{j,n}f_n(\tilde{u}_{j,n})dx\to0,
\]
since $\tilde{u}_{j,n}$ is uniformly bounded. This is a contradiction. Hence we prove the second assertion. Finally, assume $\la_*=0$ and $\mu_{k+1,n}$ is uniformly bounded up to a subsequence on the contrary.  Then for all $x\in B$, we put $\bar{u}_{k+1,n}(x):=u_n(x)$ if $r_{k,n}<|x|<1$ and $\bar{u}_{k+1,n}(x):=0$ otherwise. Since $\bar{u}_{k+1,n}\in \mathcal{N}_{\la_n,\beta_n}$ and $\la_*=0$, we can apply  Lemma \ref{b1} again and get a constant $K>0$ such that
\[
K\le \la_n\int_B \bar{u}_{k+1,n}f_n(\bar{u}_{k+1,n})dx\to0
\]
since $\bar{u}_{k+1,n}$ is uniformly bounded. This is a contradiction.  This ensures the last assertion. We complete the proof.
\end{proof} 

After this, we regard $u_n=u_n(|x|)$ and study the behavior of the function $u_n(r)$ ($r\in[0,1]$). Let us give the next three standard lemmas. 
\begin{lemma}\label{b22} Suppose \eqref{bdd4}. Assume that there exists a number $i\in\{1,\cdots,k\}$ such that $\mu_{i,n}$ is uniformly bounded for all $n\in\mathbb{N}$. Then $\mu_{i+1,n}$ is also uniformly bounded. Furthermore,  there exist constants $0\le r_{i-1}\le\rho_i<r_i<\rho_{i+1}<1$ such that $r_{i-1,n}\to r_{i-1}$, $r_{i,n}\to r_i$, $\rho_{i,n}\to \rho_i$ and $\rho_{i+1,n}\to \rho_{i+1}$ by  subtracting a subsequence if necessary. Moreover, $r_{i-1}=\rho_i$ if and only if $\rho_i=0$. 
\end{lemma}
\begin{proof} First, assume that $\mu_{i+1,n}\to \infty$ up to a subsequence on the contrary. Then the first assertion in Lemma \ref{b20} suggests that $r_{i,n}<\rho_{i+1,n}\to0$. Then the second assertion in the same lemma implies $\mu_{i,n}\to \infty$ which is a contradiction. This proves the first assertion in the present lemma. Next, we choose  constants $0\le r_{i-1}\le\rho_i\le r_{i}\le \rho_{i+1}\le1$  so that $r_{i-1,n}\to r_{i-1}$, $r_{i,n}\to r_i$, $\rho_{i,n}\to \rho_i$, and $\rho_{i+1,n}\to \rho_{i+1}$ by taking a subsequence if necessary. We claim $\rho_i<r_{i}$. In fact, if $\rho_i=r_i$ on the contrary,  Lemma \ref{lem:id} shows that 
\[
\begin{split}
1&=\int_{\rho_{i,n}}^{r_{i,n}}\la_n \left|\frac{f_n(u_n)}{\mu_{i,n}}\right|r\log{\frac{r_{i,n}}{r}}dr\le \la_n \left|\frac{f_n(\mu_{i,n})}{\mu_{i,n}}\right|\max_{r\in(0,1]}\left|r\log{\frac{1}{r}}\right|(r_{i,n}-\rho_{i,n})\\
&\to0.
\end{split}
\]
This is a contradiction. 
 Next we show $r_{i}<\rho_{i+1}$. Otherwise, we get $0<r_i=\rho_{i+1}$. Then again Lemma \ref{lem:id} implies 
\[
\begin{split}
1&=\int_{r_{i,n}}^{\rho_{i+1,n}}\la_n \left|\frac{f_n(u_n)}{\mu_{i+1,n}}\right|r\log{\frac{r}{r_{i,n}}}dr\le \la_n\left|\frac{f_n(\mu_{i+1,n})}{\mu_{i+1,n}}\right|\log{\frac{\rho_{i+1,n}}{r_{i,n}}}\frac{\rho_{i+1,n}^2-r_{i,n}^2}{2}\\
&\to 0,
\end{split}
\]
which is a contradiction. 
Next we ensure $\rho_{i+1}<1$. If not, we have $1\ge r_{i+1,n}>\rho_{i+1,n}\to1$ and then analogously, we get
\[
\begin{split}
1&=\int_{\rho_{i+1,n}}^{r_{i+1,n}}\la_n \left|\frac{f_n(u_n)}{\mu_{i+1,n}}\right|r\log{\frac{r_{i+1,n}}{r}}dr=O\left(\log{\frac{r_{i+1,n}}{\rho_{i+1,n}}}\frac{r_{i+1,n}^2-\rho_{i+1,n}^2}{2}\right)\to0.
\end{split}
\]
This is a contradiction. Finally, we suppose $r_{i-1}=\rho_i>0$ on the contrary. Then similarly, we see that
\[
\begin{split}
1&=\int^{\rho_{i,n}}_{r_{i-1,n}}\la_n \left|\frac{f_n(u_n)}{\mu_{i,n}}\right|r\log{\frac{r}{r_{i-1,n}}}dr=O\left(\log{\frac{\rho_{i,n}}{r_{i-1,n}}}\frac{\rho_{i,n}^2-r_{i-1,n}^2}{2}\right)\to0.
\end{split}
\]
This is a contradiction. This completes the proof. 
\end{proof}
\begin{lemma}\label{lem:q1}
Assume \eqref{bdd4}. Suppose that for some $i\in\{1,\cdots,k+1\}$, there exist constants $\mu_i\ge 0$, $r_{i-1}\le\rho_i<r_i\le1$ such that $\mu_{i,n}\to \mu_i$, $\rho_{i,n}\to \rho_i$, and $r_{j,n}\to r_j$ for $j=i-1,i$. Then  we have a nontrivial function $w_i$  in $(r_{i-1},r_i)$ such that $u_{i,n}/\mu_{i,n}\to w_i$ in $C^2_{\text{loc}}((r_{i-1},r_i))$. Furthermore, if $r_{i-1}<\rho_i$ (which yields $i\not=1$), then we get $r_{i-1}>0$. Finally, $\rho_i=0$ implies $\lim_{r\to 0+0}w_i(r)=1$.
\end{lemma}
\begin{proof} We may suppose $u_{i,n}\ge0$. Put $w_{i,n}:=u_{i,n}/\mu_{i,n}$. Then it satisfies 
\begin{equation}\label{wi}
\begin{cases}
-w_{i,n}''-\frac{1}{r}w_{i,n}'=\la_n w_{i,n}\frac{f_n(\mu_{i,n}w_{i,n})}{\mu_{i,n}w_{i,n}},\ 0< w_{i,n}\le 1\text{ in }(r_{i-1,n},r_{i,n}),\\
w_{i,n}(r_{i,n})=0=w_{i,n}'(\rho_{i,n}),\ w_{i,n}(\rho_{i,n})=1,\\
w_{i,n}(r_{i-1,n})=0\text{ if }i\not=1.
\end{cases}
\end{equation}
Notice that $\la_n w_{i,n}f_n(\mu_{i,n}w_{i,n})/(\mu_{i,n}w_{i,n})$ is uniformly bounded in $(r_{i-1,n},r_{i,n})$.  Then using the equation and conditions in \eqref{wi}, we clearly get  a function $w_i \ge0$ such that $w_{i,n}\to w_i$ in $C^2_{\text{loc}}((r_{i-1},r_i))$. Now, let us assume $r_{i-1}<\rho_i$. We may suppose $i\not=1$. Then we have $\rho_i\in(r_{i-1},r_i)$ and thus, we obviously see $w_i(\rho_i)=1$. It follows that $r_{i-1}>0$. Otherwise, Lemma \ref{lem:id} shows that 
\[
\begin{split}
1&=\log{\frac{1}{r_{i-1,n}}}\left(\int_{0}^{\rho_{i}}\la_* w_i \frac{f_*(\mu_i w_i)}{\mu_i w_i}rdr+o(1)\right)+\int_{0}^{\rho_{i}}\la_* w_i\frac{f_*(\mu_i w_{i})}{\mu_i w_i}r\log{r}dr\\
&\ \ +o(1)
\end{split}
\]
where we defined $f_*(t)/t=1$ if $t=0$.  Since $\la_*\not=0$ by Lemma \ref{b20}, $f(t)/t\ge1$ for any $t\ge0$ and $w_i(\ge0)$ is nontrivial on $(0,\rho_i)$, we get that the right hand side of the formula above diverges to infinity. This is a contradiction. This proves the second assertion of the lemma. Finally, let us suppose $0=\rho_i=r_{i-1}$.  Then we claim that there exists a constant $C>0$ such that
\begin{equation}\label{q2}
|w_{i,n}'(r)|\le C
\end{equation}
for all $r\in[\rho_{i,n},r_{i,n}]$ and all $n\in \mathbb{N}$. To see this, for any $r\in[\rho_{i,n},r_{i,n}]$, we multiply the equation in \eqref{wi} by $r$ and integrate over $[\rho_{i,n},r]$ and get
\begin{equation}\label{q22}
-rw_{i,n}'(r)=O\left(\frac{r^2-\rho_{i,n}^2}{2}\right),
\end{equation}
for all $r\in[\rho_{i,n},r_{i,n}]$. This readily proves the claim. Then we confirm that $\lim_{r\to 0+0}w_i(r)=1$. If not, we have a sequence $(\sigma_n)\subset (0,r_i)$ and a constant $\e_0\in(0,1]$ such that $\sigma_n\to 0$ and $w_i(\sigma_n)\to 1-\e_0$ as $n\to \infty$. Then we can choose a sequence $(\tilde{\sigma}_n)\subset (\rho_{i,n},r_{i,n})$ so that $\tilde{\sigma}_n\to 0$ and $w_{i,n}(\tilde{\sigma}_n)\to 1-\e_0$ by selecting a suitable subsequence. Consequently, it follows from  the mean value theorem that there exists a sequence $(\bar{\sigma}_n)\subset (\rho_{i,n},\tilde{\sigma}_{n})$ such that $\bar{\sigma}_n\to 0$ and 
\[
w_{i,n}'(\bar{\sigma}_n)=\frac{w_{i,n}(\tilde{\sigma}_n)-w_{i,n}(\rho_{i,n})}{\tilde{\sigma}_n-\rho_{i,n}}\to -\infty.
\]
This contradicts \eqref{q2}. This finishes the proof. 
\end{proof}
\begin{lemma}\label{b23}
We suppose \eqref{bdd4}. Assume that for some $i\in \{1,\cdots,k\}$, there exists a value $\mu_i\ge0$ such that $\mu_{i,n}\to\mu_i$. Then by subtracting a subsequence, we have  a constant $\mu_{i+1}\ge0$ such that $\mu_{i+1,n}\to\mu_{i+1}$ and $\lim_{n\to \infty}(\mu_{i+1,n}/\mu_{i,n})\in(0,\infty)$. Especially,  $\mu_i>0$ ($=0$) yields $\mu_{i+1}>0$ ($=0$ respectively).
\end{lemma}
\begin{proof}
We assume $u_i\ge0$. Put $w_{j,n}:=u_{j,n}/\mu_{j,n}$ for $j=i,i+1$. Note that, Lemmas \ref{b20} and \ref{b22} imply $\la_*\not=0$ and there exist values $\mu_{i+1}\ge0$ and $r_{i-1}\le\rho_i<r_i<\rho_{i+1}<r_{i+1}$ such that $\mu_{i+1,n}\to \mu_{i+1}$, $r_{j,n}\to r_j$ for $j=i-1,i,i+1$ and $\rho_{j,n}\to\rho_j$ for $j=i,i+1$ up to a subsequence. Moreover, by Lemma \ref{lem:q1}, there exist continuous functions $w_i\ge0$ in $[\rho_i,r_i)$ and $w_{i+1}\le 0$ in $(r_{i},r_{i+1})$ such that  $w_{j,n}\to w_j$ in $C^2_{\text{loc}}((r_{j-1},r_j))$ and $|w_j(\rho_j)|=1$ for $j=i,i+1$. Then as usual, multiplying the equation for $u_n$ by $r$ and integrating over $(\rho_{i,n},\rho_{i+1,n})$, we get
\[
\frac{\mu_{i+1,n}}{\mu_{i,n}}=-\frac{\int_{\rho_{i,n}}^{r_{i,n}}\la_n\frac{f_n(u_n)}{\mu_{i,n}}rdr}{\int_{r_{i,n}}^{\rho_{i+1,n}}\la_n\frac{f_n(u_n)}{\mu_{i+1,n}}rdr}\to -\frac{\int_{\rho_{i}}^{r_{i}}\la_*w_i\frac{f_*(\mu_i w_i)}{\mu_i w_i}rdr}{\int_{r_{i}}^{\rho_{i+1}}\la_*w_{i+1}\frac{f_*(\mu_{i+1} w_{i+1})}{\mu_{i+1} w_{i+1}}rdr}\in(0,\infty),
\]
since $f_*(t)/t\ge1$ for any $t\ge0$ where we again defined $f_*(t)/t=1$ for $t=0$. This completes the former assertion. Then the latter one is clearly confirmed. This ends the proof. 
\end{proof} 

Finally, we prove the next two important  lemmas.  
\begin{lemma}\label{b21}
Suppose \eqref{bdd4}. Let $k\ge1$ and choose $N\in \{1,\cdots,k\}$. Assume that $\mu_{N,n}\to \infty$ and the formula \eqref{f1} holds for $i=N$. Moreover, suppose there exists a constant $\mu_{N+1}\ge 0$ such that $\lim_{n\to \infty}\mu_{N+1,n}= \mu_{N+1}$. Then we get $\la_*\not=0$ and, taking a subsequence if necessary, we have $\lim_{n\to \infty}\rho_{N+1,n}=0$, $\lim_{n\to \infty}(r_{N,n}/\rho_{N+1,n})=0$, and further,
\begin{equation}\label{g1}
\lim_{n\to\infty}\rho_{N+1,n}^2\frac{f_n(\mu_{N+1,n})}{\mu_{N+1,n}}\log{\frac{1}{r_{N,n}}}=\frac{2}{\la_*  },
\end{equation}
and
\begin{equation}\label{g2}
\lim_{n\to\infty}\frac{\rho_{N+1,n}^2f_n(\mu_{N+1,n})}{r_{N,n}u_{N,n}'(r_{N,n})}=\frac2{\la_*  }.
\end{equation}
\end{lemma}
\begin{proof}
Without losing the generality we may assume $u_{N+1,n}\ge0$. From Lemma \ref{b20}, we get $\la_*\not=0$. Then \eqref{f1} with $i=N$ and our assumption $\mu_{N,n}\to \infty$ imply $r_{N,n}\to0$. Then Lemmas \ref{b22} and \ref{lem:q1} yield $\rho_{N+1,n}\to0$. Moreover, we claim $r_{N,n}/\rho_{N+1,n}\to0$. In fact, using Lemma \ref{lem:id}, we get
\[
1= \int_{r_{N,n}}^{\rho_{N+1,n}}\la_n\frac{f_n(u_n)}{\mu_{N+1,n}}r\log{\frac{r}{r_{N,n}}}dr\le \la_n \frac{f_n(\mu_{N+1,n})}{\mu_{N+1,n}}\log{\frac{\rho_{N+1,n}}{r_{N,n}}}\frac{\rho_{N+1,n}^2-r_{N,n}^2}{2}.
\]
This formula implies $\log{(\rho_{N+1,n}/r_{N,n})}\to \infty$. This shows the claim. Now, let us deduce \eqref{g1} and \eqref{g2}. To this end, we put $\hat{r}_n:=r_{N,n}/\rho_{N+1,n}$ and $\hat{w}_n(r):=u_{N+1,n}(\rho_{N+1,n}r)/\mu_{N+1,n}$ for all $r\in[\hat{r}_n,1]$. Then again using the equation in \eqref{rad} with the conditions $\hat{w}_n(1)=1$, $\hat{w}_n'(1)=0$ and previous claims, we find a function $\hat{w}_0$ such that $\hat{w}_n\to \hat{w}_0$ in $C^2_{\text{loc}}((0,1])$ and get   
\begin{equation}\label{eq:w0}
\begin{cases}
-\hat{w}_0''(r)-\frac1r \hat{w}_0'(r)=0,\ 0\le \hat{w}_0\le1,\ \hat{w}_0'\ge0\text{ in }(0,1),\\
\hat{w}_0(1)=1,\ \hat{w}_0'(1)=0.
\end{cases}
\end{equation}
We readily compute that $\hat{w}_0=1$. Finally, we use Lemma \ref{lem:id} to see
\[
\begin{split}
\mu_{N+1,n}&=\int_{r_{N,n}}^{\rho_{N+1,n}}f_n(u_n)r\log{\frac{r}{r_{N,n}}}dr\\
&=\la_n \rho_{N+1,n}^2f_n(\mu_{N+1,n})\log{\frac{\rho_{N+1,n}}{r_{N,n}}}\int_{\hat{r}_n}^{1}\frac{f_n(\mu_{N+1,n}\hat{w}_n)}{f_n(\mu_{N+1,n})}rdr\\
&\ \ \ +\la_n \rho_{N+1,n}^2f_n(\mu_{N+1,n})\int_{\hat{r}_n}^{1}\frac{f_n(\mu_{N+1,n}\hat{w}_n)}{f_n(\mu_{N+1,n})}r\log{r}dr
\end{split}
\]
Since $f_n(\mu_{N+1,n}\hat{w}_n)/f_n(\mu_{N+1,n})\to1$ on $(0,1)$, the Lebesque convergence theorem and previous claims give 
\[
\frac{\mu_{N+1,n}}{\rho_{N+1,n}^2f_n(\mu_{N+1,n})\log{\frac{\rho_{N+1,n}}{r_{N,n}}}}=\frac{\la_*}{2}+o(1).
\]
This proves \eqref{g1}. On the other hand, multiplying the equation in \eqref{rad} with $i=N+1$ by $r$ and integrating over $(r_{N,n},\rho_{N+1,n})$, we see
\[
\begin{split}
r_{N,n}u_n'(r_{N,n})&=\int_{r_{N,n}}^{\rho_{N+1,n}}\la_nf_n(u_n)rdr \\
&=\la_n \rho_{N+1,n}^2f_n(\mu_{N+1,n})\int_{\hat{r}_n}^{1}\frac{f_n(\mu_{N+1,n}\hat{w}_n)}{f_n(\mu_{N+1,n})}rdr.
\end{split}
\] 
Hence similarly we obtain \eqref{g2}. This finishes the proof.
\end{proof} 

By the previous lemma, we deduce the following.
\begin{lemma}\label{lem:g6} Suppose as in the previous lemma. In addition,  we assume that \eqref{f4} with $i=N$ is true. Then we get
\begin{equation}\label{g6}
\lim_{n\to \infty}\mu_{N+1,n}\left(\log{\frac1{r_{N,n}}}\right)^{\frac{1-\beta_n}{\beta_n}}=2^{\frac{\beta_*-1}{\beta_*}} \left\{\alpha\left(1-\frac{\beta_*}{2}\right)\right\}^{\frac{1}{\beta_*}},
\end{equation}
and 
\begin{equation}\label{g66}
\lim_{n\to \infty}\mu_{N+1,n}f_n(\mu_{N+1,n})^{\beta_n-1}\rho_{N+1,n}^{2(\beta_n-1)}=\frac{4^{\beta_*-1}\alpha\left(1-\frac{\beta_*}{2}\right)}{\la_*^{\beta_*-1}}.
\end{equation}
\end{lemma}
\begin{proof} Noting $\la_*\not=0$ by the previous lemma, we combine \eqref{g2} together with \eqref{f4} and  \eqref{f1} for $i=N$ and get 
\[
\begin{split}
\frac{2}{\la_*}+o(1)&=\frac{\rho_{N+1,n}^2f_n(\mu_{N+1,n}) \mu_{N,n}}{2+o(1)}\\
&=\frac{\rho_{N+1,n}^2f_n(\mu_{N+1,n})}{2+o(1)}\left(\frac{2\log{\frac1{r_{N,n}}}(1+o(1))}{\alpha\left(1-\beta_*/2\right)+o(1)}\right)^{\frac{1}{\beta_n}}.
\end{split}
\]
Then it holds that
\begin{equation}\label{q1}
\begin{split}
\rho_{N+1,n}^2&f_n(\mu_{N+1,n})\\
&=\left\{2^{2-\frac1{\beta_*}}\left(\alpha\left(1-\frac{\beta_*}{2}\right)\right)^{\frac1{\beta_*}}\la_*^{-1} +o(1)\right\}\left(\log{\frac1{r_{N,n}}}\right)^{-\frac1{\beta_n}}.
\end{split}
\end{equation}
Substituting this into \eqref{g1}, we obtain
\[
\begin{split}
\frac2{\la_* }+o(1)&=\frac1{\mu_{N+1,n}}\left\{2^{2-\frac1{\beta_*}}\left(\alpha\left(1-\frac{\beta_*}{2}\right)\right)^{\frac1{\beta_*}}\la_*^{-1} +o(1)\right\}\left(\log{\frac1{r_{N,n}}}\right)^{1-\frac1{\beta_n}}.
\end{split}
\]
It follows that
\[
\mu_{N+1,n}\left(\log{\frac1{r_{N,n}}}\right)^{\frac{1-\beta_n}{\beta_n}}=2^{\frac{\beta_*-1}{\beta_*}} \left\{\alpha\left(1-\frac{\beta_*}{2}\right)\right\}^{\frac{1}{\beta_*}}+o(1).
\]
This proves \eqref{g6}. Using this and \eqref{q1}, we get  
\[
f_n(\mu_{N+1,n})^{\beta_n-1}\rho_{N+1,n}^{2(\beta_n-1)}=\frac{1}{\mu_{N+1,n}}\left(\frac{4^{\beta_*-1}\alpha(1-\beta_*/2)}{\la_*^{\beta_*-1}}+o(1)\right).
\]
This shows \eqref{g66}. We finish the proof.
\end{proof}
\section{Proof of main theorems}\label{proof}

Let us complete the proof of main theorems. We shall first show Theorems \ref{a10} and \ref{a31}. Then Theorem \ref{a1} will readily follow from them.  We begin with the proof of  Theorem \ref{a10}. 

\begin{proof}[Proof of Theorem \ref{a10}] We first note that the standard argument shows that the weak limit $u_0$ of $(u_n)$ is a radially symmetric smooth solution of \eqref{p} with $(\la,\beta)=(\la_*,\beta_*)$. In particular, writing $u_0=u_0(|x|)$, we have that $u_0(r)$ ($r\in[0,1]$) satisfies 
\begin{equation}\label{wl}
-u_0''-\frac1ru_0'=\la_* f_*(u_0)\text{ in }(0,1)\text{ and }u_0'(0)=0=u_0(1).
\end{equation}

We begin with the case $\max_{r\in[0,1]}|u_n(r)|\to \infty$. Set $N:=\max\{i=1,\cdots,k+1\ |\ (\mu_{i,n})\text{ is unbounded}\}$. Suppose $N=k+1$. Then we have $\mu_{k+1}\to \infty$ as $n\to \infty$ by subtracting a subsequence if necessary. We shall confirm all the assertions in  the case of (i). By Lemma \ref{b20}, we get $\rho_{k+1,n}\to0$ and $\mu_{i,n}\to \infty$ for all $i=1,\cdots,k+1$.  Especially, the assumption \eqref{infty} is satisfied for $i=1$. As a consequence, Propositions  \ref{dd}, \ref{e00} and \ref{f0} hold true for $i=1$. It follows that all of the assumptions \eqref{infty}-\eqref{a7} are verified for $i=2$. Consequently,  the assertions in Propositions  \ref{dd}, \ref{e00} and \ref{f0} are true for $i=2$. Repeating the same argument, we ensure all the assertions in Propositions  \ref{dd}, \ref{e00} and \ref{f0} for any  $i=1,\cdots,k+1$. This completes the former assertions in (i). Finally, by \eqref{f1} with $i=k+1$ and the facts that $r_{k+1,n}=1$ and $\mu_{k+1,n}\to \infty$, we get $\la_*=0$. This yields that $u_0=0$ by \eqref{wl}. Moreover, Lemma \ref{RL} implies that $u_n$ is locally uniformly  bounded in $(0,1]$. Then using the equation for $u_n$ as usual, it is easy to  show that $u_n\to 0\text{ in }C^2_{\text{loc}}((0,1])$ up to a subsequence. This finishes the case of  (i).

Next we assume $N<k+1$. Then similarly to the previous argument, we get $\rho_{N,n}\to 0$ and $\mu_{i,n}\to \infty$ for all $i=1,\cdots,N$ up to a subsequence. Then, analogously, we have that all the assertions in Propositions \ref{dd}, \ref{e00} and \ref{f0} are true for all $i=1,\cdots,N$. On the other hand, by the definition of $N$, for each $i=N+1,\cdots,k+1$, there exists a value $\mu_i\ge0$ such that $\mu_{i,n}\to \mu_i$ up to a subsequence.  Then from Lemma \ref{b22}, we get numbers $0\le r_N\le \rho_{N+1} <r_{N+1}<\cdots<\rho_{k+1}< r_{k+1}=1$ if $N<k$ and $0\le r_N\le \rho_{N+1}<r_{N+1}=1$ if $N=k$ such that $r_{i,n}\to r_i$ for all $i=N,\cdots,k+1$ and $\rho_{i,n}\to \rho_i$ for all $i=N+1,\cdots,k+1$ by taking a subsequence again if necessary. Moreover, from Lemmas \ref{b21}, we get $\la_*\not=0$, and $r_N=\rho_{N+1}=0$. Furthermore, a usual argument shows that $u_n|_{[r_{N,n},1]}\to u_0$ in $C^2_{\text{loc}}((0,1])$ and $\lim_{r\to 0+0}(-1)^Nu_0(r)=\mu_{N+1}$ by Lemma \ref{lem:q1}. It follows that
\begin{equation}\label{q4}
\begin{split}
\int_{r_{N,n}}^1u_n'(r)^2rdr&=\int_{r_{N,n}}^1\la_nf_n(u_n)u_nrdr\\
&\to \int_{0}^1\la_*f_*(u_0)u_0rdr=\int_0^1u_0'(r)^2rdr
\end{split}
\end{equation}
by \eqref{wl}. This proves the former part of (ii). 

Now, we assume  $\mu_{N+1}>0$. Then noting \eqref{wl} and Lemmas \ref{lem:q1} and \ref{b23}, we get that $u_0(r_i)=0$, $(-1)^{i-1}u_0\ge0$ on $[r_{i-1},r_i]$, $u_0'(\rho_i)=0$ and $(-1)^{i-1}u_0(\rho_i)=\mu_i>0$ for all $i=N+1,\cdots,k+1$. Moreover, by \eqref{wl}, we readily see $(-1)^{i-1}u_i>0$ on $(r_{i-1},r_i)$ for all $i=N+1,\cdots,k+1$. This completes the case of (a). Next we suppose $\mu_{N+1}=0$. By Lemma \ref{b23}, it is obvious that $u_0=0$. Put $w_n:=u_n|_{[r_{N,n},1]}/\mu_{N+1,n}$ on $[r_{N,n},1]$. By Lemma \ref{b23} again, for every $i=N+1,\dots,k+1$, we have a constant $\mu_i^*>0$ such that  $\max_{r\in[r_{i-1,n},r_{i,n}]}w_n(r)=\mu_{i,n}/\mu_{N+1,n}\to \mu_i^*$  up to a subsequence. In particular, $w_n$ is uniformly bounded in $[r_{N,n},1]$. Then, by the standard argument and Lemma \ref{lem:q1}, we get a continuous function $w_0$ in $[0,1]$ such that $w_n\to w_0$ in $C^2_{\text{loc}}((0,1])$ and
\[
\begin{cases}
-w_0''-\frac1r w_0'=\la_* w_0\text{ in }(0,1),\\
(-1)^Nw_0(0)=1,\ w_0(r_i)=0,\\
(-1)^{i-1}w_0>0 \text{ on }(r_{i-1},r_i)\ (i=N+1,\cdots,k+1).
\end{cases}
\]
Using the equation and the condition $(-1)^Nw_0(0)=1$, we obtain $w_0(r)=(-1)^N J_0(\sqrt{\la_*}r)$ in $[0,1]$ where $J_0$ is the first kind Bessel function of order zero  defined in Section \ref{org}. Moreover, since $w_0$ has just $(k-N)$ interior zero points in $(0,1)$ and $w_0(1)=0$, we get that $\sqrt{\la_*}$ coincides with  the $(k-N+1)$--th zero point of $J_0$ on $(0,\infty)$, i.e., $\sqrt{\la_*}=t_{N-k+1}$. It follows that $w_0=(-1)^N\varphi_{k-N+1}$ and $\la_*=\Lambda_{k-N+1}$.  This completes the case of (ii). 

 Finally, if $u_n$ is uniformly bounded in $[0,1]$, repeating the similar (and simpler) argument based on Lemmas \ref{b20}-\ref{b23} as above, we can confirm all the assertions in (iii). This finishes the proof. 

\end{proof}

Next we prove Theorem \ref{a31}.

\begin{proof}[Proof of Theorem \ref{a31}]
 We first assume that $k>0$, $(\beta_n)\subset(0,1]$, and $\mu_{1,n}\to \infty$. Then we claim that for every $i=2,\cdots,k+1$, $\mu_{i,n}$ is bounded uniformly  for all $n\in \mathbb{N}$. To see this, we shall show that $\mu_{2,n}$ is uniformly bounded. Otherwise, we get $\mu_{2,n}\to \infty$ up to a subsequence. Then arguing as in the previous proof, we ensure  that all the assumptions \eqref{infty}-\eqref{a7} are satisfied for $i=2$. Then we get \eqref{d8} and \eqref{e1} for $i=2$ by Lemmas \ref{lem:d8} and  \ref{lem:e} respectively. But if $\beta_n\le1$ for all $n\in \mathbb{N}$, \eqref{d8} and \eqref{e1} with $i=2$ yield that $\mu_{2,n}$ is uniformly bounded. This contradicts \eqref{infty}. Consequently, Lemma \ref{b22} proves the claim. 
  
Now, we assume that (i) of Theorem \ref{a10} occurs. Then the first conclusion follows by the previous claim. Moreover, if $k\in \mathbb{N}\cup\{0\}$, arguing as in the previous proof again, we get that all  the assertions in Lemma \ref{lem:e} and Proposition \ref{f0} hold true for any $i=1,\cdots,k+1$. It follows from \eqref{f1} with $i=k+1$ that \eqref{aaa1} holds true. Then if $k\ge1$, it follows from \eqref{f11} that
\[
\frac{\left\{\alpha\left(1-\frac{\beta_*}2\right)+o(1)\right\}^{-\frac1{\beta_n}}\left(\log{\frac1{\la_n}}\right)^{\frac1{\beta_n}}}{\mu_{k,n}^{(\beta_n-1)}}=\alpha\left(1-\frac{\beta_*}{2}\right)+o(1).
\] 
This gives \eqref{aa1} with $i=k$ after an easy calculation. Then we get \eqref{aa1} for all $i=1,\cdots,k$ by induction. In fact, we assume \eqref{aa1} is true for some $i=j\in \{2,\cdots,k\}$. Then using \eqref{f11} with $i=j$, we similarly get 
\[
\begin{split}
&\frac{\left[\left\{\alpha \left(1-\frac{\beta_*}{2}\right)\right\}^{\frac{2-\beta_*(\beta_*-1)^{k-j+1}}{2-\beta_*}}+o(1)\right]^{-\frac1{\beta_n(\beta_n-1)^{k-j+1}}}\left(\log{\frac1{\la_n}}\right)^{\frac1{\beta_n(\beta_n-1)^{k-j+1}}}}{\mu_{j-1,n}^{\beta_n-1}}\\
&=\alpha \left(1-\frac{\beta_*}{2}\right)+o(1).
\end{split}
\]
It follows that
\[
\lim_{n\to \infty}\frac{\log{\frac1{\la_n}}}{\mu_{j-1,n}^{\beta_n(\beta_n-1)^{k-j+2}}}=\left\{\alpha \left(1-\frac{\beta_*}{2}\right)\right\}^{\frac{2-\beta_*(\beta_*-1)^{k-j+2}}{2-\beta_*}}.
\]
This is \eqref{aa1} with $i=j-1$. This shows the desired conclusion. Moreover, since \eqref{f1} and \eqref{f4} with $i=k+1$ imply 
\[
\mu_{k+1,n}=\left(\frac{\log{\frac1{\la_n}}}{\alpha(1-\beta_*/2)+o(1)}\right)^{\frac1{\beta_n}},
\]
and 
\[
\mu_{k+1,n}|u_{k+1,n}'(1)|=2+o(1)
\]
respectively, combining these two formulas, we get \eqref{aa44}. Next assume $k\ge1$ and  $i=1,\cdots,k$. Noting the first conclusion, we may assume $\beta_n>1$ for all $n\in \mathbb{N}$. Then by \eqref{aa1}, we have that 
\[
\frac{\log{\frac{1}{\la_n}}}{\mu_{i,n}^{\beta_n}}=\frac{\log{\frac{1}{\la_n}}}{\mu_{i,n}^{\beta_n(\beta_n-1)^{k-i+1}}}\mu_{i,n}^{\beta_n\{(\beta_n-1)^{k-i+1}-1\}}\to0
\]
since $1\le k-i+1\le k$ and $1\le\beta_*<2$. Using this and \eqref{f1}, we obtain
\[
\frac{\alpha}{2}\left(1-\frac{\beta_*}{2}\right)+o(1)=\frac{\log{\frac1{r_{i,n}}}}{\mu_{i,n}^{\beta_n}}.
\]
Therefore, it follows from \eqref{aa1} that
\[
\begin{split}
&\left\{\frac{\alpha}{2}\left(1-\frac{\beta_*}{2}\right)+o(1)\right\}^{(\beta_n-1)^{k-i+1}}\\
&\ \ \ \ \ \ \ \ \ \ \ \ \ =\frac{\left(\log{\frac1{r_{i,n}}}\right)^{(\beta_n-1)^{k-i+1}}}{\log{\frac{1}{\la_n}}}\left[\left\{\alpha \left(1-\frac{\beta_*}{2}\right)\right\}^{\frac{2-\beta_*(\beta_*-1)^{k-i+1}}{2-\beta_*}}+o(1)\right].
\end{split}
\]
Using this formula, we readily get \eqref{aa2}. It follows that
\begin{equation}\label{gg1}
\frac{\log{ \frac1{r_{i,n}}}}{\log{\frac1{\la_n}}}=\frac{\left(\log{\frac1{r_{i,n}}}\right)^{(\beta_n-1)^{k-i+1}}}{\log{\frac1{\la_n}}}\left(\log{\frac1{r_{i,n}}}\right)^{1-(\beta_n-1)^{k-i+1}}\to \infty
\end{equation}
as $n\to \infty$.  Then we get by \eqref{f4}, \eqref{f1} and \eqref{gg1} that,
\[
\begin{split}
\log{|u_{i,n}'(r_{i,n})|}&=\log{\frac{1}{r_{i,n}}}-\log{\mu_{i,n}}+O(1)\\
&=\log{\frac{1}{r_{i,n}}}-\frac{1}{\beta_n}\log{\log{\frac1{r_{i,n}}}}-\frac{1}{\beta_n}\log{\left(1+\frac{\log{\frac1{\la_n}}}{\log{\frac{1}{r_{i,n}^2}}}\right)}+O(1)\\
&=\log{\frac{1}{r_{i,n}}}(1+o(1))\\
&=\left(\frac{\log{\frac1{\la_n}}}{2^{(\beta_*-1)^{k-i+1}}\left\{\alpha \left(1-\frac{\beta_*}{2}\right)\right\}^{\frac{2-2(\beta_*-1)^{k-i+1}}{2-\beta_*}}+o(1)}\right)^{\frac1{(\beta_n-1)^{k-i+1}}}\\
&\ \ \ \times(1+o(1))
\end{split}
\]
by \eqref{aa2}. This proves \eqref{aa4}. Next, for any $i=2,\cdots,k+1$, from \eqref{e1} and the definition of $\ga_{i,n}$, we get 
\begin{equation}\label{q9}
\begin{split}
\log{\left(8^{\beta_*-1}\alpha(1-\beta_*/2)+o(1)\right)}&=(\beta_n-1)\mu_{i,n}^2\left(1+o(1)\right)+\beta_n \log{\mu_{i,n}}\\
&-2(\beta_n-1)\log{\frac1{\rho_{i,n}}}
\end{split}
\end{equation}
where we noted
\[
\frac{\log{\frac1{\la_n}}}{\mu_{i,n}^2}\le\frac{\log{\frac1{\la_nr_{i,n}^2}}}{\mu_{i,n}^{\beta_n}}\frac1{\mu_{i,n}^{2-\beta_n}}\to0
\]
as $n\to \infty$ by \eqref{f1}. Now, we suppose $k\ge2$ and $i=2,\cdots,k$. Then we have by \eqref{aa1} that , 
\[
\begin{split}
&\frac{\log{\mu_{i,n}}}{(\beta_n-1)\mu_{i,n}^2}\\
&=\frac{(1+o(1))\log{\log{\frac1{\la_n}}}}{\beta_n(\beta_n-1)^{k-i+2}\left(\frac{\log{\frac1{\la_n}}}{(\alpha(1-\beta_*/2))^{(2-\beta_*(\beta_*-1)^{k-i+1})/(2-\beta_*)}+o(1)}\right)^{\frac2{\beta_n(\beta_n-1)^{k-i+1}}}}.
\end{split}
\]
Since $i<k+1$, we get $(\log{\mu_{i,n}})/((\beta_n-1)\mu_{i,n}^2)\to0$. Therefore, we obtain  from \eqref{q9} and \eqref{aa1} that 
\[
\begin{split}
\frac12+o(1)&=\frac{\log{\frac1{\rho_{i,n}}}}{\mu_{i,n}^2}\\
&=\left(\frac{(\alpha(1-\beta_*/2))^{(2-\beta_*(\beta_*-1)^{k-i+1})/(2-\beta_*)}+o(1)}{\log{\frac1{\la_n}}}\right)^{\frac{2}{\beta_n(\beta_n-1)^{k-i+1}}}\\
&\ \ \ \times \log{\frac1{\rho_{i,n}}}.
\end{split}
\]
This proves \eqref{aa3}. On the other hand, if $k\ge1$ and $i=k+1$, we use \eqref{aaa1} and see
\[
\frac{\log{\mu_{k+1,n}}}{(\beta_n-1)\mu_{k+1,n}^2}= \left\{\alpha\left(1-\frac{\beta_*}{2}\right)+o(1)\right\}^{\frac2{\beta_n}} \frac{\log{\log{\frac1{\la_n}}}}{\beta_n(\beta_n-1)\left(\log{\frac1{\la_n}}\right)^{\frac2{\beta_n}}}.
\]
Hence if there exists a constant $L\ge0$ such that
\[
\frac{\log{\log{\frac1{\la_n}}}}{(\beta_n-1)\left(\log{\frac1{\la_n}}\right)^{\frac2{\beta_n}}}=L,
\]
we get from \eqref{q9} and \eqref{aaa1} that
\[
\begin{split}
1+\left\{\alpha\left(1-\frac{\beta_*}{2}\right)\right\}^{\frac2{\beta_*}} L+o(1)&=\frac{2\log{\frac1{\rho_{k+1,n}}}}{\mu_{k+1,n}^2}\\
&=2\left(\frac{\alpha(1-\beta_*/2)+o(1)}{\log{\frac1{\la_n}}}\right)^{\frac{2}{\beta_n}}\log{\frac1{\rho_{k+1,n}}}.
\end{split}
\]
This ensures \eqref{ab1}. On the other hand, if 
\[
\frac{\log{\log{\frac1{\la_n}}}}{(\beta_n-1)\left(\log{\frac1{\la_n}}\right)^{\frac2{\beta_n}}}\to \infty,
\]
we necessarily have $\beta_*=1$ and then, by \eqref{q9} and \eqref{aaa1}, we get
\[
1+o(1)=\frac{2(\beta_n-1)\log{\frac1{\rho_{k+1,n}}}}{\beta_n \log{\mu_{k+1,n}}}=\frac{2(\beta_n-1)\log{\frac1{\rho_{k+1,n}}}}{ (1+o(1))\log{\log{\frac1{\la_n}}}}.
\]
This proves \eqref{ab2}. This completes the case of (i). 

Next we assume that (ii) of Theorem \ref{a10} happens. Then, since $\mu_{N+1,n}$ is uniformly bounded and $r_{N,n}\to0$, we have $\beta_*\le1$ by \eqref{g6} in Lemma \ref{lem:g6}. Then, the first claim above completes the first assertion in the case of (ii). Let us suppose $\mu_{N+1}>0$ and complete the case of (a). Again by \eqref{lem:g6}, we see $\beta_*=1$. We shall prove \eqref{aa5}--\eqref{aa8}. Multiplying the equation in \eqref{rad} with $i=N+1$ by $r$ and integrating over $(\rho_{N+1,n},r_{N+1,n})$, we get
\[
r_{N+1,n}u_{N+1,n}'(1)=-\int_{\rho_{N+1,n}}^{r_{N+1,n}}\la_n f_n(u_{N+1,n})rdr.
\]
Then recalling the assertions in (ii) of Theorem \ref{a10}, we ensure \eqref{aa10}. Moreover, noting $\beta_*=1$, we get by \eqref{g66} that
\[
\rho_{N+1,n}^{\beta_n-1}=\sqrt{\frac{\alpha}{2\mu_{N+1}}}+o(1).
\]
This gives \eqref{aa9}. Next we see by \eqref{g6} that  
\begin{equation}\label{g3}
\left(\log{\frac1{r_{N,n}}}\right)^{\beta_n-1}=\frac{2\mu_{N+1}}{\alpha}+o(1).
\end{equation} 
Then we use  \eqref{f1} with $i=N$ and \eqref{g3} to obtain
\[
\mu_{N,n}^{\beta_n-1}=(1+o(1))\left(\log{\frac{1}{r_{N,n}}}\right)^{\frac{\beta_n-1}{\beta_n}}=\frac{2\mu_{N+1}}{\alpha}+o(1).
\]
This gives \eqref{aa5} with $i=N$. Then if $N\ge2$, \eqref{aa5} is true for all $i=1,\cdots,N$ by induction. Indeed, assuming \eqref{aa5} is true for some $i=j\in \{2,\cdots,N\}$, we use \eqref{f11} with $i=j$ to see
\[
\mu_{j-1,n}^{\beta_n-1}=\left(\frac2\alpha+o(1)\right)\mu_{j,n}.
\]
This and the assumption suggest
\[
\mu_{j-1,n}^{(\beta_n-1)^{N-j+2}}=(1+o(1))\mu_{j,n}^{(\beta_n-1)^{N-j+1}}=\frac{2\mu_{N+1}}{\alpha}+o(1).
\]
This shows \eqref{aa5} with $i=j-1$. This finishes  \eqref{aa5}. This and \eqref{f1} show
\[
\left(\log{\frac{1}{r_{i,n}}}\right)^{(\beta_n-1)^{N-i+1}}=(1+o(1))\mu_{i,n}^{\beta_n(\beta_n-1)^{N-i+1}}=\frac{2\mu_{N+1}}{\alpha}+o(1)
\]
for all $i=1,\cdots,N$. This proves \eqref{aa6}. Moreover, for any $i=1,\cdots,N$, we get by \eqref{f4} and \eqref{f1}  that
\[
\begin{split}
\log{|u_{i,n}'(r_{i,n})|}&=\log{\frac1{r_{i,n}}}\left(1-\frac{\log{\mu_{i,n}}}{\log{\frac1{r_{i,n}}}}+o(1)\right)=\log{\frac1{r_{i,n}}}\left(1+o(1)\right),
\end{split}
\]
where we noted
\[
\frac{\log{\mu_{i,n}}}{\log{\frac1{r_{i,n}}}}=\frac{\mu_{i,n}^{\beta_n}}{\frac12(1+o(1))\log{\frac1{\la_n r_{i,n}^2}}}\frac{\log{\mu_{i,n}}}{\mu_{i,n}^{\beta_n}}\to0
\] 
by \eqref{f1}. Then it follows from \eqref{aa6} that 
\[
\left(\log{|u_{i,n}'(r_{i,n})|}\right)^{(\beta_n-1)^{N-i+1}}=\frac{2\mu_{N+1}}{\alpha}(1+o(1)).
\]
This proves \eqref{aa7}.  In particular, \eqref{aa5} clearly shows that $2\mu_{N+1}/\alpha>1$ ($\in (0,1)$) yields $\beta_n>1$ ($<1$ respectively) for all $n\in \mathbb{N}$.  On the other hand,  $\beta_n>1$ ($=1$, $<1$) for all $n\in \mathbb{N}$ suggests $2\mu_{N+1}/\alpha\ge 1$ ($=1$, $\le1$ respectively). Finally, suppose $1<N\le k$ and $2\mu_{N+1}/\alpha>1$. Then we have $\beta_n>1$ for all $n\in \mathbb{N}$ by the first claim above.  Then, for any $i=2,\dots,N$, we use \eqref{e1} and the definition of $\ga_{i,n}$ to deduce 
\[
\mu_{i,n}f_n(\mu_{i,n})^{\beta_n-1}\rho_{i,n}^{2(\beta_n-1)}=\frac{\alpha}{2}+o(1).
\]
It follows that 
\begin{equation}\label{q6}
\log{\frac1{\rho_{i,n}}}=\frac{\mu_{i,n}^2}{2}\left(1+\frac{\log{\mu_{i,n}}}{\mu_{i,n}^2(\beta_n-1)}(1+o(1))+o(1)\right).
\end{equation}
Here note that \eqref{aa5} implies 
\[
\frac{\log{\mu_{i,n}}}{\mu_{i,n}^2(\beta_n-1)}=\frac{\log{(2\mu_{N+1}/\alpha+o(1))}}{(\beta_n-1)^{N-i+2}(2\mu_{N+1}/\alpha+o(1))^{2/(\beta_n-1)^{N-i+1}}}\to 0
\]
since $2\mu_{N+1}/\alpha>1$. Consequently, \eqref{q6} and \eqref{aa5} ensure
\[
\left(\log{\frac1{\rho_{i,n}}}\right)^{(\beta_n-1)^{N-i+1}}=\mu_{i,n}^{2(\beta_n-1)^{N-i+1}}\left(1+o(1)\right)=\left(\frac{2\mu_{N+1}}{\alpha}\right)^2+o(1).
\]
This gives \eqref{aa8}.  This completes the case of $(a)$.

Lastly, if $\mu_{N+1}=0$, Lemma \ref{b23} shows $\mu_i=0$ for all $i=N+1,\cdots,k+1$. Moreover, \eqref{g6} confirms that $\beta_n<1$ for all $n\in \mathbb{N}$ since $r_{N,n}\to0$. This completes the case of (b). This finishes the proof. 

\end{proof}

Let us complete the proof of Theorem \ref{a1}. 

\begin{proof}[Proof of Theorem \ref{a1}] 
Assume $\max_{n\to \infty}|u_n(x)|\to \infty$. Then writing $u_n=u_n(|x|)$, the function $u_n(r)$ $(r\in[0,1])$ verifies (i) or (ii) of Theorem \ref{a10}. If (i) occurs and $\beta_n>1$ for all $n\in \mathbb{N}$, the assertions in (i) of Theorem \ref{a10} completes (i) of Theorem \ref{a1}. On the other hand, if (i) of Theorem \ref{a10} happens and $\beta_n\le 1$ for all $n\in \mathbb{N}$, by the first conclusion in Theorems \ref{a31}, we get $k=0$. This shows (ii) of Theorem \ref{a1}. Next, we suppose (ii) of Theorem \ref{a10} occurs. Then if $\beta_n>1$ for all $n\in \mathbb{N}$, we have $\beta_*=1$ and $(-1)^Nu_0(0)\ge \alpha/2$ by Theorem \ref{a31}. This implies (iii) of Theorem \ref{a1} occurs. If $\beta_n=1$ for all $n\in \mathbb{N}$, we have $N=1$ and $-u_0(0)=\alpha/2$ by Theorem \ref{a31}. This is (iv) of Theorem \ref{a1}. Lastly, if $\beta_n < 1$ for all $n\in \mathbb{N}$, we have two cases. The first case is $u_0\not=0$. Since this case corresponds to (a) of Theorems \ref{a10} and \ref{a31}, we have $u_0(0)\not=0$. Then, we see $\beta_n\uparrow 1$, $N=1$ and $-u_0(0)\in(0,\alpha/2]$ by Theorem \ref{a31}. This confirms  (v) of Theorem \ref{a1}. The second case is $u_0=0$. This suggests that (b) of Theorems \ref{a10} and \ref{a31} occurs. Hence we have $\la_*=\Lambda_{k}$ since $N=1$ by Theorem \ref{a31}. This shows that (vi) occurs. Finally, if $u_n$ is uniformly bounded, we get (iii) of Theorem \ref{a10}. This completes (vii) of Theorem \ref{a1}. This finishes the proof. 

\end{proof}

Corollary \ref{cor:ne} immediately follows from Theorem \ref{a1}.

\begin{proof}[Proof of Corollary \ref{cor:ne}]
We assume that there exist such sequences of values $\{(\la_n,\beta_n)\}\subset (0,\infty)\times (0,1]$ and nodal radial solutions $(u_n)$ on the contrary. Then, in view of the fact that $\la_n\to0$ and $\beta_n\le1$ for all $n\in \mathbb{N}$, we have that (ii) of Theorem \ref{a1} occurs. But then we get $k=0$ which is a contradiction. This finishes the proof.  

\end{proof}

Next, we shall prove Corollary \ref{cor:final}. We recall  Lemma 2.1 in \cite{GN} with slight generalization.
\begin{lemma}\label{crl1}
Assume $k\in \mathbb{N}\cup\{0\}$, $\{(\la_n,\beta_n)\}\subset (0,\Lambda_1)\times (0,2)$, and $(\la_n,\beta_n)\to (\la_*,\beta_*)\in(0,\Lambda_1)\times (0,2)$. 
 Then we have 
\[
\limsup_{n\to \infty}c_{k,\la_n,\beta_n}\le 2\pi k+c_{0,\la_*,\beta_*},
\]
where the number $c_{k,\la,\beta}$ is defined as in Section \ref{main}. 
\end{lemma}
\begin{proof} Noting $\la_*\not=0$ and $\beta_*>0$, we can repeat the completely same argument with the proof of Lemma 2.1 in \cite{GN}. This ensures the proof.
\end{proof}

Using this, we give the proof.  

\begin{proof}[Proof of Corollary \ref{cor:final}]
We first assume that the first conclusion does not hold. Then we have sequences of positive values $(\la_n)$, natural numbers $(k_n)$ and nodal radial solutions $(u_n)$ such that $\la_n\to0$ as $n\to \infty$, $u_n\in S_{k_n,\la_n,\beta_n}$, $u_n(0)>0$, and $\int_B|\nabla u_n|^2dx$ is uniformly  bounded for all $n\in \mathbb{N}$. Then, we claim that, up to a suitable subsequence, there exists a number $k\in \mathbb{N}$ such that $u_n\in S_{k,\la_n,\beta}$ for all $n\in \mathbb{N}$. Otherwise, we get $k_n\to \infty$ as $n\to \infty$. Then choose numbers $0=r_{0,n}<r_{1,n}<\cdots<r_{k+1,n}=1$ so that $u_n(x)=0$ if $|x|=r_{i,n}$ and $(-1)^{i-1}u_n(x)>0$ if $r_{i-1,n}<|x|<r_{i,n}$ for all $i=1,\cdots,k+1$. Moreover,   for all $i=1,\cdots,k+1$, define a function $u_{i,n}\in \mathcal{N}_{\la_n,\beta}$ by $u_{i,n}:=u_n|_{\{r_{i-1,n}<|x|<r_{i,n}\}}$ with zero extension to whole $B$. Then since $\la_n\to 0$, Lemma \ref{b1} implies that there exists a constant $K>0$ such that 
\[
\int_{B}|\nabla u_n|^2dx=\sum_{i=1}^{k+1}\int_{B}|\nabla u_{i,n}|^2dx\ge (k_n+1)K
\]
for all $n \in \mathbb{N}$. Since the right hand side diverges to infinity, we get  a contradiction. This proves the claim. But, then the existence of such sequence $(u_n)$ contradicts Corollary \ref{cor:ne} since $\beta\le1$. This proves the first assertion. Next we suppose the latter conclusion fails on the contrary. Then there exists a number $k\in \mathbb{N}$ and sequences of positive values $(\la_n)$ and solutions $(u_n)$ such that $\la_n\to0$, $u_n\in S_{k,\la_n,\beta}$, and $I_{\la_n,\beta}(u_n)=c_{k,\la_n,\beta}$ for all $n\in\mathbb{N}$. In addition, for any $k\in \mathbb{N}$ and $\la\in(0,\Lambda_1)$, it holds that $c_{k,\la,\beta}\le 2\pi k+c_{0,\la,\beta}<2\pi(k+1)$. In fact, the first inequality is obtained by just choosing $\la_n=\la$ and $\beta_n=\beta$ for all $n\in \mathbb{N}$ in Lemma \ref{crl1} and the second one comes from the fact that $c_{0,\la,\beta}<2\pi$ by \cite{A}. In particular, we get $I_{\la_n,\beta}(u_n)<2\pi(k+1)$ for all $n\in \mathbb{N}$. Consequently, the standard argument shows that $(u_n)$ is bounded in $H^1_0(B)$. But this is again impossible in view of Corollary \ref{cor:ne} since $\beta\le1$. This completes the proof.

\end{proof}

Finally we prove Corollary \ref{expansion}.

\begin{proof}[Proof of Corollary \ref{expansion}] Assume as in the corollary. We write $u_n=u_n(|x|)$ $(x\in \overline{B})$ and consider the function $u_n(r)$ $(r\in[0,1])$. Then we get all the assertions in (i) of Theorems \ref{a10} and \ref{a31}. It follows from \eqref{aaa1} that  
\[
\mu_{k+1,n}=\left(\frac{\log{\frac1{\la_n}}}{\alpha\left(1-\frac{\beta_*}{2}\right)+o(1)}\right)^{\frac1{\beta_n}}.
\]
Moreover, if $k\ge1$, we also have by \eqref{aa1} (or \eqref{d77}) that $\mu_{k+1,n}/\mu_{i,n}\to 0$ as $n\to \infty$ for all $i=1,\cdots,k$. Then from Theorem \ref{a10}, we derive
\[
\begin{split}
\int_0^1u_n'(r)^2rdr&=\sum_{i=1}^{k+1}\int_{r_{i-1,n}}^{r_{i,n}}u_{i,n}'(r)^2rdr=2(k+1)-\frac{\alpha \beta_*}{\mu_{k+1,n}^{2-\beta_n}}+o\left(\frac{1}{\mu_{k+1,n}^{2-\beta_n}}\right)\\
&=2(k+1)-\frac{\alpha^{\frac{2}{\beta_*}}\beta_*\left(1-\frac{\beta_*}{2}\right)^{\frac{2-\beta_*}{\beta_*}}}{\left(\log\frac1{\la_n}\right)^{\frac{2-\beta_n}{\beta_n}}}+o\left(\frac1{{\left(\log\frac1{\la_n}\right)^{\frac{2-\beta_n}{\beta_n}}}}\right).
\end{split}
\]
This finishes the proof.

\end{proof}

\section{Counterparts}\label{sec:cor}
In this final section, we discuss the counterparts of our classification result, Theorem \ref{a1}. In the following, we suppose $k\in \{0\}\cup \mathbb{N}$. Then we first remark that for any sequences $(\la_n)\subset (0,\Lambda_1)$ and $(\beta_n)\subset (1,2)$ ($(0,2)$ if $k=0$) of values, there exists a sequence $(u_n)$ of radial solutions which satisfies the assumptions in the theorem. To see this,  for any such sequences $(\la_n)$ and $(\beta_n)$, we define the sequence $(u_n)$ of solutions so that $u_n\in S_{k,\la_n,\beta_n}$ and $I_{\la_n,\beta_n}(u_n)=c_{k,\la_n,\beta_n}$ for all $n\in \mathbb{N}$ where $c_{k,\la,\beta}$ is the number defined in Section \ref{main}. This choice is possible by \cite{A} and \cite{AY1}. Consequently, since $I_{\la_n,\beta_n}(u_n)<2\pi k+c_{0,\la_n,\beta_n}<2\pi(k+1)$ for any $n\in \mathbb{N}$, which is proved in the same papers, a standard argument shows that $(u_n)$ is bounded in $H^1_0(B)$. Hence $(u_n)$ satisfies all the assumptions in the theorem.

Then we can immediately show some easy examples with this sequence $(u_n)$. Indeed, let us suppose $(\la_n,\beta_n)\to (\la_*,\beta_*)$ for  $(\la_*,\beta_*)\in[0,\Lambda_1)\times(0,2)$ if $k=0$ and $(\la_*,\beta_*)\in[0,\Lambda_1)\times[1,3/2)$ if $k\ge1$. Then, it follows from Theorem \ref{a1} that if $\la_*=0$ and $\beta_n>1$ ($\beta_n\le1$) for all $n\in \mathbb{N}$, then  $(u_n)$ behaves as in (i) ((ii) respectively) of the theorem. On the other hand, if $k=0$ and $\la_*\not=0$,  or $k\ge1$, $\la_*\not=0$, and $\beta_*>1$, then $(u_n)$ behaves as in (vii) with $u_0\not=0$. 

We shall find more  examples for (iii) and (vii). To this end, recalling $S_{0,\la,1}\not=\emptyset$ if and only if $\la\in(0,\Lambda_1)$ (\cite{A}),  we define
\[
\begin{split}
\Lambda^*:=\inf\{\Lambda\in(0,\Lambda_1)\ |\ u(0)<\alpha/2\text{ for any }u\in S_{0,\la,1}\text{ with }&I_{\la,1}(u)=c_{0,\la,1}\\
&\text{ if } \la\in(\Lambda,\Lambda_1)\}.
\end{split}
\]
It follows that $\Lambda^*\in(0,\Lambda_1)$. (See Lemma \ref{posi} below.) On the other hand, noting the nonexistence result by \cite{AY2}, we define 
\[
\begin{split}
\Lambda_*:=\inf\{\Lambda>0\ |\ S_{1,\Lambda,1}\not=\emptyset\}(=\sup\{\Lambda>0\ |\ S_{k,\Lambda,1}=\emptyset \text{ for any $k\in \mathbb{N}$}\}),
\end{split}
\]
and get $\Lambda_*>0$. Moreover, our necessary condition on the weak limit in (iii) of Theorem \ref{a1} allows to see $\Lambda_*\le\Lambda^*$ as follows.
\begin{corollary}\label{cr:0}
Assume $\{(\la_n,\beta_n)\}\subset (0,\Lambda_1)\times (1,2)$ and let $(u_n)$ be a sequence of solutions such that $u_n\in S_{1,\la_n,\beta_n}$ and $I_{\la_n,\beta_n}(u_n)=c_{1,\la_n,\beta_n}$ for all $n\in \mathbb{N}$. Moreover, suppose $(\la_n,\beta_n)\to (\la_*,\beta_*)\in(0,\Lambda_1)\times\{1\}$ and $|u(0)|<\alpha/2$ for all $u\in S_{0,\la_*,1}$ with $I_{\la_*,1}(u)=c_{0,\la_*,1}$ (which is verified if $\la_*\in(\Lambda^*,\Lambda_1)$). Then $(u_n)$ behaves as in (vii) of Theorem \ref{a1} with $u_0\not=0$. In particular, there exists at least one pair of solutions $u^{\pm}\in S_{1,\la,1}$ such that $u^-(0)<0<u^+(0)$, $u^+=-u^-$, and $I_{\la,1}(u^{\pm})\le 2\pi+c_{0,\la,\beta}$, for all $\la\in (\Lambda^*,\Lambda_1)$ and thus, it holds that $0<\Lambda_*\le\Lambda^*<\Lambda_1$. 
\end{corollary}
\begin{remark}
The latter assertion is not covered by Theorem 1.3 in \cite{AY1} since the nonlinearity $f(t)=ts^{t^2+\alpha|t|}$ does not satisfy (2) of Theorem 1.2 there. 
\end{remark}

Then we can give the next result. Notice that by the previous corollary and the argument in the first paragraph of this section, we  ensure the existence of a sequence satisfying each assumption of (a)-(d) below. 
\begin{proposition}\label{cr:1} Let $k\in\{0\}\cup \mathbb{N}$ and $\{(\la_n,\beta_n)\}\subset  (0,\Lambda_1)\times (0,2)$ and suppose $(\la_n,\beta_n)\to(\la_*,\beta_*)\in (0,\Lambda_1]\times (0,2)$. Moreover, we assume $(u_n)$ is a sequence of solutions such that $u_n\in S_{k,\la_n,\beta_n}$ and $\int_B|\nabla u_n|^2dx$ is uniformly bounded for all $n\in \mathbb{N}$. 
 Then we have the next assertion.
\begin{enumerate} 
\item[(a)] Let $\la_*=\Lambda_1$. Then, if $k=0$, (vii) of Theorem \ref{a1}  occurs with $u_0=0$. Moreover, assume $\beta_n\ge1$ for all $n\in \mathbb{N}$. Then if $k\ge1$ and $\beta_*\in(1,3/2)$, or $k=1$ and $\beta_*=1$, we have that $(u_n)$ behaves as in (vii) with $u_0\not=0$. 
\end{enumerate}
Moreover, we assume $k\ge1$, $\beta_n>1$ for all $n\in \mathbb{N}$, $\beta_*=1$ and $\la_*\in(0,\Lambda_1)$. Then we get the following. 
\begin{enumerate}
\item[(b)] If $S_{\kappa,\la_*,1}=\emptyset$ for all $1\le \kappa\le k$, (which is satisfied if $\la_*\in(0,\Lambda_*)$,) then $(u_n)$ behaves as in (iii) of Theorem \ref{a1} with $N=k$.
\end{enumerate}
In addition, we suppose $I_{\la_n,\beta_n}(u_n)=c_{k,\la_n,\beta_n}$ for any $n\in \mathbb{N}$. Then we obtain the following.
\begin{enumerate}
\item[(c)] For any $\la_*\in[\Lambda_*,\Lambda_1)$, there exists a natural number $k_{\la_*}$ such that if $k\ge k_{\la_*}$, $(u_n)$ behaves as in (iii) of Theorem \ref{a1} with some natural number  $k-k_{\la_*}< N\le k$. Moreover, we have $k_{\la_*}\ge2$ if $\la_*> \Lambda^*$.  
\item[(d)] Assume that any solution $u\in S_{0,\la_*,1}$ with $I_{\la_*,1}(u)=c_{0,\la_*,1}$ satisfies $|u(0)|<\alpha/2$ (which is verified if $\la_*\in(\Lambda^*,\Lambda_1)$). Then if $k\ge2$ and (iii) of Theorem \ref{a1} holds true, then $N\not =k$. In particular, for any $\la_*\in(\Lambda^*,\Lambda_1)$, chossing the number $k_{\la_*}\ge2$ from (c) above, we get that for all $k\ge k_{\la_*}$, $(u_n)$ behaves as in (iii) of Theorem \ref{a1} with some number $k-k_{\la_*}< N<k$.
\end{enumerate}
\end{proposition}

From the conclusion in (a), we get an additional existence result for $\la=\Lambda_1$ as follows.
\begin{corollary}\label{cr:2}
Let $\beta\in[1,3/2)$. Then  for any $k\in \mathbb{N}$ if $\beta>1$ and for $k=1$ if $\beta=1$, there exists at least one pair of solutions $u_{k,\beta}^{\pm}\in S_{k,\Lambda_1,\beta}$ such that $u_{k,\beta}^-(0)<0<u_{k,\beta}^+(0)$, $u_{k,\beta}^+=-u_{k,\beta}^-$ and $I_{\Lambda_1,\beta}(u_{k,\beta}^\pm)\le 2\pi k$. Moreover, choosing a suitable sequence  $(\beta_n)\subset (1,3/2)$ such that $\beta_n\to 1$, we have that $(u_{1,\beta_n}^+)$ behaves as in (vii) of Theorem \ref{a1} with $u_0\not=0$  and there exists a natural number $k_{\Lambda_1}\ge2$, such that if $k\ge k_{\Lambda_1}$, $(u_{k,\beta_n}^+)$ behaves as in (iii)  of Theorem \ref{a1} with $k-k_{\Lambda_1}<N<k$.
\end{corollary}

The behavior in (b) of Proposition \ref{cr:1} has already been observed in \cite{GN} for low energy nodal radial solutions. Our present work gives new information by Theorem \ref{a31} without imposing the low energy characterization. Moreover, notice that our necessary condition in (iii) of Theorem \ref{a1} suggests that such behavior (i.e., (iii)  with $k=N$) can happen only if $0<\la_*<\Lambda_1$ and $\la_*$ is not too closed to $\Lambda_1$ (by Lemma \ref{posi} below). 
 On the other hand, in view of (v) of Theorem \ref{a1}, a similar phenomenon seems possibly to occur also in the case of $\beta_n\uparrow1$. Interestingly, the corresponding necessary condition has the inequality opposite to that in the case of $\beta_n\downarrow 1$. It leads us to expect the following. 
\begin{conjecture}\label{con:1}
Let $(\la_*,\beta_*)\in(\Lambda^*,\Lambda_1)\times \{1\}$. Then there exist sequences  of values $(\la_n)\subset (0,\Lambda_1)$ and $(\beta_n)\subset (0,1)$  and solutions $(u_n)$  such that $(\la_n,\beta_n)\to(\la_*,\beta_*)$, $u_n\in S_{1,\la_n,\beta_n}$ for all $n\in \mathbb{N}$, and $(u_n)$ behaves as in (v) of Theorem \ref{a1}.
\end{conjecture}
\begin{remark}
Similar behavior would also occur in the general bounded domain case. 
\end{remark}
\begin{remark}
We also expect that there exist sequences of concentrating solutions which behave as in (iv) with $k=1$ and $0<\la_*<\Lambda_1$ and (vi) with $k=1$, $\la_*=\Lambda_1$,  and $\beta_*\in(0,1]$ respectively. The corresponding phenomena on the Brezis-Nirenberg problem are observed in \cite{IP} and \cite{IV}.
\end{remark}

We also remark on the final assertions in (d) of Proposition \ref{cr:1} and in Corollary \ref{cr:2}. Since $0<N<k$, these assertions prove the existence of a concentrating sequence of solutions which weakly converges to a sign-changing solution of \eqref{p}. We emphasize that this conclusion holds true when $\la_*\le \Lambda_1$ is sufficiently closed to $\Lambda_1$.  This phenomenon is new in view of the previous works, \cite{GN} and \cite{GMNP}, where the authors observed a concentrating sequence of solutions which weakly converges to a sign-definite solution of \eqref{p} with sufficiently small $\la>0$. We naturally expect that we can choose $k_{\la_*}=2$ for any $\la_*\in(\Lambda^*,\Lambda_1]$. 

We finally conjecture that  more counterparts of (iii)-(vi) would exist in the case $\la_*>\Lambda_1$ and $\beta_*\le1$. Our necessary condition on the weak limit will be useful to detect such sequences of solutions.  

\subsection{Proofs}
Let us prove the results above. We first put a basic lemma. Recall that $\max_{x\in\overline{B}}|u(x)|=|u(0)|$ for any  $u\in S_{0,\la,\beta}$ and all $(\la,\beta)\in(0,\Lambda_1)\times (0,2)$ by \cite{GNN}.
\begin{lemma}\label{posi} Fix any $\beta\in(0,2)$. Then, for any constants $M>\e>0$, there exist values $0<\bar{\la}\le\tilde{\la}<\Lambda_1$ such that if $\la\in(0,\bar{\la})$, then $|u(0)|\ge M$ for any $u\in S_{0,\la,\beta}$, and if $\la\in(\tilde{\la},\Lambda_1)$, then $|u(0)|<\e$  for all $u\in S_{0,\la,\beta}$. 
\end{lemma}
\begin{proof} 
If the former assertion fails, there exists a constant $M>0$ and sequences $(\la_n)\subset (0,\Lambda_1)$ and $(u_n)$ such that $\la_n\to 0$, $u_n\in S_{0,\la_n,\beta}$ and $0<u_n(0)\le M$ for all $n\in \mathbb{N}$. This is impossible since by Theorem \ref{a1}, we get that $\la_n\to 0$ yields  $u_n(0)\to \infty$. On the other hand, if the latter conclusion does not hold,  there exists a constant $\e_0>0$   and sequences $(\la_n)\subset (0,\Lambda_1)$ and $(u_n)$ such that $\la_n\to \Lambda_1$, $u_n\in S_{0,\la_n,\beta}$ and $u_n(0)\ge \e_0$ for any $n\in \mathbb{N}$. This is again impossible since by Theorem \ref{a1} and the fact that $S_{0,\Lambda_1,\beta}=\emptyset$, we get that $\la_n\to \Lambda_1$ implies  $u_n\to 0$ in $C^2(\overline{B})$.  This finishes the proof. 
\end{proof}

We show Corollary \ref{cr:0}.

\begin{proof}[Proof of Corollary \ref{cr:0}]
Assume as in the corollary. Then we claim that $(u_n)$ behaves as in (vii) of Theorem \ref{a1}. If not, since $\la_*\not=0$, (iii) would happen for $N=1$ and then, by Lemma \ref{crl1}, the weak limit $u_0$ of $(u_n)$ would verify $u_0\in S_{0,\la_*,1}$, $I_{\la_*,1}(u_0)=c_{0,\la_*,1}$ and $|u_0(0)|\ge \alpha/2$. This contradicts our choice of $\la_*$. This proves the claim. Then, since $\la_*<\Lambda_1$, the weak limit $u_0$ of $(u_n)$ is nontrivial. This ensures the former conclusion. Moreover notice that by Lemma \ref{crl1}, we get $I_{\la_*,1}(u_0)\le 2\pi+c_{0,\la_*,1}$. Then, since there exists a  sequence of solutions $(u_n)$ verifying the assumption of this corollary for any $\la_*\in(\Lambda^*,\Lambda_1)$ by the argument in the first paragraph of this section, the latter conclusion clearly follows. This finishes the proof.

\end{proof}

In order to prove (d) of Proposition \ref{cr:1}, we use Lemma \ref{lem1} below which ensures the nonexistence of low energy solutions with many nods. To show the lemma, we refer to the argument in \cite{K} and obtain the a priori lower estimate of the energy of elements in $S_{k,\la,1}$. We apply the next lemma. (See Corollary 5.2 on p346 in \cite{H} or Lemma 3 in \cite{K}.) In the following we let $a,b\in \mathbb{R}$ be constants such that $a<b$.
\begin{lemma}\label{h}  Let $q(t)$ be a continuous function on $[a,b]$. Let $v(t)\not=0$ be a solution of the equation:
\[
v''+q(t)v=0,\ \ t\in[a,b].
\]
Assume that $v(t)$ has exactly $k$ zeros in $(a,b]$. Then we have
\[
k<\frac12\left((b-a)\int_a^bq^+(t)dt\right)^{\frac12}+1
\]
where $q^+(t)\equiv \max\{q(t),0\}$.
\end{lemma}

Using this, we prove the next lemma.
\begin{lemma}\label{crl2} Assume $k\in \mathbb{N}$, $k\ge2$ and $\la\ge\Lambda_*$. Then there exists a number $L=L(\la)>0$ which depends on $\la$ and is independent of $k$ such that 
\[
\int_B|\nabla u|^2dx\ge L(k-2)^2-1
\]
 for any $u\in S_{k,\la,1}$.
\end{lemma}
\begin{proof} Fix $\la\ge\Lambda_*$ and suppose $k\ge2$. For any $u\in S_{k,\la,1}$ with $u(0)>0$, we write $u=u(|x|)$ and define a number $r_{u,2}\in(0,1)$ so that there exists a value $r_{u,1}\in(0,r_{u,2})$ such that $u(r_{u,i})=0$ for $i=1,2$ and $u(r)>0$ for all $r\in[0,r_{u,1})$ and $u(r)<0$ for all $r\in(r_{u,1},r_{u,2})$. Then we claim that 
\[
r_\la:=\inf_{k\ge2}\inf_{u\in S_{k,\la,1},u(0)>0}r_{u,2}>0.
\]
Assume that $r_\la=0$ on the contrary. Then for any $\e\in(0,1)$, we can choose $k_\e\ge2$ and a nodal radial solution $u_\e\in S_{k_\e,\la,1}$ with $u_\e(0)>0$ such that $r_{u_{\e},2}\in(0,\e)$. Here we fix $\e\in(0,1)$ so small that $\e^2\la \in(0,\Lambda_*)$ and  define $v_\e(x):=u_\e(r_{u_{\e},2} |x|)$ ($x\in B$). Then we get $v_\e\in S_{1,r_{u_{\e},2}^2\la,1}$. This is impossible by the definition of $\Lambda_*$. Hence we ensure the claim. Choose any $u\in S_{k,\la,1}$ with $u(0)>0$ and regard $u=u(|x|)$. Then we perform the Liouville transformation $r=1/(1-\log{t})$, $v(r)=ru(t)$ for $t\in(r_{u,2},1]$. It follows that 
\[
v''+q(r)v=0,\ \ \ \ r\in\left[\frac{1}{1-\log{r_{u,2}}},1\right]
\] 
where 
\[
q(r)=t^2\left(\log{\frac{e}{t}}\right)^4\frac{\la f(u(t))}{u(t)}\ \ (t=e^{1-1/r}),
\]
and $f(u)=ue^{u^2+\alpha|u|}$. Notice that $v$ has exactly $k-1$ zeros in $\left(1/(1-\log{r_{u,2}}),1\right]$. Then noting $f(t)/t\le e^{1+\alpha}+t^2e^{t^2+\alpha|t|}$ for any $t\in \mathbb{R}$, we apply Lemma \ref{h} and get
\[
\begin{split}
&k-1\\
&<\frac12\left(\left(1-\frac{1}{1-\log{r_{u,2}}}\right)\int_{\frac{1}{1-\log{r_{u,2}}}}^1 q(r)dr\right)^{\frac12}+1\\
&=\frac12\left(\left(1-\frac{1}{1-\log{r_{u,2}}}\right)\int_{r_{u,2}}^1 \frac{\left(\log{\frac{e}{t}}\right)^4}{(1-\log{t})^2}\frac{\la f(u(t))}{u(t)}tdt\right)^{\frac12}+1\\
&\le \frac12\left(1-\frac{1}{1-\log{r_\la}}\right)^{\frac12}\left(\log{\frac{e}{r_\la}}\right)^2\left(\int_{r_{u,2}}^1 \la (e^{1+\alpha}+u(t)f(u(t)))tdt\right)^\frac12+1\\
&\le \frac1{\sqrt{L}}\left(1+\int_B|\nabla u|^2dx\right)^{\frac12}+1
\end{split}
\] 
for some constant $L=L(\la)>0$. This completes the proof.
\end{proof}

Then we can prove the next key lemma. 
\begin{lemma}\label{lem1}
For any $\la\in[\Lambda_*,\Lambda_1]$, there exists a natural number $k_\la$ such that for any $k\ge k_\la$, there exists no radial nodal solution $u\in S_{k,\la,1}$ satisfying $I_{\la,1}(u)\le 2\pi k+c_{0,\la,1}$. 
\end{lemma}
\begin{proof}
Assume $\la\in[\Lambda_*,\Lambda_1]$, $k\in \mathbb{N}$ and $u\in S_{k,\la,1}$ satisfies  $I_{\la,1}(u)\le 2\pi k+c_{0,\la,1}$. We claim that there exists a constant $M=M(\la)>0$ which depends only on $\la$ and independent of $k$ and $u$ such that  
\begin{equation}\label{cr1lem2}
\int_B|\nabla u|^2dx\le 4\pi (k+1)+M.
\end{equation}
To see this, choose $u$ as above and put $f(t)=te^{t^2+|t|}$ and $F(t)=\int_0^t f(s)ds$. Then an elementary calculation shows that for any $\e>0$, there exists a constant $t_\e>0$ such that $\e f(t)t-F(t)\ge0$ for any $|t|\ge t_\e$. (See formula (2.3) in \cite{FMR}). We choose $\e\in(0,1/2)$ so small that $(2\pi k+c_{0,\la,1})/(1/2-\e)\le 4\pi(k+1)$. This choice is possible since $c_{0,\la,1}<2\pi$ (\cite{A}). Then there exists a constant $C_\e>0$ which is independent of $k$ and $u$ such that 
\[
\begin{split}
2\pi k+c_{0,\la,1}&\ge I_{\la,\beta}(u)-\e\langle I_{\la,\beta}'(u),u\rangle\\
&\ge \left(\frac{1}{2}-\e\right)\int_B|\nabla u|^2dx+\la \int_{B\cap\{u\le t_\e\}}(\e f(u)u-F(u))dx\\
&\ge \left(\frac{1}{2}-\e\right)\int_B|\nabla u|^2dx-\la C_\e.
\end{split}
\]
This proves \eqref{cr1lem2} by putting $M=\la C_\e/(1/2-\e)$. Then take a constant  $L=L(\la)>0$ from the previous lemma. Choose a natural number $\tilde{k}_\la\ge2$ so that $L(k-2)^2-1> 4\pi(k+1)+M$ for all $k\ge \tilde{k}_\la$. As a consequence, it follows from the previous lemma and \eqref{cr1lem2} that $k\ge \tilde{k}_\la$ and $u\in S_{k,\la,1}$ yield $I_{\la,1}(u)> 2\pi k+c_{0,\la,1}$. This finishes the proof.
\end{proof}

Let us complete the proof of Proposition \ref{cr:1}.

\begin{proof}[Proof of Proposition \ref{cr:1}] Assume as in the proposition. First suppose $\la_*=\Lambda_1$. Then if $k=0$, since $\la_*\not=0$ and $S_{0,\Lambda_1,\beta_*}=\emptyset$, we clearly have that (vii) of Theorem \ref{a1} occurs  with $u_0=0$. In addition, assume $\beta_n\ge1$ for all $n\in \mathbb{N}$ and suppose $k\ge1$ and $\beta_*\in(1,3/2)$ or $k=1$ and $\beta_*=1$. Then in the former case, noting $\beta_*>1$ and on the other hand, in the latter case, using $S_{0,\Lambda_1,1}=\emptyset$, we see that $(u_n)$ does not blow up and thus,  $(u_n)$ behaves as in (vii) with $u_0\not=0$ since $0<\la_*=\Lambda_1<\Lambda_{k+1}$ in both cases. This completes (a). Next suppose $k\ge1$, $\la_*\in(0,\Lambda_1)$, $\beta_n>1$ for all $n\in \mathbb{N}$ and $\beta_*=1$. Then, since $\la_*\not=0$, $(u_n)$ behaves as in (iii) or (vii) of Theorem \ref{a1}. Hence, if we additionally suppose $S_{\kappa,\la_*,1}=\emptyset$ for all $0<\kappa\le k$, the only possibility is (iii) with $N=k$ to happen. This completes  (b).  Finally, suppose $I_{\la_n,\beta_n}(u_n)=c_{k,\la_n,\beta_n}$ for all $n\in \mathbb{N}$. Set $\la_*\in[\Lambda_*,\Lambda_1)$ and $k\ge k_{\la_*}$ where $k_\la$ is chosen from Lemma \ref{lem1}. It follows that $(u_n)$ behaves as in (iii) with a natural number $k-k_{\la_*}<N\le k$. Otherwise, there would be integers $0\le N\le k-k_{\la_*}$,  $\kappa=k-N$ and an element $u_0\in S_{\kappa,\la_*,1}$ such that
\[
I_{\la_n,\beta_n}(u_n)\to 2\pi N+I_{\la_*,1}(u_0)\le 2\pi k+c_{0,\la_*,1}
\]
by Lemma \ref{crl1}. This implies $I_{\la_*,1}(u_0)\le 2\pi \kappa+c_{0,\la_*,1}$. Since $\kappa\ge k_{\la_*}$, this is impossible by Lemma \ref{lem1}. This proves the former conclusion of (c). Then noting the first conclusion of Corollary \ref{cr:0}, we get the latter one. Lastly let us show (d). Set $\la_*$  as in the assumption. Suppose $k\ge2$ and $(u_n)$ behaves as in (iii). If $N=k$ on the contrary, then by Lemma \ref{crl1}, the weak limit $u_0$ of $(u_n)$ satisfies $u_0\in S_{0,\la_*,1}$, $I_{\la_*,1}(u_0)=c_{0,\la_*,1}$ and $|u_0(0)|\ge \alpha/2$. This is a contradiction. Hence we get $N\not=k$. This proves the first assertion. Especially, combining (c) with the previous conclusion, we obviously get the final conclusion. This completes the proof.

\end{proof}

Finally, we complete the proof of Corollary \ref{cr:2}.

\begin{proof}[Proof of Corollary \ref{cr:2}] Fix $\beta\in[1,3/2)$. Note that  the first assertion in (a) of the previous proposition implies $c_{0,\la_n,\beta}\to 0$ if $\la_n\uparrow\Lambda_1$. Then, we set $k=1$ if $\beta=1$ and $k\ge1$ if $\beta>1$ and choose sequences $(\la_n)\subset (0,\Lambda_1)$ and $(u_n)$ so that $\la_n\uparrow \Lambda_1$, $u_n\in S_{k,\la_n,\beta}$ and $I_{\la_n,\beta}(u_n)\le 2\pi k+c_{0,\la_n,\beta}$ for all $n\in \mathbb{N}$. This choice is valid by the argument in the first paragraph of this section and Corollary \ref{cr:0}. Then from the latter assertion in (a) of the previous proposition, we have a solution $u_0\in S_{k,\Lambda_1,\beta}$ such that $I_{\la_n,\beta}(u_n)\to I_{\Lambda_1,\beta}(u_0)\le 2\pi k$. This proves the former conclusion.  The latter assertion is clear by noting $S_{0,\Lambda_1,1}=\emptyset$, choosing $k_{\Lambda_1}$ from Lemma \ref{lem1} and arguing as in the proof of Corollary \ref{cr:0} and (d) of Proposition \ref{cr:1}. This finishes the proof.  

\end{proof}

\appendix
\section{Proof of Lemma \ref{ap1}}\label{ap}
In this appendix we show the proof of Lemma \ref{ap1}.

\begin{proof}[Proof of Lemma \ref{ap1}]
First we claim that $r_{i,n}/\ga_{i,n}\to\infty$. In fact, putting $v_{n}(r)=u_{i,n}(r_{i,n}r)$ for $r\in[r_{i-1,n}/r_{i,n},1]$ and noting  Lemma \ref{newl1}, we have
a constant $K_0>0$ such that
\begin{equation}
K_0\le \la_n r_{i,n}^2 \int_{r_{i-1,n}/r_{i,n}}^{1}v_nf_n(v_n)rdr\le \la_n r_{i,n}^2 e^{\mu_{i,n}^2+\alpha \mu_{i,n}^{\beta_n}}\int_{r_{i-1,n}/r_{i,n}}^{1}v_{n}^2rdr,\label{d0}
\end{equation}
up to a subsequence. Then, since $\int_{r_{i-1,n}/r_{i,n}}^{1}v_{n}^2rdr$ is uniformly bounded from above by the Poincare inequality and our assumption \eqref{bdd}, we get a constant $C>0$ such that \[
\la_n  r_{i,n}^2e^{\mu_{i,n}^2+\alpha \mu_{i,n}^{\beta_n}}\ge C.
\]
Multiplying this inequality by $\mu_{i,n}^2$ and noting the definition of $\ga_{i,n}$, we see that 
\[
\frac12\left(\frac{r_{i,n}}{\ga_{i,n}}\right)^2\ge C\mu_{i,n}^2.
\]
Consequently, \eqref{infty} ensures the claim. In particular, we get $\ga_{i,n}\to0$. Next,  we claim $\rho_{i,n}/r_{i,n}\to0$. This is trivial for $i=1$. 
 Hence we assume $i\ge 2$. Define $v_n$ as above.
 It follows from our assumption \eqref{bdd} and Lemma  \ref{RL} that there exist constants $c,C>0$ such that 
\[
C\ge \int_{r_{i-1,n}/r_{i,n}}^{1} v_n'(r)^2rdr\ge\frac{\mu_{i,n}^2}{2\pi c^2}\left(\frac{\rho_{i,n}}{r_{i,n}}\right).
\]
Then \eqref{infty} shows the claim. In particular, we get $\rho_{i,n}\to 0$ and  $(r_{i,n}-\rho_{i,n})/\ga_{i,n}\to \infty$ by the first claim. Next, by the definition of $z_{i,n}$ and \eqref{d1}, we get that
\[
0\le -z_{i,n}''-\frac1{r+\frac{\rho_{i,n}}{\ga_{i,n}}}z_{i,n}'\le1 \text{ on }\left[\frac{r_{i-1,n}-\rho_{i,n}}{\ga_{i,n}},\frac{r_{i,n}-\rho_{i,n}}{\ga_{i,n}}\right].
\]
Then, for any $r\in \left[\frac{r_{i-1,n}-\rho_{i,n}}{\ga_{i,n}},\frac{r_{i,n}-\rho_{i,n}}{\ga_{i,n}}\right]$, multiplying the equation by $r+\frac{\rho_{i,n}}{\ga_{i,n}}$ and integrating over $(0,r)$ if $r\ge0$ and over $(r,0)$ if $r<0$ give
\begin{equation}
\begin{cases}
0\le -z_{i,n}'(r)\le \ \frac{\frac{r^2}2+\frac{\rho_{i,n}}{\ga_{i,n}}r}{r+\frac{\rho_{i,n}}{\ga_{i,n}}} \text{ if }r\ge0, \text{ and }\\
0\le z_{i,n}'(r)\le -\frac{\frac{r^2}2+\frac{\rho_{i,n}}{\ga_{i,n}}r}{r+\frac{\rho_{i,n}}{\ga_{i,n}}} \text{ if }r<0.
\end{cases}
\label{d2}
\end{equation}
Integrating this again, we get
\begin{equation}
\begin{split}
0\le-z_{i,n}(r)&\le \int_0^r\frac{\frac{r^2}2+\frac{\rho_{i,n}}{\ga_{i,n}}r}{r+\frac{\rho_{i,n}}{\ga_{i,n}}}dr\\
&=
\begin{cases}
\frac{r^2}{4}\text{ if }i=1,\\
\frac{r^2}{4}+\left(\frac{\rho_{i,n}}{\ga_{i,n}}\right)\frac r2-\frac12\left(\frac{\rho_{i,n}}{\ga_{i,n}}\right)^2\log{\frac{\frac{\rho_{i,n}}{\ga_{i,n}}+r}{\frac{\rho_{i,n}}{\ga_{i,n}}}},\text{ if }i\ge2,
\end{cases}\label{d3}
\end{split}
\end{equation}
for any $r\in \left[\frac{r_{i-1,n}-\rho_{i,n}}{\ga_{i,n}},\frac{r_{i,n}-\rho_{i,n}}{\ga_{i,n}}\right]$. Then from \eqref{d2} and \eqref{d3}, we get that $z_{i,n}$ is uniformly bounded in $C^1_{\text{loc}}((-l,\infty))$ ($C^1_{\text{loc}}([0,\infty))$ if $l=0$). Furthermore, since \eqref{d2} implies $|z_{i,n}'(r)/(r+\rho_{i,n}/\ga_{i,n})|$ is locally uniformly bounded in $(-l,\infty)$ ($[0,\infty)$ if $l=0$), using the equation in \eqref{d1}, we get that $z_{i,n}$ is uniformly bounded in $C^2_{\text{loc}}((-l,\infty))$ ($C^2_{\text{loc}}([0,\infty))$ if $l=0$). Then it follows from the Ascoli-Arzel\`{a} theorem and the equation in \eqref{d1} that there exists a function $z$ such that $z_{i,n}\to z$  in $C^2_{\text{loc}}((-l,\infty))$  ($C^2_{\text{loc}}((0,\infty))\cap C^1_{\text{loc}}([0,\infty))$ if $l=0$). 

Now our final aim is to show $l=m<\infty$. This is clear if $i=1$. Hence after this we assume $i\ge2$. We first claim $l<\infty$. We suppose $l=\infty$ on the contrary. Then $m=\infty$. It follows from \eqref{d1} that $z$ satisfies 
\[
\begin{cases}
-z''=e^z\hbox{ in }\mathbb{R},\\
z(0)=0=z'(0).
\end{cases}
\]
This implies $z(r)=\log\frac{4e^{\sqrt2r}}{\left(1+e^{\sqrt2r}\right)^2}$ $(r\in \mathbb{R})$. Then by \eqref{bdd}, there exits a constant $C>0$ such that
\begin{align}
C&\ge \la_n\int_{\rho_{i,n}}^{r_{i,n}}u_nf_n(u_n)rdr\notag\\
&\ge\frac12 \int_{0}^{{\frac{r_{i,n}-\rho_{i,n}}{\ga_{i,n}}}}\left(\frac{z_{i,n}}{2\mu_{i,n}^2}+1\right)e^{z_{i,n}+\frac{z_{i,n}^2}{4\mu_{i,n}^2}+\alpha \mu_{i,n}^{\beta_n}\left\{\left(\frac{z_{i,n}}{2\mu_{i,n}^2}+1\right)^{\beta}-1\right\}}\left(r+\frac{\rho_{i,n}}{\ga_{i,n}}\right)dr\notag\\
&\ge\frac12 \frac{\rho_{i,n}}{\ga_{i,n}}\int_{0}^{{\frac{r_{i,n}-\rho_{i,n}}{\ga_{i,n}}}}\left(\frac{z_{i,n}}{2\mu_{i,n}^2}+1\right)e^{z_{i,n}+\frac{z_{i,n}^2}{4\mu_{i,n}^2}+\alpha \mu_{i,n}^{\beta_n}\left\{\left(\frac{z_{i,n}}{2\mu_{i,n}^2}+1\right)^{\beta}-1\right\}}dr.\label{linfty}
\end{align}
Here the Fatou lemma implies
\[
\liminf_{n\to \infty}\int_{0}^{{\frac{r_{i,n}-\rho_{i,n}}{\ga_{i,n}}}}\left(\frac{z_{i,n}}{2\mu_{i,n}^2}+1\right)e^{z_{i,n}+\frac{z_{i,n}^2}{4\mu_{i,n}^2}+\alpha \mu_{i,n}^{\beta_n}\left\{\left(\frac{z_{i,n}}{2\mu_{i,n}^2}+1\right)^{\beta}-1\right\}}dr\ge \int_0^\infty e^zdr.
\]
Since the right hand side of the inequality above is positive value and $m= \infty$, we get that the right hand side of \eqref{linfty} diverges to infinity which is a contradiction. This proves the claim. Finally we show $l=m$. Let us  assume $l<m$ on the contrary. Then we claim that there exists a constant $C>0$ such that 
\[
|z_{i,n}'(r)|\le C \text{ for all $n\in \mathbb{N}$ and $r\in \left[\frac{r_{i-1,n}-\rho_{i,n}}{\ga_{i,n}},0\right]$.}
\]   In fact, if $m=\infty$, the claim follows easily by the latter formula in \eqref{d2}. If $m<\infty$, using \eqref{d2} again we get a constant $C>0$ such that
\[
0\le z_{i,n}'(r)\le \frac{C}{m-l+o(1)}
\]
uniformly for all $r\in\left[\frac{r_{i-1,n}-\rho_{i,n}}{\ga_{i,n}},0\right]$. This proves the claim. On the other hand, by the mean value theorem, we have a sequence $(\xi_n)\subset \left(\frac{r_{i-1,n}-\rho_{i,n}}{\ga_{i,n}},0\right)$ such that
\[
z_{i,n}'(\xi_n)=\frac{-z_{i,n}\left(\frac{r_{i-1,n}-\rho_{i,n}}{\ga_{i,n}}\right)}{\frac{\rho_{i,n}-r_{i-1,n}}{\ga_{i,n}}}=\frac{2\mu_{i,n}^2}{l+o(1)}\to \infty
\]
since $l\in[0,\infty)$. This is a contradiction. We finish the proof.

\end{proof}

\subsection*{acknowledgements}
The author sincerely thanks Prof. Massimo Grossi at Sapienza Universit\`{a} di Roma since some important questions and ideas in the present paper are inspired by the extensive discussion on the previous works with him. This work is supported by JSPS KAKENHI Grant-in-Aid for Young Scientists (B), Grant Number 17K14214. This work is also partly supported by Osaka City University Advanced Mathematical Institute (MEXT Joint Usage/Research Center on Mathematics and Theoretical Physics).

\end{document}